\documentclass[12pt]{amsart}

\usepackage[T1]{fontenc}
\usepackage[utf8]{inputenc}
\usepackage{amsmath,amssymb,amsthm,mathtools}
\usepackage{xcolor,enumerate}
\usepackage[margin=30mm]{geometry}

\allowdisplaybreaks
\RequirePackage[authoryear]{natbib}
\RequirePackage[colorlinks,citecolor=blue,urlcolor=blue]{hyperref}
\RequirePackage{graphicx}

\usepackage{nicematrix, tikz}
\colorlet{darkpink}{magenta!70!black}
\usetikzlibrary{positioning,arrows.meta,shapes.geometric}

\definecolor{graphblue}{RGB}{41, 54, 130}
\definecolor{graphred}{RGB}{186, 45, 48}
\definecolor{graphpurple}{RGB}{147, 51, 234}
\definecolor{graphgreen}{RGB}{34, 197, 94}
\definecolor{graphbrown}{RGB}{146, 64, 14}
\definecolor{graphblue}{RGB}{37, 99, 235}
\definecolor{graphyellow}{RGB}{234, 179, 8}
\definecolor{graphorange}{RGB}{217, 119, 6}

\definecolor{vcolorA}{RGB}{124, 58, 237}   
\definecolor{vcolorB}{RGB}{13, 148, 136}   
\definecolor{vcolorC}{RGB}{217, 119, 6}    
\definecolor{vcolorD}{RGB}{219, 39, 119}   
\definecolor{vcolorE}{RGB}{29, 78, 216}    

\definecolor{ecolor12}{RGB}{159, 18, 57}  
\definecolor{ecolor13}{RGB}{101, 163, 13}  
\definecolor{ecolor34}{RGB}{22, 163, 74}   
\definecolor{ecolor35}{RGB}{180, 83, 9}   
\definecolor{ecolor45}{RGB}{8, 145, 178}   
\definecolor{ecolor5X}{RGB}{14, 165, 233}  
\definecolor{ecolorM1}{RGB}{109, 40, 217}   
\definecolor{ecolorM2}{RGB}{220, 38, 38}   
\definecolor{ecolorM3}{RGB}{5, 150, 105}    

\theoremstyle{plain}

\newtheorem{theorem}{Theorem}[section]
\newtheorem{propo}[theorem]{Proposition}
\newtheorem{lemma}[theorem]{Lemma}
\theoremstyle{definition}
\newtheorem{definition}[theorem]{Definition}
\newtheorem{example}[theorem]{Example}

\theoremstyle{remark}
\newtheorem{remark}[theorem]{Remark}
\newtheorem{corol}[theorem]{Corollary}

\renewcommand{\P}{\mathbb{P}}

\newcommand{\R}{\mathbb{R}}
\newcommand{\N}{\mathbb{N}}

\newcommand{\1}{\mathbf{1}}

\newcommand{\dd}{\mathrm{d}}

\newcommand{\tr}{\operatorname{tr}}

\newcommand{\coloredZ}{\mathcal{Z}_{G}^{\mathcal{C}}}

\def\nb{{\mathfrak{nb}}}

\def\ch{{\mathfrak{ch}}}

\newcommand{\nc}{\newcommand}
\nc{\ep}{\varepsilon}
\nc{\Zsp}{\mathcal{Z}}
\nc{\Lsp}{\mathcal{L}}
\nc{\Ssp}{\mathcal{S}}
\nc{\Msp}{\mathcal{M}}
\nc{\Pcone}{\mathcal{P}}
\nc{\BC}{BC}
\nc{\cBC}{DCBC}
\nc{\IG}{$I$}
\nc{\cIG}{$cI$}
\nc{\hh}{\mathfrak{h}}

\newcommand{\mytransp}[1]{ #1^{\top}}

\newcommand{\bdiag}{\mathrm{BlockDiag}}
\newcommand{\tri}{\mathrm{BlockTri}}
\newcommand{\RefExt}{\mathrm{RefExt}}

\usepackage{listings} 
\lstdefinestyle{Rstyle}{
  language=R,
  alsoletter={._},
  otherkeywords={!,!=,~,$,*,\&,\%/\%,\%*\%,\%\%,<-,<<-,/},
  deletekeywords={D,formula},
  morekeywords={gips,log_posteriori_of_gips},
  basicstyle=\ttfamily\small,
  keywordstyle=\color{black}\bfseries,
  commentstyle=\color{gray},
  stringstyle=\color{red!60!black},
  numbers=none,
  showstringspaces=false,
  columns=fullflexible,
  keepspaces=true,
  frame=single,
  rulecolor=\color{black!20},
  breaklines=true
}
\usepackage[ruled,vlined]{algorithm2e}

\title[A new class of colored Gaussian graphical models]
{A new class of colored Gaussian graphical models
with explicit normalizing constants}

\author[A. Chojecki]{Adam Chojecki}
\address{Faculty of Mathematics and Information Sciences,\newline
Warsaw University of Technology}
\email{adam.chojecki@pw.edu.pl}
\urladdr{https://orcid.org/0009-0008-2902-4096}

\author[P. Graczyk]{Piotr Graczyk}
\address{Univ Angers, CNRS, LAREMA, SFR MATHSTIC}
\email{piotr.graczyk@univ-angers.fr}

\author[H. Ishi]{Hideyuki Ishi}
\address{Department of Mathematics,
Osaka Metropolitan University}
\email{hideyuki-ishi@omu.ac.jp}
\urladdr{https://orcid.org/0000-0002-6825-1004}

\author[B. Ko{\l}odziejek]{Bartosz Ko{\l}odziejek}
\address{Faculty of Mathematics and Information Sciences,\newline
Warsaw University of Technology}
\email{bartosz.kolodziejek@pw.edu.pl}
\urladdr{https://orcid.org/0000-0002-5220-9012}

\thanks{For the purpose of Open Access, the authors have applied
a CC-BY public copyright licence to any Author Accepted Manuscript
(AAM) version arising from this submission.}



\date{}

\begin{document}

\begin{abstract}
We study Bayesian model selection in colored Gaussian graphical models (CGGMs), which combine sparsity of conditional independencies with symmetry constraints encoded by vertex- and edge-colored graphs. A key computational bottleneck in Bayesian inference for CGGMs is the evaluation of the Diaconis-Ylvisaker normalizing constants, given by gamma-type integrals over cones of precision matrices with prescribed zeros and equality constraints. While simple closed-form expressions are known for standard Gaussian graphical models only in special cases (e.g. decomposable graphs) and for a limited class of RCOP models, no general tractable framework has been available for broader families of CGGMs.

We introduce a new subclass of RCON models for which these normalizing constants admit closed-form expressions. On the algebraic side, we identify conditions on the space of colored precision matrices that guarantee tractability of the associated integrals, leading to the notions of Block-Cholesky spaces (BC-spaces) and Diagonally Commutative Block-Cholesky spaces (\cBC-spaces). On the combinatorial side, we characterize the colored graphs inducing such spaces via a color perfect elimination ordering and a $2$-path regularity condition, and define the resulting Color Elimination-Regular (CER) graphs and their symmetric variants. Our framework extends the theory of uncolored decomposable graphs to the colored setting, and the CER class contains all RCOP models associated with decomposable graphs.
For models with a single vertex color, our framework reveals a close connection between \cBC-spaces and Bose-Mesner algebras.

For models defined on \BC-spaces, we derive finite product formulas for the normalizing constants in terms of gamma functions, matrix determinants, and model-specific structure constants. For RCON models whose matrix spaces are \cBC-spaces, we provide efficient methods for computing all structure constants and determinant factors required to evaluate these formulas. Our results substantially broaden the range of CGGMs amenable to Bayesian structure learning in high-dimensional applications.
\end{abstract}

\maketitle
\tableofcontents

\section{Introduction}

Graphical Gaussian models (GGMs), also known as covariance-selection models \citep{D72,L96}, encode conditional independencies among variables (indexed by an undirected graph $G$) through zeros in a precision matrix $K=\Sigma^{-1}$. These models have become indispensable tools for high-dimensional applications across fields such as neuroimaging, genetics, and finance. However, when the number of variables $p$ significantly exceeds the sample size $n$ and the graph is not sparse enough, estimating the non-zero entries of the precision matrix becomes challenging.

To address this, \citet{HL08} introduced colored Gaussian graphical models (CGGMs), which fuse two complementary forms of parsimony: the sparsity induced by conditional-independence zeros and symmetry constraints imposed through equality restrictions on the model parameters. Within their framework, three distinct model families, RCON (constraints on precision entries), RCOR (constraints on partial correlations), and RCOP (constraints induced by permutation-group invariance), are represented by vertex- and edge-colored graphs that explicitly encode which entries must coincide. This approach builds upon a rich tradition of exploiting invariance in multivariate Gaussian analysis \citep{Wilks,OP69}, and extends earlier ``lattice'' symmetry models \citep{AM98,Ma00} by unifying the independence structure and symmetry conditions into a single combinatorial object, a colored graph. The CGGMs generalize GGMs: if all vertices and edges of a graph have distinct colors, then the CGGM coincides with a GGM.

Subsequent research has expanded upon this foundation across several key directions. First, structural theory and lattice-based search methods have been explored: \citet{Ge11} demonstrated that RCON, RCOR and RCOP model classes possess complete-lattice properties and adapted the Edwards-Havránek search procedure to colored models.
Second, advances have been made in likelihood-based and Bayesian inference: \citet{HL07} provided iterative MLE algorithms; \citet{GM15} and \citet{MLG15} introduced composite likelihood and Bayesian methodologies utilizing the Diaconis-Ylvisaker conjugate prior (the colored $G$-Wishart), respectively. \citet{LiGM20} further developed double reversible jump MCMC algorithms coupled with linear regression to explore the space of RCON models. More recently, \citet{CAPP24} proposed a Bayesian framework to learn block-structured graphs, while \citet{Biom25} introduced a birth-death MCMC approach for structure learning in graphical models with explicit symmetry constraints. Third, penalized and high-dimensional estimation methodologies were developed: \citet{QXNX21} proposed an $\ell_1$-penalized composite-likelihood approach that jointly performs graph selection and symmetric shrinkage, and \citet{RRL21} applied fused graphical lasso penalties to simultaneously learn network structure and symmetries, motivated by brain connectivity studies. Finally, \citet{GIKM22} leveraged explicit gamma-integral formulas derived from group-representation theory to enable model selection within RCOP complete graphs. Their model selection procedure was recently implemented in the \textsf{gips} package \citep{gips}; Appendix~\ref{subsec:random_method_example} compares that implementation with our method in a worked example.

Related symmetry constraints have also been studied beyond Gaussian models:
\citet{RCG26} introduced colored graphical models for multivariate extremes.

Together, these contributions underscore that symmetry constraints not only reduce the required sample size for reliable estimation but also enhance the interpretability of model parameters, capturing meaningful structures such as gene exchangeability.

\subsection{Bayesian model selection in CGGMs}\label{sec:bayesian-model-selection}
A colored graph is identified with a pair $(G,\mathcal{C})$, where $G$ is an undirected graph and $\mathcal{C}$ is the coloring of $G$ \citep{HL08}.

We assume that a colored graph-valued random pair $(G,\mathcal{C})$ is distributed on the set $\mathcal{G}$ of colored graphs with probabilities $(p_{g,c})_{(g,c)\in\mathcal{G}}$, i.e.  $p_{g,c}=\P(G=g,\mathcal{C}=c)$. 

Let $\mathcal{P}_g^c$ be the colored cone, i.e., the set of positive definite (precision) matrices respecting conditional independencies and symmetry constraints given by $(g,c)$, i.e., $K\in \mathcal{P}_G^{\mathcal{C}}$ with $G=(V,E)$ if and only if
\begin{itemize}
    \item $K\in\mathrm{Sym}^+(p)$,
    \item $i\ne j,\ \{i,j\}\notin E \implies K_{ij}=0$,
    \item if vertices $i,j\in V$ have the same color, then $K_{ii}=K_{jj}$,
    \item if edges $\{i,j\}, \{i',j'\}\in E$ have the same color, then $K_{ij}=K_{i'j'}$,
\end{itemize}
see Section \ref{sec:matrix_subspaces}.
We assume that $K\mid (G,\mathcal{C})=(g,c)$ follows the Diaconis-Ylvisaker prior distribution for $K$ on $\mathcal{P}_g^c$ with hyperparameters $(\delta,D)\in(0,\infty)\times\mathrm{Sym}^+(p)$, namely, the distribution with density (with respect to the Lebesgue measure normalized with respect to the trace inner product, see Remark \ref{rem:Leb}) 
\[
\frac{1}{I_{g,c}(\delta,D)}\det(K)^{(\delta-2)/2}e^{-\frac12\tr(D K)} \1_{\mathcal{P}_g^c}(K),
\]
where
\[
I_{g,c}(\delta,D)
=\int_{\mathcal{P}_g^c}
\det(K)^{(\delta-2)/2}e^{-\frac12\tr(DK)}\,\dd K
\]
is the normalizing constant, which we assume to be finite.
Assuming that the Gaussian vectors $Z_1,\ldots,Z_n$ given $\{K=k, (G,\mathcal{C})=(g,c)\}$ are i.i.d. $\mathrm{N}_{p}(0,k^{-1})$ random vectors, standard calculations show that the posterior probability is proportional to
\begin{equation}\label{eq:posterior-model-probability}
\P((G,\mathcal{C})=(g,c)\mid Z_1,\ldots,Z_n) \propto \frac{I_{g,c}(\delta+n,D+U)}{I_{g,c}(\delta,D)} p_{g,c}
\end{equation}
with $U=\sum_{i=1}^n Z_i Z_i^\top$.  
These constants take the form of gamma-like integrals over non-standard cones of symmetric matrices characterized by zero patterns and equality constraints (clustered entries).

Under the Bayesian paradigm, the aim is to identify the maximum a posteriori estimator - the colored graph $(g,c)$ that maximizes the posterior probability given the observed data. As shown by \citet{MohWit15,Moh24}, Bayesian approaches often outperform a single point estimate obtained via regularized optimization (e.g., the graphical lasso \citep{glasso}).
However, computing normalizing constants is widely recognized as a significant bottleneck in Bayesian analyses involving graphical (and hence colored-graphical) Gaussian models. In cases where explicit formulas are unavailable, one can resort to approximation techniques, Monte Carlo simulations, or other numerical strategies, see the next section or a recent review paper \citep{Moh24}. Even in modest dimensions, an exhaustive enumeration of the model space is infeasible \citep{Ge11}, therefore Markov chain Monte Carlo methods are naturally used to perform posterior inference. We propose to tackle this problem by additionally limiting the search to a subclass of colored graphical models.

The primary objective of this paper is to introduce and investigate specific subclasses of colored graphs $\mathcal{G}$, along with a dedicated methodology for efficiently computing their normalizing constants.

\subsection{Normalizing constants for GGMs and CGGMs}
Closed-form expressions for the $I_G(\delta,D)$ (GGM version) have long been known only for the two ``easy'' cases: complete graphs \citep{Muir} and their decomposable (chordal) relatives \citep{Ro02}, where the density factorizes over cliques. For non-decomposable graphs, practitioners resorted to Monte Carlo integration \citep{A-KM05} or to analytic surrogates such as Laplace approximation \citep{LD11}, or they avoided the normalizing constant calculation via an exchange algorithm embedded in trans-dimensional MCMC \citep{Len13,MohWit15}. A conceptual breakthrough came from \citet{ULR18}, who showed that an explicit, though algebraically intricate, formula exists for arbitrary $G$; their result establishes existence but is still too heavy to power large-scale structure learning. Building on the representation established by \citet{A-KM05}, \citet{MML23} derived an explicit closed-form approximation of the ratio of the normalizing constants, cutting computational cost without a significant loss of accuracy. Most recently, \citet{WMK25} pushed this line of work further by importing tools from random-matrix theory and Fourier analysis: they expressed $I_G(\delta,D)$ as an integral involving $I_{G^\ast}(\delta,\cdot)$, where $G^\ast$ is any chordal completion of $G$ (so that a closed-form expression for $I_{G^\ast}(\delta,D)$ is known). Exploiting this relation, they reduced the original high-dimensional integral $I_G(\delta,D)$ to a one-dimensional integral for graphs with minimum fill-in $1$, and, when $D$ is the identity matrix, derived explicit closed-form expressions (involving gamma and generalized hypergeometric functions) for many prime graphs.
They also developed a Monte Carlo method for general graphs
and implemented their methods in the \textsf{GWnorm} package.

The literature addressing normalizing constants for CGGMs is relatively limited. Notably, a few specific, illustrative examples have been examined \citep{MLG15,LiGM20}. Broader results have been established for RCOP models on complete graphs \citep{GIKM22} and for RCOP models on homogeneous graphs \citep{GIK22}.\footnote{A graph is homogeneous if it is decomposable and does not contain the chain $\bullet-\bullet-\bullet-\bullet$ as an induced subgraph, see \citep{LM07}.}

\subsection{Contribution of the paper}
We study the normalizing constants from Section~\ref{sec:bayesian-model-selection} within the more general framework of gamma-like integrals
\begin{equation}\label{eq:general-gamma-integral}
\int_{\mathcal{Z}\cap\mathrm{Sym}^+(p)}
(\det x)^s e^{-\tr(Ax)}\,\dd x,
\end{equation}
where $\mathcal{Z}$ is a linear subspace of $\mathrm{Sym}(p)$, $s\in\mathbb{R}$, $A\in\mathrm{Sym}^+(p)$, and $\dd x$ is the Lebesgue measure on $\mathcal{Z}$ normalized with respect to the trace inner product. The space $\mathcal{Z}$ need not arise from a colored graph. When it does, the associated Gaussian graphical model is an RCON model, with $\mathcal{Z}=\mathcal{Z}_g^c$ and $\mathcal{P}_g^c=\mathcal{Z}_g^c\cap\mathrm{Sym}^+(p)$, where $\mathcal{Z}_g^c$ is the colored matrix space defined in Section~\ref{sec:matrix_subspaces}. Thus, $I_{g,c}(\delta,D)$ is the special case of \eqref{eq:general-gamma-integral} with $\mathcal{Z}=\mathcal{Z}_g^c$, $s=(\delta-2)/2$, and $A=D/2$.

Our main analytic result, Theorem~\ref{thm:integral}, gives a convergence criterion and a product formula for \eqref{eq:general-gamma-integral} whenever $\mathcal{Z}$ satisfies conditions (Z0) and (Z1), that is, whenever $\mathcal{Z}$ is a Block-Cholesky space (\BC-space). The formula is expressed in terms of model-specific structure constants, whose determination is in general a difficult problem. If condition (Z2) is also satisfied, the space is a Diagonally Commutative Block-Cholesky space (\cBC-space). For \cBC-spaces arising from colored RCON models, we provide an explicit and efficient procedure for computing all the structure constants and determinant factors required by the formula. In this paper, we introduce novel subclasses of RCON models whose matrix spaces satisfy these algebraic conditions.

On the combinatorial side, we identify graph-theoretic conditions that guarantee these algebraic properties. Recall that decomposability of a graph is equivalent to the existence of a perfect elimination ordering (peo) of its vertices. We introduce the related notion of a color perfect elimination ordering (cpeo), namely, a peo that respects the vertex coloring. A cpeo together with a $2$-path regularity condition leads to the class of Color Elimination-Regular graphs (CER graphs), whose matrix spaces are \BC-spaces. An additional symmetry condition defines symmetric CER graphs, whose matrix spaces are \cBC-spaces. The definition of $2$-path regularity is closely related to the Butterfly criterion introduced by \citet{Seigal23}. Interestingly, counts of certain paths of length $2$ also play a crucial role in the approximation of normalizing constants for general GGMs derived by \citet[Theorem~1]{MML23}, suggesting that their result might be extended to $2$-path regular colored graphs.

Since a cpeo is, in particular, a peo, every CER graph is decomposable. Our framework extends the classical theory of uncolored decomposable graphs to the colored setting, thus substantially broadening its applicability to CGGMs in high-dimensional scenarios.

Our proposed class encompasses all RCOP models associated with decomposable graphs. Additionally, we emphasize that RCOP models provide a natural interpretation of graph coloring through the concept of distributional invariance of Gaussian vectors \citep{A75,gips}. In the symmetric setting with a single vertex color, we also uncover a close connection between \cBC-spaces and Bose-Mesner algebras, matrix algebras arising in the context of association schemes \citep{BM59}.

Our computational contribution extends the range of models for which the required structure constants can be determined efficiently. Whereas \citet{GIKM22} provided an efficient method for computing the structure constants for RCOP models corresponding to complete graphs colored by cyclic groups, our procedure applies to colored RCON models whose matrix spaces are \cBC-spaces.

\subsection{Structure of the paper}

Section~\ref{sec:preli} reviews the graph-theoretic and matrix
preliminaries. Section~\ref{sec:newclass} introduces the novel classes of CER
and symmetric CER graphs and studies their structural properties,
including connections with Bose--Mesner algebras.
Section~\ref{sec:bcspace} develops the framework of \BC-spaces
and \cBC-spaces and presents the main integral formula.
Section~\ref{sec:computation} describes the Random Method and
the Gram--Cholesky method for computing the required structure
constants and determinant factors. Section~\ref{sec:relations}
discusses relations to restricted DAG (RDAG) models.

Section~\ref{sec:conclusion} summarizes the results and outlines
directions for future research. Detailed proofs are collected
in Section~\ref{sec:proofs}. Appendix~\ref{app:graphical-examples}
provides graphical examples and specializes the integral formula
to strongly regular graph models.
Appendix~\ref{subsec:random_method_example} gives a fully worked
example, from the computation of structure constants to a Bayesian
model-selection calculation, with an independent numerical check
using the \textsf{gips} package.

\section{Preliminaries}\label{sec:preli}

We write $[n]=\{1,2,\ldots,n\}$ for $n\in\mathbb{N}=\{1,2,\ldots\}$, and set $[0]=\emptyset$. For $n\in\mathbb{N}$, we denote by $I_n$ the $n\times n$ identity matrix.

\subsection{Basic graph theoretic notions}\label{sec:basics}

Let \(G=(V,E)\) be an undirected simple graph, where \(V\) is a nonempty finite set and \(E\) is a subset of the set $\mathcal{P}_2(V)$ of unordered pairs of distinct elements of $V$. We denote the number of vertices by \(p=|V|\). For vertices $v,w\in V$, we write \(v\sim w\) if \(\{v,w\}\in E\) and $v\nsim w$ otherwise.

For any subset of vertices \(A\subset V\), the \textbf{induced subgraph} \(G[A]\) is the graph \((A,E_A)\), where $E_A = \{ \{v,w\}\in E\colon v,w\in A\}$.
A graph $G$ is \textbf{decomposable} if every induced cycle in the graph is a triangle (i.e. there are no induced cycles of length greater than 3). The \textbf{neighborhood} of a vertex $v\in V$ is the set of its adjacent vertices, defined as \(\nb(v)=\{w\in V\colon w\sim v\}\).
A subset $C\subseteq V$ is called a \textbf{clique} if every pair of distinct vertices in $C$ is adjacent in $G$.
A vertex $v\in V$ is called \textbf{simplicial} if its neighborhood \(\nb(v)\) forms a clique in $G$ (equivalently, if $\{v\}\cup\nb(v)$ forms a clique). 

Consider a permutation $\sigma\in\mathfrak{S}_p$, where $\mathfrak{S}_p$ is the symmetric group of degree $p$. An ordering of the vertices  $(v_{\sigma_1},\ldots,v_{\sigma_p})$ is a \textbf{perfect elimination ordering (peo)} if 
\[
\forall\, k\in[p]\quad v_{\sigma_k}\mbox{ is simplicial in }G[\{v_{\sigma_k},\ldots,v_{\sigma_p}\}].
\]
 The \textbf{automorphism group} of $G$, denoted $\mathrm{Aut}(G)$,  
is the set of all permutations 
$\sigma$ of $V$ that preserve the adjacency structure of the graph:
\[
\sigma(v) \sim \sigma(w)\quad\iff\quad v \sim w\qquad \mbox{for all } v, w \in V. 
\]

\subsection{Edge-vertex coloring of a graph}\label{sec:ev_colorings}
An \textbf{edge-vertex coloring} (or simply a \textbf{coloring}) of a graph \(G=(V,E)\) is an ordered pair $\mathcal{C} = (\mathcal{V}, \mathcal{E})$,
where $\mathcal{V} = \{V_1, V_2, \ldots, V_r\}$ is a partition of the vertex set $V$ into \textbf{vertex color classes} and 
$\mathcal{E} = \{E_{r+1}, E_{r+2}, \ldots, E_{r+R}\}$ is a partition of the edge set $E$ into \textbf{edge color classes}. 

This edge-vertex coloring of $G$ can be conveniently represented by extending the graph to include loops. A \textbf{reflexive extension} of \(G\), denoted \(\tilde{G} = (V,\tilde{E})\), is formed by adding a loop at every vertex: 
\[
\tilde{E} = E\cup \{\{v\}\colon v\in V\}.
\]
We refer to \(\tilde E\) as the set of \textbf{extended edges} of \(G\). 

The vertex coloring \(\mathcal{V}\) can be translated into a coloring of loops. 
For each vertex color class \(V_i\), we define a corresponding loop color class:
\[
E_i = \big\{\{v\}\colon v\in V_i\big\}.
\]
This creates a unified edge coloring
\[
\tilde{\mathcal{E}} = \{E_1,\ldots,E_r,E_{r+1},\ldots,E_{r+R}\},
\] 
which is a partition of the extended edge set \(\tilde{E}\).

This reflexive extension transformation, denoted \(\RefExt\colon(G, \mathcal{C}) \mapsto (\tilde{G}, \tilde{\mathcal{E}})\), preserves all information from the original colored graph. Consequently, we can identify the coloring $\mathcal{C}$ of $G$ with the partition \(\tilde{\mathcal{E}}\) of \(\tilde{G}\).

This allows us to define a single coloring function \(c\colon \tilde{E}\to [r+R]\) by 
\[
c(\{v, w\}) = k \quad \text{if and only if} \quad \{v, w\} \in E_k.
\]
We adopt the convention of writing \(c(v,w)\) for the color of the edge \(\{v,w\}\) and \(c(v) = c(v,v)\) for the color of vertex \(v\). 
We also extend this function to non-edges by setting $c(v,w)=0$ if $\{v,w\}\notin\tilde E$. 

\begin{definition}
A colored graph $(G,\mathcal{C})$ is \textbf{edge-regular} if any two edges of the same color connect the
same vertex color classes. Formally, if $\{v,w\}$ and $\{v',w'\}$ are edges of the same color, then
\[
\{c(v),c(w)\}=\{c(v'),c(w')\}.
\]

A colored graph  $(G,\mathcal{C})$ is \textbf{vertex-regular} if any two vertices with the same color have the same number of neighbors of a given color in each vertex color class. Formally, if $c(v)=c(w)$, then 
\[
\forall\,i\in [r],\, j\in[r+R]\setminus[r] \quad |\{ u\in V_i\colon c(v,u)=j \}| = |\{ u\in V_i\colon c(w,u) = j\}|.
\]
\end{definition}
For illustration, the colored three-vertex path in
Example~\ref{ex:simple-cpeo} is both edge-regular and vertex-regular.
The two regularity conditions are independent: the coloring in
Example~\ref{ex:cpeo-dependence} is vertex-regular but not edge-regular,
whereas the coloring in Example~\ref{ex:elementary-non-group-cer} is
edge-regular but not vertex-regular.

These regularity conditions were discussed by \citet{GeLau2012}, with vertex regularity appearing earlier in a non-statistical context \citep{Siemons83,Sachs}.

An important example of graph colorings arises from permutation subgroups.  Let $\Gamma$ be a subgroup of $\mathrm{Aut}(G)$. The action of $\Gamma$ on the set of extended edges $\tilde{E}$ is defined by 
\[
\sigma\cdot\{v,w\} = \{\sigma(v),\sigma(w)\}
\]
for $\sigma\in\Gamma$ and $\{v,w\}\in \tilde{E}$. This action partitions $\tilde{E}$ into orbits $\tilde{\mathcal{E}}=\{E_k\}_k$, which naturally define a coloring of $G$. By construction, any permutation in $\Gamma\subset\mathrm{Aut}(G)$ is an automorphism of the colored graph $(G,\mathcal{C})$.

\subsection{Matrix subspaces related to a colored graph}\label{sec:matrix_subspaces}

Let $G=(V,E)$ be a simple undirected graph with $V=\{1,\ldots,p\}$. We define the matrix space
\[
\mathcal{Z}_G = \left\{ x\in\mathrm{Sym}(p)\colon x_{ij}=0\mbox{ whenever } \{i,j\}\notin E\mbox{ for }i\neq j\right\}, 
\]
where $\mathrm{Sym}(p)$ is the space of real symmetric $p\times p$ matrices.

Given a coloring $\mathcal{C}=(\mathcal{V},\mathcal{E})$ of $G$, we define
\[
\mathcal{Z}_G^{\mathcal{C}} = \left\{ x\in\mathcal{Z}_G \colon x_{ij}=x_{kl}\mbox{ whenever } c(i,j) = c(k,l) \mbox{ for }\{i,j\},\{k,l\}\in\tilde{E}\right\}. 
\]
This space imposes equality constraints on the entries of the matrices corresponding to vertices and edges of the same color. 

The subspace $\coloredZ$ corresponds to the RCON (Restricted CONcentration) model introduced by \citet{HL08}, which places equality restrictions on the entries of the concentration matrix (inverse covariance matrix).
When the coloring is derived from a subgroup $\Gamma$ of $\mathrm{Aut}(G)$, we obtain the RCOP (Restricted by Permutation) model: 
\[
\mathcal{Z}_G^{\Gamma} = \left\{ x\in\mathcal{Z}_G \colon   x_{\sigma(i)\sigma(j)}=x_{ij}\mbox{ for all }i,j\in V, \sigma\in\Gamma\right\}. 
\] 
If the coloring $\mathcal{C}$ is defined by the orbits of $\Gamma$, then $\mathcal{Z}_G^{\Gamma} = \mathcal{Z}_G^{\mathcal{C}}$. 

It is possible for different groups to generate the same subspace.  For instance, for a complete graph on $3$ vertices, $G=K_3$, the groups $\Gamma_1=\left<(1,2,3)\right>$ and $\Gamma_2=\mathfrak{S}_3$ both yield the same subspace, $\mathcal{Z}_{G}^{\Gamma_1} = \mathcal{Z}_{G}^{\Gamma_2}$. 
To resolve this ambiguity, in the case of complete graphs one can associate each such subspace with the largest group that generates it, known as the \textbf{$2^\ast$-closure} of the group \citep{Wielandt,Siemons82,Siemons83} and \citep[p.~1502]{Comb}.

As shown by \citet{GeLau2012}, if a coloring is edge-regular, the same restrictions define an RCOR (Restricted CORrelation) model. RCOP models are both edge and vertex regular, making them a subclass of both RCON and RCOR models. This subclass is notable for its combinatorial structure and its statistical interpretation in terms of distributional invariance under the group action.

We consider colored graphs that need be neither vertex- nor edge-regular. This family sits strictly inside the RCON models while strictly containing the RCOP models associated with decomposable graphs.

 If $(G,\mathcal{C})$ is edge-regular, the space $\coloredZ$ decomposes into a direct sum of blocks:
    \[
\coloredZ= \bigoplus_{1\leq k\leq l \leq r} M_{lk},
    \]
where $M_{lk}$ is the linear space spanned by the $V_l\times V_k$ and $V_k\times V_l$ blocks of $\mathcal{Z}_G^{\mathcal{C}}$. 

With each colored graph $(G,\mathcal{C})$, we associate a \textbf{colored Gaussian graphical model} (or \textbf{Markov Random Field}) given by
\[
\{ \mathrm{N}_p(0,K^{-1})\colon K\in \coloredZ\cap\mathrm{Sym}^+(p)\},
\]
where $\mathrm{Sym}^+(p)$ is the cone of symmetric positive definite $p\times p$ matrices.
This RCON model forms a natural exponential family.

\section{New class of colored graphs}\label{sec:newclass}

Our goal is to identify colored graphs whose matrix spaces admit the generalized Cholesky decomposition used to evaluate the gamma-like integrals in Section~\ref{sec:bcspace}. We introduce two ingredients: an elimination ordering of vertex-color classes and a regularity condition on colored $2$-path counts. Together, they define the class of Color Elimination-Regular (CER) graphs.

\subsection{Color perfect elimination ordering}\label{sec:cpeo}

Grouping vertices by color determines the blocks in the matrix decomposition considered in Section~\ref{sec:bcspace}. To carry out elimination at the level of these blocks, we seek an ordering of the vertex-color classes such that every vertex in the class being eliminated is simplicial in the remaining graph.

Let the vertex set $V$ be partitioned into color classes $\mathcal{V} = \{V_1,\ldots,V_r\}$. This defines a vertex coloring function $c\colon V\to [r]$ given by
\[
c(v)=i\quad\mbox{ if and only if }\quad v\in V_i.
\]

\begin{definition}
A \textbf{color perfect elimination ordering (cpeo)} is an ordered partition of $V$, $(V_{\eta_1},\ldots, V_{\eta_r})$, determined by a permutation $\eta\in\mathfrak{S}_r$ of  vertex-colors, such that every vertex is simplicial in the subgraph induced by its own color class and all subsequent color classes. Formally:
\[
\forall\, i\in [r]\quad \forall\, v\in V_{\eta_i}\quad v\mbox{ is simplicial in }G[V_{\eta_i}\cup \ldots \cup V_{\eta_r}]. 
\]
\end{definition}
The definition of a cpeo depends only on the underlying graph and its vertex coloring; edge colors play no role. A cpeo allows vertices to be eliminated one color class at a time, in an arbitrary order within each class. Every vertex ordering obtained in this way is a peo: each vertex is simplicial when elimination of its class begins, and deleting other vertices preserves this property. Consequently, the existence of a cpeo implies that the underlying graph is decomposable, although decomposability alone does not guarantee a cpeo for a prescribed vertex coloring.

For example, consider the path $\textcolor{graphblue}{1}-\textcolor{graphred}{2}-\textcolor{graphgreen}{3}-\textcolor{graphblue}{4}-\textcolor{graphred}{5}$ with vertex-color classes $\{\textcolor{graphblue}{1},\textcolor{graphblue}{4}\}$, $\{\textcolor{graphred}{2},\textcolor{graphred}{5}\}$, and $\{\textcolor{graphgreen}{3}\}$. The underlying graph is decomposable and admits the peo $(\textcolor{graphblue}{1},\textcolor{graphred}{2},\textcolor{graphgreen}{3},\textcolor{graphblue}{4},\textcolor{graphred}{5})$. However, its only simplicial vertices are $\textcolor{graphblue}{1}$ and $\textcolor{graphred}{5}$, so no vertex-color class consists entirely of simplicial vertices. Hence, this vertex coloring admits no cpeo.

\begin{remark}
\begin{enumerate}
\item[(i)] If $G$ is complete, then every ordering of its vertex-color classes is a cpeo.
\item[(ii)] If every vertex has a distinct color, then, identifying each singleton color class with its unique vertex, a cpeo is exactly a peo. Thus, cpeo generalizes the classical notion of peo.
\end{enumerate}
\end{remark}

Algorithm~\ref{alg:1} searches for a cpeo by successively removing a vertex-color class whose vertices are all simplicial in the current induced graph. When every vertex has a distinct color, this is the usual simplicial-elimination rule for finding a peo: select a simplicial vertex, remove it, and repeat; see \citet[Section~4.2]{VA14}. The algorithm therefore extends this elimination procedure from individual vertices to entire vertex-color classes.
\begin{algorithm}
\caption{Greedy Algorithm for finding a cpeo \\
\texttt{greedy\_find\_cpeo(\(G, \mathcal{V}\))} }\label{alg:1}
\SetAlgoLined
\KwIn{Graph \(G = (V,E)\) and a vertex coloring \(\mathcal{V} = \{V_1, V_2, \ldots, V_r\}\)}
\KwOut{cpeo $\eta\in\mathfrak{S}_r$ if it exists, or \texttt{null} otherwise}

\(\eta \gets ()\)\;
\While{\(G\) has at least one vertex}{
    Find a remaining class \(V_k\) whose vertices are all simplicial in \(G\)\;
    \If{no such class exists}{
        \Return \texttt{null}\;
    }
    Append \(k\) to \(\eta\)\;
    Remove the vertices of \(V_k\) from \(G\)\;
}
\Return \(\eta\)\;
\end{algorithm}

The correctness of Algorithm~\ref{alg:1} follows from the fact that deleting vertices preserves simpliciality of the remaining vertices that were simplicial before deletion. If a cpeo exists, removing any vertex-color class consisting entirely of simplicial vertices leaves a graph admitting a cpeo: simply delete that class from the existing ordering. Thus, every eligible choice preserves the existence of a cpeo, whereas the absence of an eligible class in a nonempty graph rules out a cpeo. Since each iteration removes a nonempty class, the algorithm terminates, and every complete ordering it returns satisfies the definition of a cpeo.

\begin{example}[A cpeo for a three-vertex path]\label{ex:simple-cpeo}
\mbox{}\par\noindent
\begin{minipage}[c]{0.60\linewidth}
Let $G=(V,E)$ be the path with $V=\{1,2,3\}$,
$E=\bigl\{\{1,3\},\{2,3\}\bigr\}$, and vertex-color classes
$V_1=\{1,2\}$ and $V_2=\{3\}$. The two edges are given the same color,
although their color plays no role in the cpeo.
\end{minipage}
\hfill
\raisebox{0.2\baselineskip}{%
\begin{minipage}[c]{0.36\linewidth}
    \centering
    \begin{tikzpicture}[
        thick,
        blue node/.style={circle, draw=graphblue, minimum size=0.7cm, inner sep=0pt},
        red node/.style={rectangle, draw=graphred, minimum width=0.7cm,
            minimum height=0.7cm, inner sep=0pt},
        green edge/.style={draw=graphgreen, semithick}
    ]
        \node[blue node] (1) at (0, 1.5) {\color{graphblue}\textbf{1}};
        \node[blue node] (2) at (2.5, 1.5) {\color{graphblue}\textbf{2}};
        \node[red node] (3) at (1.25, 0) {\color{graphred}\textbf{3}};

        \draw[green edge] (1) -- (3);
        \draw[green edge] (2) -- (3);
    \end{tikzpicture}
\end{minipage}%
}
\par

The graph is a tree, so it has no cycles and is therefore chordal. The ordering $(V_1,V_2)$ is a cpeo:
vertices $1$ and $2$ are simplicial in $G$, since each has only one neighbor,
and after eliminating $V_1$ the remaining induced graph $G[V_2]$ consists of
the single vertex $3$. The reverse ordering $(V_2,V_1)$ is not a cpeo,
because vertex $3$ is not simplicial in $G$: its neighbors $1$ and $2$ are not
adjacent. Thus, Algorithm~\ref{alg:1} selects $V_1$ first and returns the
ordering $(V_1,V_2)$.

For comparison, $(1,3,2)$ is an ordinary peo. After vertex $1$ is eliminated,
the remaining graph consists of the single edge $\{2,3\}$. This peo is not
consistent with an ordering of the color classes, since vertex $3\in V_2$ occurs
between the two vertices of $V_1$. Thus, a cpeo controls elimination at the level
of entire vertex-color classes, rather than merely providing an ordering of the
individual vertices.
\end{example}

By Definition~3.1, an ordering $(V_{\eta_1},\ldots,V_{\eta_r})$ is a cpeo if and only if, for all positions $1\le k\le j\le i\le r$ and vertices $v\in V_{\eta_k}$, $w\in V_{\eta_j}$, and $w'\in V_{\eta_i}$ with $w\ne w'$, one has
\begin{equation}\label{eq:cpeo-chordality}
w\sim v\mbox{ and }v\sim w'\implies w\sim w'.
\end{equation}

\begin{remark}\label{rem:cliques}
If a graph with vertex-color classes $V_1,\ldots,V_r$ admits a cpeo, then each induced subgraph $G[V_i]$ is a disjoint union of cliques. Indeed, applying \eqref{eq:cpeo-chordality} with $k=j=i$ shows that any two distinct vertices in the same vertex-color class with a common neighbor in that class are adjacent.
\end{remark}

\subsection{2-path regularity}\label{sec:2path}
Given an ordering $\eta\in\mathfrak{S}_r$ of the vertex-color classes, which is not
necessarily a cpeo, we index the vertex-color classes according to their positions
in this ordering. Thus, without loss of generality, we write the ordered partition
as $(V_1,\ldots,V_r)$, with $c(v)=i$ whenever $v\in V_i$.

We define the set of vertices up to color $i$ as
\[
V_{\leq i} \coloneq V_1 \cup \cdots \cup V_i
= \{u\in V\colon c(u)\leq i\}.
\]
We now introduce a function that counts specific $2$-paths in the graph and depends on the chosen ordering of the vertex-color classes.

For any vertices $v,w\in V$, the $2$-path function from $v$ to $w$, 
\[
m_{v\to w}\colon [r+R] \times [r+R] \to \mathbb{N}\cup\{0\},
\]
is defined as
\begin{align*}
m_{v\to w}(k, h) = \Bigl| \Big\{ u \in V_{\le c(v) \wedge c(w)} \colon c(v,u)=k\mbox{ and } c(u,w)=h\Big\} \Bigr|,
\end{align*}
where we write $a\wedge b = \min\{a,b\}$.
This function counts the number of paths of length $2$ from \(v\) to \(w\), where the intermediate vertex
$u$ belongs to a color class that appears no later than those of $v$ and $w$ and where the edge from \(v\) to $u$ is colored \(k\) and the edge from $u$ to \(w\) is colored \(h\). 

For any vertices $v,w\in V$, we introduce the symmetric $2$-path function 
\[
m_{v\leftrightarrow w}\colon [r+R]\times[r+R]\to \mathbb{N}\cup\{0\},
\]
defined by 
\begin{align*}
m_{v\leftrightarrow w}(k, h) = m_{v\to w}(k, h) + m_{v\to w}(h, k).
\end{align*}
This symmetrization ensures that the count of $2$-paths is unchanged when swapping the roles of $k$ and $h$ or when swapping the roles of $v$ and $w$; indeed we have $m_{v\to w}(h, k) = m_{w\to v}(k, h)$.

\begin{remark}
Both \(m_{v\to w}\) and \(m_{v\leftrightarrow w}\) depend on the ordering of the vertex color classes, since the intermediate vertices are restricted to \(V_{\le c(v)\wedge c(w)}\).
\end{remark}

\begin{example}[Computing $2$-path functions]\label{ex:computing-two-paths}
\mbox{}\par\noindent
\begin{minipage}[c]{0.60\linewidth}
Consider the complete graph on vertices
$\textcolor{graphblue}{1},\textcolor{graphred}{2},\textcolor{graphgreen}{3}$,
with distinct vertex colors and edge colors
$c(\textcolor{graphblue}{1},\textcolor{graphred}{2})
=c(\textcolor{graphblue}{1},\textcolor{graphgreen}{3})=\textcolor{graphorange}{a}$
and $c(\textcolor{graphred}{2},\textcolor{graphgreen}{3})=\textcolor{graphpurple}{b}$.
We keep this coloring fixed and compare two orderings of its singleton
vertex-color classes.

\end{minipage}
\hfill
\begin{minipage}[c]{0.36\linewidth}
    \centering
    \begin{tikzpicture}[-, thick,
        blue node/.style={circle, draw=graphblue, font=\sffamily\bfseries, minimum size=0.65cm, inner sep=0pt},
        red node/.style={rectangle, draw=graphred, font=\sffamily\bfseries, minimum width=0.65cm, minimum height=0.65cm, inner sep=0pt},
        green node/.style={regular polygon, regular polygon sides=3, draw=graphgreen, font=\sffamily\bfseries, minimum size=0.8cm, inner sep=0pt}]
        \node[blue node]  (1) at (0, 1.5) {\color{graphblue}1};
        \node[red node]   (2) at (3.2, 1.5) {\color{graphred}2};
        \node[green node] (3) at (1.6, 0) {\color{graphgreen}3};
        \path
            (1) edge[graphorange] node[above] {$\textcolor{graphorange}{a}$} (2)
            (1) edge[graphorange] node[below left] {$\textcolor{graphorange}{a}$} (3)
            (2) edge[graphpurple, dashed] node[below right] {$\textcolor{graphpurple}{b}$} (3);
    \end{tikzpicture}
\end{minipage}
\par

For the ordering $(\{\textcolor{graphblue}{1}\},\{\textcolor{graphred}{2}\},\{\textcolor{graphgreen}{3}\})$, the only admissible intermediate
vertex for a $2$-path from $\textcolor{graphblue}{1}$ to $\textcolor{graphgreen}{3}$ is $u=\textcolor{graphblue}{1}$. Since the loop at $\textcolor{graphblue}{1}$ has
a vertex color, distinct from the edge colors $\textcolor{graphorange}{a},\textcolor{graphpurple}{b}$, we obtain
\[
m_{\textcolor{graphblue}{1}\to\textcolor{graphgreen}{3}}(\textcolor{graphorange}{a},\textcolor{graphpurple}{b})=0,\qquad m_{\textcolor{graphblue}{1}\to\textcolor{graphgreen}{3}}(\textcolor{graphpurple}{b},\textcolor{graphorange}{a})=0,\qquad
m_{\textcolor{graphblue}{1}\leftrightarrow\textcolor{graphgreen}{3}}(\textcolor{graphorange}{a},\textcolor{graphpurple}{b})=0.
\]
For the ordering $(\{\textcolor{graphred}{2}\},\{\textcolor{graphgreen}{3}\},\{\textcolor{graphblue}{1}\})$, the admissible intermediate
vertices are $u\in\{\textcolor{graphred}{2},\textcolor{graphgreen}{3}\}$. Now $\textcolor{graphblue}{1}\to\textcolor{graphred}{2}\to\textcolor{graphgreen}{3}$ contributes one $2$-path
with successive edge colors $(\textcolor{graphorange}{a},\textcolor{graphpurple}{b})$, whereas no $2$-path has successive
colors $(\textcolor{graphpurple}{b},\textcolor{graphorange}{a})$. Thus
\[
m_{\textcolor{graphblue}{1}\to\textcolor{graphgreen}{3}}(\textcolor{graphorange}{a},\textcolor{graphpurple}{b})=1,\qquad m_{\textcolor{graphblue}{1}\to\textcolor{graphgreen}{3}}(\textcolor{graphpurple}{b},\textcolor{graphorange}{a})=0,\qquad
m_{\textcolor{graphblue}{1}\leftrightarrow\textcolor{graphgreen}{3}}(\textcolor{graphorange}{a},\textcolor{graphpurple}{b})=1.
\]
In these two displays, the functions are evaluated with respect to the
corresponding ordering. Their different values arise solely from the
change in the admissible intermediate vertices.
\end{example}

For $i\in[r]$, define
\begin{align}\label{eq:Fi}
F_i
=
\left\{
c(v,w)\colon
c(v)=i=c(w)
\mbox{ and }
\{v,w\}\in\tilde E
\right\}.
\end{align}
Thus, $F_i$ collects the colors of all extended edges contained in the
vertex-color class $V_i$: the loop color $i$ and the colors of
edges in the induced graph $G[V_i]$. Equivalently, $F_i$ is the set of
colors represented in the $V_i\times V_i$ diagonal block of
$\mathcal Z_G^{\mathcal C}$.

We also define
\[
F=\bigcup_{i\in[r]}F_i.
\]
Thus, $F$ collects all colors that occur in at least one diagonal block.
Colors outside $F$ occur only on edges between distinct vertex-color
classes and hence only in off-diagonal blocks. This distinction motivates
the restriction to $F\times F$ in condition (M2) below and will be used
later when analyzing the block-diagonal part of
$\mathcal Z_G^{\mathcal C}$.

\begin{definition}[$2$-path regularity]\label{def:2_path_regularity}
With the convention above (that is, after relabeling the vertex-color classes so
that the fixed ordering satisfies $\eta_i=i$ for every $i\in[r]$), a colored
graph \((G,\mathcal{C})\) is
\textbf{$2$-path regular with respect to the fixed ordered partition}
$(V_1,\ldots,V_r)$ if the symmetric $2$-path function is constant on
extended-edge color classes, i.e.,
\begin{enumerate}
    \item[(M1)] for all \(\{v, w\}, \{v', w'\} \in \tilde{E}\), we have
    \[
c(v,w)=c(v',w')\quad\implies\quad     m_{v\leftrightarrow w} = m_{v'\leftrightarrow w'}.
    \]
\end{enumerate}
The colored graph \((G,\mathcal{C})\) is
\textbf{symmetric $2$-path regular with respect to the fixed ordered partition}
$(V_1,\ldots,V_r)$ if it is
$2$-path regular with respect to this ordered partition and, additionally, for
any two vertices $v$ and $w$ of the same color, the $2$-path count from $v$ to
$w$ is the same as from $w$ to $v$:
\begin{enumerate}
 \item[(M2)] for all $v,w\in V$, 
 \[
    c(v)=c(w)\quad\implies\quad 
    m_{v\to w}\big|_{F\times F} = m_{w\to v}\big|_{F\times F}.
    \]
\end{enumerate}

We say simply that $(G,\mathcal{C})$ is \textbf{$2$-path regular} if there exists an
ordering of its vertex-color classes with respect to which it is $2$-path
regular. Similarly, $(G,\mathcal{C})$ is \textbf{symmetric $2$-path regular} if there
exists an ordering of its vertex-color classes with respect to which it is
symmetric $2$-path regular. For each such ordering, the vertex-color classes
are relabeled according to the convention above.
\end{definition}

Under (M1), colors in $F$ never connect distinct vertex-color classes:
if $k\in F_i$ and $c(v,w)=k$, then necessarily $c(v)=c(w)=i$, see Lemma \ref{lem:edge_regularity_on_diagonal}.
Consequently, for $c(v)=c(w)=i$ and $k,h\in F$, we have
$m_{v\to w}(k,h)=0$ unless both $k$ and $h$ belong to $F_i$.
For $k,h\in F_i$, we may restrict the intermediate vertex to $V_i$,
\begin{align*}
m_{v\to w}(k, h) = \Bigl| \Big\{ u \in V_{i} \colon c(v,u)=k\mbox{ and } c(u,w)=h\Big\} \Bigr|.
\end{align*}

\subsection{Color Elimination-Regular Graph}\label{sec:cer}

\begin{definition}
Let $(G,\mathcal{C})$ be a colored graph with a vertex partition $\mathcal{V}=\{V_1,\ldots,V_r\}$. 

We say that $(G,\mathcal{C})$ is \textbf{color elimination-regular (CER)} if there exists an ordered partition of its vertices, $(V_{\eta_1},\ldots,V_{\eta_r})$, that satisfies two conditions simultaneously:
\begin{itemize}
    \item[(i)] The ordering is a color perfect elimination ordering (cpeo) for $G$. 
    \item[(ii)] The graph $(G,\mathcal{C})$ is $2$-path regular with respect to the same ordering.
\end{itemize}

A \textbf{symmetric CER graph} is a CER graph for which, in addition, $(G,\mathcal{C})$ is symmetric $2$-path regular with respect to the same cpeo.
\end{definition}

Recall from Remark~\ref{rem:cliques} that if a graph admits a cpeo, then for each color class $V_i$ the induced subgraph $G[V_i]$ is a disjoint union of cliques. 
If $(G,\mathcal{C})$ is CER, we can refine this as follows.

\begin{propo}\label{prop:regularity-within-color-classes}
Let $(G,\mathcal{C})$ be a CER graph, and let $\mathcal{C}_i$ denote the
coloring of $G[V_i]$ induced by $\mathcal{C}$. Then, for every $i\in[r]$,
$(G[V_i],\mathcal{C}_i)$ is both vertex-regular and edge-regular.
In particular, all connected components of $G[V_i]$ are cliques of the same size.
However, even if $(G,\mathcal{C})$ is symmetric CER, the connected components
of $G[V_i]$ need not be isomorphic as colored graphs.
\end{propo}

\begin{remark}[Basic caveats]
Several caveats concerning CER graphs are worth recording.
\begin{enumerate}
    \item[(i)] The $2$-path conditions (M1) and (M2) alone do not imply the
    existence of a cpeo. For example, a $4$-cycle with a single vertex color
    and a single edge color satisfies (M1)--(M2) with respect to the unique
    vertex-color ordering $(V_1)$, but this ordering is not a cpeo, since the
    graph is not decomposable.

    \item[(ii)] A CER graph need be neither edge-regular nor vertex-regular in
    the sense of Section~\ref{sec:ev_colorings}. The graph in
    Example~\ref{ex:elementary-non-group-cer} is not vertex-regular, whereas
    the graph in Example~\ref{ex:cpeo-dependence} is not edge-regular.

    \item[(iii)] Symmetric CER colorings need not arise from permutation
    groups. Example~\ref{ex:elementary-non-group-cer} gives a simple instance
    with two vertex colors. The same phenomenon can occur even for a complete
    graph with a single vertex color, as shown by the Shrikhande graph in
    Example~\ref{ex:shrikhande-non-schurian}.
\end{enumerate}
\end{remark}

The next two examples illustrate two further distinctions that require more
detailed constructions: $2$-path regularity may depend on the choice of cpeo,
and a CER graph need not be symmetric CER.

\begin{example}[Dependence on the choice of cpeo]\label{ex:cpeo-dependence}
We revisit the colored graph from Example~\ref{ex:computing-two-paths}.
The choice of cpeo is essential: a colored graph may admit two different cpeos, one
for which (M1) and (M2) hold and another for which (M1) fails. Consider the following
$3$-vertex example. The three vertices have distinct colors, the two edges incident
to vertex~1 share the same color, and the remaining edge has a different color. The
colored graph and its associated matrix space are shown below.

\begin{center}
\begin{minipage}[c]{0.31\textwidth}
    \centering
    \begin{tikzpicture}[-, thick,
        blue node/.style={circle, draw=graphblue, font=\sffamily\bfseries, minimum size=0.65cm, inner sep=0pt},
        red node/.style={rectangle, draw=graphred, font=\sffamily\bfseries, minimum width=0.65cm, minimum height=0.65cm, inner sep=0pt},
        green node/.style={regular polygon, regular polygon sides=3, draw=graphgreen, font=\sffamily\bfseries, minimum size=0.8cm, inner sep=0pt}]

        \node[blue node]  (1) at (0, 1.5) {\color{graphblue}1};
        \node[red node]   (2) at (3.2, 1.5) {\color{graphred}2};
        \node[green node] (3) at (1.6, 0) {\color{graphgreen}3};

        \path
            (1) edge[graphorange] (2)
            (1) edge[graphorange] (3)
            (2) edge[graphpurple, dashed] (3);
    \end{tikzpicture}
\end{minipage}
\hfill
\begin{minipage}[c]{0.66\textwidth}
    \centering
    \[
    \coloredZ=
    \left\{
    \begin{bmatrix}
    \textcolor{graphblue}{\alpha} & \textcolor{graphorange}{a} & \textcolor{graphorange}{a}\\
    \textcolor{graphorange}{a} & \textcolor{graphred}{\beta} & \textcolor{graphpurple}{b}\\
    \textcolor{graphorange}{a} & \textcolor{graphpurple}{b} & \textcolor{graphgreen}{\gamma}
    \end{bmatrix}
    \colon
    \textcolor{graphblue}{\alpha},\textcolor{graphred}{\beta},\textcolor{graphgreen}{\gamma},
    \textcolor{graphorange}{a},\textcolor{graphpurple}{b}\in\R
    \right\}.
    \]
\end{minipage}
\end{center}

The vertex-color classes are
\[
V_1=\{1\}\ \text{(\textcolor{graphblue}{blue})},\qquad
V_2=\{2\}\ \text{(\textcolor{graphred}{red})},\qquad
V_3=\{3\}\ \text{(\textcolor{graphgreen}{green})}.
\]
Both $\eta=(1,2,3)$ and $\eta'=(2,3,1)$ are cpeos. Under the first
ordering, the graph is symmetric CER, whereas (M1) fails for the second
ordering. The claims about (M1) follow by applying the counting argument
from Example~\ref{ex:computing-two-paths} to the equally colored
edges $\{1,2\}$ and $\{1,3\}$; (M2) is automatic since
all vertex colors are distinct.
Thus, $2$-path regularity depends on the chosen ordering
of the vertex-color classes.
\end{example}

\begin{example}[A CER graph that is not symmetric CER]\label{exp:not_nice}
We now give an explicit example of a colored graph that is CER but not symmetric CER.
The colored complete graph on six vertices and its associated RCON space are shown
below. All vertices have the same color, and there are four edge colors.

\begin{center}
\begin{minipage}[c]{0.35\textwidth}
    \centering
    \begin{tikzpicture}[-, thick,
        main node/.style={circle, draw=graphpurple, font=\sffamily\bfseries,
            minimum size=0.6cm, inner sep=0pt},
        a edge/.style={draw=graphgreen},
        b edge/.style={draw=graphbrown, dashed},
        c edge/.style={draw=graphyellow, densely dotted},
        d edge/.style={draw=graphblue, densely dashdotted}]

        \node[main node] (1) at (0, 2) {\color{graphpurple}1};
        \node[main node] (2) at (1.7, 2) {\color{graphpurple}2};
        \node[main node] (3) at (2.5, 0.5) {\color{graphpurple}3};
        \node[main node] (4) at (1.7, -1) {\color{graphpurple}4};
        \node[main node] (5) at (0, -1) {\color{graphpurple}5};
        \node[main node] (6) at (-0.8, 0.5) {\color{graphpurple}6};

        \path
            (1) edge[a edge] (2)
            (2) edge[b edge] (3)
            (3) edge[a edge] (4)
            (4) edge[b edge] (5)
            (5) edge[a edge] (6)
            (6) edge[b edge] (1)
            (1) edge[d edge] (4)
            (2) edge[d edge] (5)
            (3) edge[d edge] (6)
            (1) edge[c edge] (3)
            (1) edge[c edge] (5)
            (2) edge[c edge] (4)
            (2) edge[c edge] (6)
            (3) edge[c edge] (5)
            (4) edge[c edge] (6);
    \end{tikzpicture}

\end{minipage}
\hfill
\begin{minipage}[c]{0.62\textwidth}
    \centering
    \[
    \mathcal{Z}_{G}^{\mathcal{C}} = \left\{
    \begin{bmatrix}
        \textcolor{graphpurple}{\alpha}& \textcolor{graphgreen}{a} & \textcolor{graphyellow}{c} & \textcolor{graphblue}{d} & \textcolor{graphyellow}{c} & \textcolor{graphbrown}{b} \\
        \textcolor{graphgreen}{a} & \textcolor{graphpurple}{\alpha}& \textcolor{graphbrown}{b} & \textcolor{graphyellow}{c} & \textcolor{graphblue}{d} & \textcolor{graphyellow}{c} \\
        \textcolor{graphyellow}{c} & \textcolor{graphbrown}{b} & \textcolor{graphpurple}{\alpha}& \textcolor{graphgreen}{a} & \textcolor{graphyellow}{c} & \textcolor{graphblue}{d} \\
        \textcolor{graphblue}{d} & \textcolor{graphyellow}{c} & \textcolor{graphgreen}{a} & \textcolor{graphpurple}{\alpha}& \textcolor{graphbrown}{b} & \textcolor{graphyellow}{c} \\
        \textcolor{graphyellow}{c} & \textcolor{graphblue}{d} & \textcolor{graphyellow}{c} & \textcolor{graphbrown}{b} & \textcolor{graphpurple}{\alpha}& \textcolor{graphgreen}{a} \\
        \textcolor{graphbrown}{b} & \textcolor{graphyellow}{c} & \textcolor{graphblue}{d} & \textcolor{graphyellow}{c} & \textcolor{graphgreen}{a} & \textcolor{graphpurple}{\alpha}
    \end{bmatrix}\colon
    \substack{
    \textcolor{graphpurple}{\alpha},
    \textcolor{graphgreen}{a},
    \textcolor{graphbrown}{b},
    \textcolor{graphyellow}{c},
    \textcolor{graphblue}{d}\in\mathbb{R}}
    \right\}.
    \]
\end{minipage}
\end{center}

The coloring is induced by the permutation group
\[
\Gamma=\left\langle (1\,3\,5)(2\,4\,6),\ (1\,2)(3\,6)(4\,5)\right\rangle.
\]
To see that (M2) fails, note that
\[
m_{1\to 3}(\textcolor{graphgreen}{a},\textcolor{graphbrown}{b})=1,
\qquad
m_{3\to 1}(\textcolor{graphgreen}{a},\textcolor{graphbrown}{b})=0.
\]
A detailed verification is given in
Appendix~\ref{subsec:three-levels}.
\end{example}

Perhaps the most important subclass of CER graphs consists of decomposable graphs whose colorings arise from permutation subgroups; we study them in the next section.

\subsection{Decomposable RCOP models}\label{sec:decomp}
We say that a permutation subgroup $\Gamma$ on $V=\{1,\ldots,p\}$ is \textbf{transitive} if for any $i,j\in V$ there exists $\sigma\in\Gamma$ such that $\sigma(i)=j$.
Furthermore, a transitive subgroup $\Gamma$ is \textbf{generously transitive} if for any distinct $i,j\in V$ there exists $\sigma\in\Gamma$ such that $\sigma(i)=j$ and $\sigma(j)=i$.

More generally, for an arbitrary (not necessarily transitive) subgroup $\Gamma$ with orbit decomposition
$V=\cup_{i\in [r]}  V_i$, we say that $\Gamma$ is \textbf{orbitwise generously transitive} if for every orbit $V_i$
the induced action $\Gamma|_{V_i}$ is generously transitive on $V_i$. We also note that cyclic subgroups are generally not orbitwise generously transitive. 

\begin{remark}\label{rem:trivial-group}
For $|V|\geq 2$, the trivial group $\Gamma=\{\mathrm{id}\}$ has singleton orbits, so the induced action on every orbit is generously transitive. Hence, $\Gamma$ is orbitwise generously transitive. On the other hand, $\Gamma$ is not generously transitive on $V$, since the identity permutation cannot interchange two distinct vertices.
\end{remark}

We begin with basic invariance properties of the $2$-path function for colorings induced by permutation subgroups.

\begin{lemma}\label{lem:GG}
Assume that $\Gamma\subset\mathrm{Aut}(G)$ and consider the corresponding colored graph $(G,\mathcal{C})$. Fix any ordering of vertex colors $(V_{\eta_1},\ldots,V_{\eta_r})$ for some permutation $\eta\in\mathfrak{S}_r$, and let $m$ be the $2$-path function computed with respect to this ordering. 
    \begin{itemize}
        \item[(i)] For any $v,w\in V$ and $\sigma\in\Gamma$,
        \[
      m_{v\to w} = m_{\sigma(v)\to \sigma(w)}.  
        \]
        \item[(ii)] If $\Gamma$ is orbitwise generously transitive or cyclic, then 
        \[
        c(v)=c(w) \implies  m_{v\to w}\big|_{F\times F} = m_{w\to v}\big|_{F\times F}. 
        \]
    \end{itemize}
\end{lemma}

We are now ready to state the main result relating RCOP models to our new
class of CER graphs. Its implications are summarized in
Section~\ref{sec:model-inclusions}.

\begin{propo}\label{prop:RCOP}
Assume that $\Gamma\subset\mathrm{Aut}(G)$ and consider the corresponding colored graph $(G,\mathcal{C})$. Then the following assertions hold:
\begin{itemize}
\item[(i)] $(G,\mathcal{C})$ is $2$-path regular with respect to any ordering of the vertex color classes.
\item[(ii)] If $G$ is decomposable, then a cpeo can be constructed from any peo of $G$ as follows. Let $(v_{\sigma_1},\ldots,v_{\sigma_p})$ be a peo of $G$, set $j_1=1$ and $\eta_1=c(v_{\sigma_1})$. Given $\eta_1,\ldots,\eta_{k-1}$ for some $2\le k\le r$, let $j_k$ be the smallest index such that $v_{\sigma_{j_k}}\in V\setminus\bigcup_{i=1}^{k-1}V_{\eta_i}$, and put $\eta_k=c(v_{\sigma_{j_k}})$. Then $(V_{\eta_1},\ldots,V_{\eta_r})$ is a cpeo.
\item[(iii)] If $G$ is decomposable, then $(G,\mathcal{C})$ is a CER graph with respect to any cpeo. 
\item[(iv)] If $G$ is decomposable and $\Gamma$ is orbitwise generously transitive or cyclic, then $(G,\mathcal{C})$ is a symmetric CER graph with respect to any cpeo. 
\end{itemize}
\end{propo}

\subsection{Position of CER models among colored Gaussian graphical models}
\label{sec:model-inclusions}

We now place CER and symmetric CER models within the hierarchy
of colored Gaussian graphical models.
Figure~\ref{fig:model-class-inclusions} summarizes these relations,
which we explain below.

Each point in the diagram represents a matrix space
$\mathcal Z_G^{\mathcal C}$.
The classes in this subsection range over all finite colored
graphs $(G,\mathcal C)$ and hence over varying dimensions $p$.
Thus, the diagram describes inclusions between families of
models rather than subsets of one fixed ambient vector space.

\begin{figure}[htbp]
    \centering
    \includegraphics[width=\textwidth]{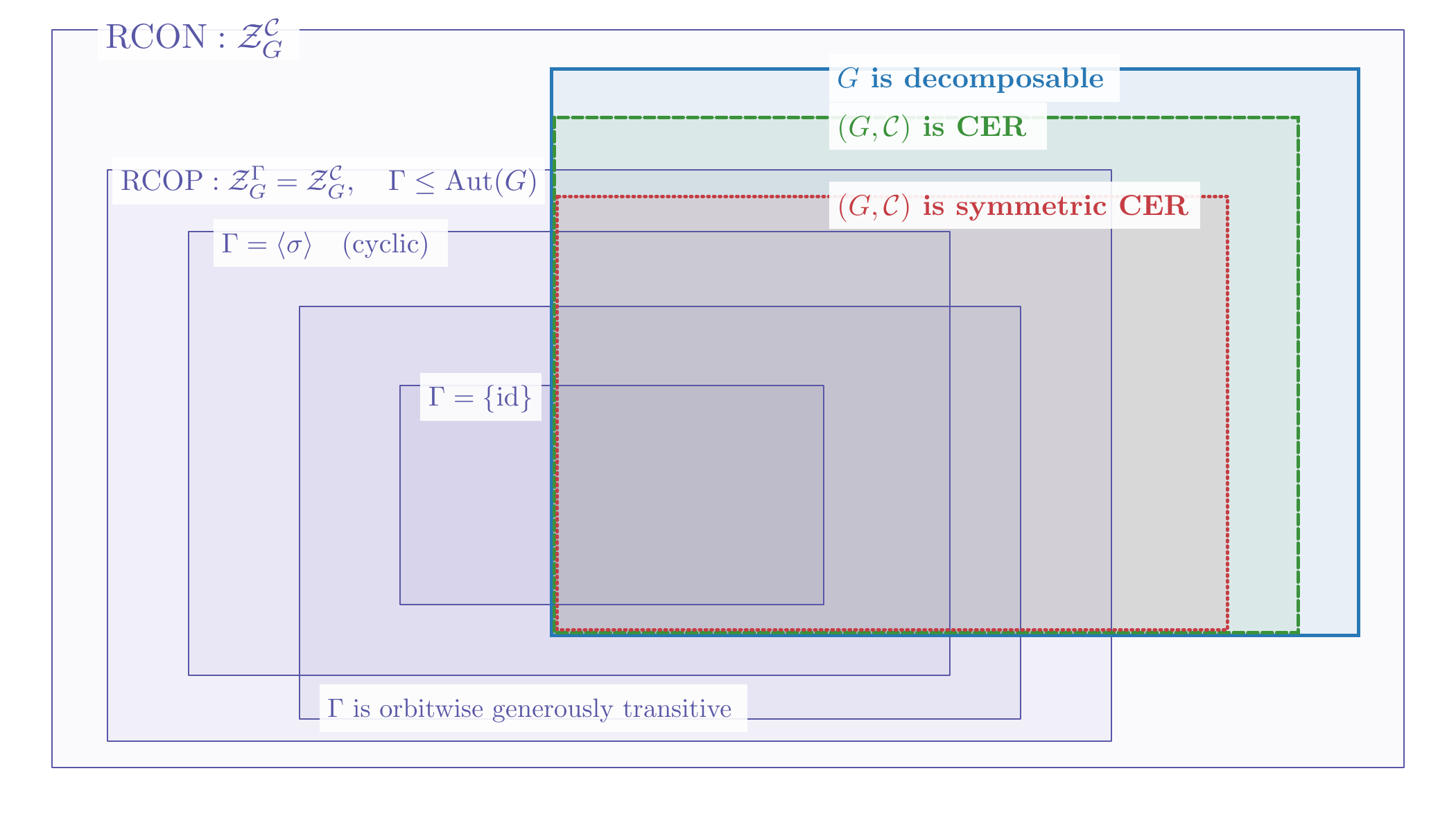}
    \caption{
    Inclusions among colored Gaussian graphical models, extending
    \citet[Figure~10]{HL08} to CER and symmetric CER models.
    Solid blue, dashed green, and dotted red contours indicate,
    respectively, models with decomposable underlying graphs,
    CER models, and symmetric CER models.
    Coincident portions express the intersection equalities
    \eqref{eq:rcop-cer}--\eqref{eq:special-rcop-scer}.
    The diagram is schematic: areas and distances have no quantitative
    meaning.
    }
    \label{fig:model-class-inclusions}
\end{figure}

The basic containments shown in the diagram follow directly from the
definitions. Every RCOP model is an RCON model. Every symmetric CER model
is CER, and every CER model has a decomposable underlying graph, since a
cpeo is in particular a peo. Neither decomposability nor the CER property
requires a model to be RCOP.

Proposition~\ref{prop:RCOP}(iii), together with the reverse implication contained
in the definition of a CER graph, gives
\begin{equation}
\label{eq:rcop-cer}
\begin{gathered}
\text{For RCOP models:}\\[3pt]
\text{the underlying graph is decomposable}
\quad\Longleftrightarrow\quad
\text{the model is CER}.
\end{gathered}
\end{equation}
Thus, within the RCOP class, the CER property imposes no additional
restriction beyond decomposability.

Under the additional group-theoretic assumption of
Proposition~\ref{prop:RCOP}(iv), this equivalence strengthens to
\begin{equation}
\label{eq:special-rcop-scer}
\begin{gathered}
\text{For RCOP models whose coloring is generated by}\\[-1pt]
\text{a cyclic or orbitwise generously transitive group:}\\[3pt]
\text{decomposable}
\quad\Longleftrightarrow\quad
\text{CER}
\quad\Longleftrightarrow\quad
\text{symmetric CER}.
\end{gathered}
\end{equation}
Here, if the same RCOP model admits several group representations, it is
enough that one representing subgroup be cyclic or orbitwise generously
transitive.

This conclusion includes the classical uncolored case.
Indeed, the trivial group $\Gamma=\{\mathrm{id}\}$ is cyclic and, by
Remark~\ref{rem:trivial-group}, orbitwise generously transitive.
Consequently, every uncolored decomposable model is symmetric CER.

The equivalences \eqref{eq:rcop-cer} and
\eqref{eq:special-rcop-scer} explain the coincident portions of
the colored contours in Figure~\ref{fig:model-class-inclusions}.

Having located CER models within the RCON--RCOP hierarchy, we next examine
the algebraic structure of their diagonal blocks.

\subsection{Intersection algebra of diagonal blocks}\label{sec:intersection}

In this section we offer a new view on the properties of the diagonal blocks of matrix color spaces $\coloredZ$ defined in Section \ref{sec:matrix_subspaces}.
Recall that $c(v,w) = k \iff \{v,w\}\in E_k$.

Let \((G,\mathcal{C})\) be a CER graph.
Fix $i\in [r]$ and let $n_i=|V_i|$. Let $\mathcal{C}_i$ denote the coloring of $G[V_i]$ induced by $\mathcal{C}$.  Consider the corresponding diagonal block $\mathcal{Z}_i:=\mathcal{Z}_{G[V_i]}^{\mathcal{C}_i}$.
Let $F_i$ denote the set of extended-edge colors between vertices of color $i$, recall \eqref{eq:Fi}.  
For each $k\in F_i$, define the matrix $A_k^{(i)}\in\{0,1\}^{n_i\times n_i}$ by
\begin{equation}\label{eq:diagonal-color-matrices}
(A_k^{(i)})_{vw} =\begin{cases}
1, & \text{if }c(v,w) = k\\
0, & \text{otherwise},
\end{cases}
\end{equation}
for $v,w\in V_i$. The set $\{A_k^{(i)}\colon k\in F_i\}$ forms a basis of $\mathcal{Z}_i$.

Under (M1), counting the intermediate vertices gives
$[A_k^{(i)} A_h^{(i)} + A_h^{(i)} A_k^{(i)}]_{v,w}
= m_{v\leftrightarrow w}(k,h)$ for $k,h\in F_i$
(see Proposition~\ref{prop:T_diag_properties}, part~3, for details). Moreover, the function $\{v,w\}\mapsto m_{v\leftrightarrow w}(k,h)$ is constant on extended-edge color classes. If $c(v,w)=c$, define $\eta_{kh}^c$ by
\begin{equation}\label{eq:intersection-numbers}
m_{v\leftrightarrow w}(k,h)=2\eta_{kh}^c.
\end{equation}
In this way, we obtain for $k,h\in F_i$,
\[
\frac12(A_k^{(i)} A_h^{(i)} + A_h^{(i)} A_k^{(i)})= \sum_{c\in F_i} \eta_{kh}^c A_c^{(i)}. 
\]
The $\eta_{kh}^c$ are called the \textbf{intersection numbers}. In general, they are nonnegative half-integers; under (M2), however, $m_{v\leftrightarrow w}(k,h)=2m_{v\to w}(k,h)$, and hence $\eta_{kh}^c=m_{v\to w}(k,h)\in\mathbb{N}\cup\{0\}$.
If $G[V_i]$ is a complete graph, then $\mathcal{Z}_i=\mathrm{span}\{A_k^{(i)}\colon k\in F_i\}$ is a Jordan algebra scheme in the sense of \citet{JordanScheme0,JordanScheme1}.
Jordan schemes also play a role in the characterization of colored Gaussian
models with binomial vanishing ideals obtained by \citet{BDM26}.

Under both (M1) and (M2), we have
$A_k^{(i)} A_h^{(i)}=A_h^{(i)} A_k^{(i)}$
(see Proposition~\ref{prop:T_diag_properties}, part~4, for details).
Consequently,
\begin{equation}\label{eq:intersection-product}
A_k^{(i)} A_h^{(i)}=\sum_{c\in F_i}\eta_{kh}^c A_c^{(i)},
\qquad k,h\in F_i.
\end{equation}
When $G[V_i]$ is complete, the color matrices are the adjacency
matrices of a symmetric association scheme, and $\mathcal Z_i$
is its Bose--Mesner algebra \citep{BM59}. More generally,
restricting the color matrices to any connected component of
$G[V_i]$ yields a Bose--Mesner algebra.

For the computations below, the key consequence of
\eqref{eq:intersection-product} is that $\mathcal Z_i$ is a
commutative algebra of symmetric matrices. Hence it has a basis
of mutually orthogonal projection matrices
$\{c_\alpha^{(i)}\}_{\alpha\in[d_i]}$, where $d_i=|F_i|$, satisfying
\[
\sum_{\alpha\in[d_i]}c_\alpha^{(i)}=I_{n_i},
\qquad
c_\alpha^{(i)}c_\beta^{(i)}
=\delta_{\alpha\beta}c_\alpha^{(i)}.
\]
Every $x\in\mathcal Z_i$ therefore has the spectral decomposition
\[
x=\sum_{\alpha\in[d_i]}\lambda_\alpha c_\alpha^{(i)},
\qquad \lambda\in\mathbb R^{d_i}.
\]

These projections and their ranks are used to compute the
structure constants and determinant factors in the
normalizing-constant formula. Given the intersection matrices,
the method in Section~\ref{sec:random_method} computes the
coefficients of these projections in the color-matrix basis
and their ranks in $O(d_i^3)$ operations, rather than the
$O(n_i^3)$ cost of direct diagonalization of the $n_i\times n_i$ block.

\section{General block matrix spaces admitting explicit gamma-like integrals}\label{sec:bcspace}

The goal of this section is to obtain an explicit formula for the gamma-like integral \eqref{eq:general-gamma-integral}. For this purpose, we introduce Block-Cholesky spaces, a class of block-structured linear subspaces of $\mathrm{Sym}(p)$ that generalizes an earlier construction \citep{IshiEntropy}. Their defining conditions allow us to establish a generalized Cholesky decomposition and derive the convergence criterion and product formula in Theorem~\ref{thm:integral}.

\begin{definition}\label{def:I-space}
Let $p\in\N$ and fix an ordered partition $(n_i)_{i\in[r]}\in \N^r$ of $p=\sum_{i\in[r]} n_i$. For a symmetric matrix \(x\in \mathrm{Sym}(p)\), we consider its block decomposition with respect to this partition as
\[
x =
\begin{bmatrix}
X_{11} & X_{12} & \cdots & X_{1r} \\
X_{21} & X_{22} & \cdots & X_{2r} \\
\vdots & \vdots & \ddots & \vdots \\
X_{r1} & X_{r2} & \cdots & X_{rr}
\end{bmatrix}, \,\, \mbox{with $X_{kh}\in\mathbb{R}^{n_k\times n_h}$ and $X_{hk} = X_{kh}^\top$ for all $1\le k\le h\le r$.}
\]

Given this block decomposition of $x\in \mathrm{Sym}(p)$, we define block lower triangular matrix $\tri(x)$ and block diagonal matrix $\bdiag(x)$ by
\[
\tri(x)=
\begin{bmatrix}
X_{11} & 0 & \cdots & 0 \\
X_{21} & X_{22} & \cdots & 0 \\
\vdots & \vdots & \ddots & \vdots \\
X_{r1} & X_{r2} & \cdots & X_{rr}
\end{bmatrix},\qquad 
\bdiag(x)=
\begin{bmatrix}
X_{11} & 0 & \cdots & 0 \\
0 & X_{22} & \cdots & 0 \\
\vdots & \vdots & \ddots & \vdots \\
0 & 0 & \cdots & X_{rr}
\end{bmatrix},
\]
respectively.

Let $\mathcal{Z}$ be a linear subspace of $\mathrm{Sym}(p)$, structured according to the block decomposition induced by partition $(n_i)_{i\in[r]}$. 
For $i \in[r]$, we define the subspaces
 \begin{align*}
 M_i(\mathcal{Z}) &= \{x \in \mathcal{Z}\colon X_{kh} = 0 \mbox{ unless }(k,h) = (i,i)\},\\
 L_i(\mathcal{Z}) & = \{x \in \mathcal{Z}\colon X_{kh} = 0 \mbox{ unless } k > h = i \mbox{ or } h > k = i \}.
\end{align*}
We say that $\mathcal{Z}$ is a \textbf{Block-Cholesky space (\BC-space)} if the following conditions hold:
\begin{enumerate}
\item[(Z0)] $I_p \in \mathcal{Z}$ and
$\mathcal{Z} 
= \Bigl( \bigoplus_{i\in[r]} M_i(\mathcal{Z}) \Bigr) \oplus 
  \Bigl( \bigoplus_{i\in[r-1]} L_i(\mathcal{Z}) \Bigr)$,
    \item[(Z1)] $\tri(x)\tri(x)^\top\in\mathcal{Z}$ for all $x\in \mathcal{Z}$.
\end{enumerate}
If, in addition, 
\begin{enumerate}
    \item[(Z2)] $\bdiag(x) \bdiag(y)\in\mathcal{Z}$ for all $x,y\in\mathcal{Z}$,
\end{enumerate}
then $\mathcal{Z}$ is called a \textbf{Diagonally Commutative Block-Cholesky space (\cBC-space)}.
\end{definition}

We refer to the chosen ordered block decomposition satisfying (Z0) and (Z1) as a \textbf{BC~structure} on $\mathcal{Z}$, and call the number $r$ of blocks its \textbf{rank}. Throughout, \BC-spaces are considered with a fixed BC structure. 

\begin{figure}[h]
\centering
\begin{tikzpicture}[x=0.78cm,y=0.78cm,font=\small]
    \fill[black!5]  (0,0) rectangle (9.20,9.20);
    \fill[black!10] (1.75,0) rectangle (9.20,7.45);
    \fill[white]    (3.50,0) rectangle (9.20,5.70);
    \fill[black!15] (6.00,0) rectangle (9.20,3.20);
    \fill[white]    (7.75,0) rectangle (9.20,1.45);

    \draw (0,0) rectangle (9.20,9.20);
    \draw (1.75,0) -- (1.75,7.45) -- (9.20,7.45);
    \draw (3.50,0) -- (3.50,5.70) -- (9.20,5.70);
    \draw (6.00,0) -- (6.00,3.20) -- (9.20,3.20);
    \draw (7.75,0) -- (7.75,1.45) -- (9.20,1.45);

    \path[fill=white,draw=black]
        (0,9.20) -- (1.55,9.20) -- (1.55,8.05)
        .. controls (1.55,7.80) and (1.40,7.65) .. (1.15,7.65)
        -- (0,7.65) -- cycle;
    \path[fill=white,draw=black]
        (1.75,7.45) -- (3.30,7.45) -- (3.30,6.30)
        .. controls (3.30,6.05) and (3.15,5.90) .. (2.90,5.90)
        -- (1.75,5.90) -- cycle;
    \path[fill=white,draw=black]
        (6.00,3.20) -- (7.55,3.20) -- (7.55,2.05)
        .. controls (7.55,1.80) and (7.40,1.65) .. (7.15,1.65)
        -- (6.00,1.65) -- cycle;

    \node[anchor=east] at (-0.35,4.60) {$\mathcal Z={}$};
    \node at (0.675,8.525) {$M_1(\mathcal Z)$};
    \node at (5.300,8.325) {$L_1(\mathcal Z)$};
    \node at (2.450,6.775) {$M_2(\mathcal Z)$};
    \node at (6.200,6.575) {$L_2(\mathcal Z)$};
    \node at (4.750,4.650) {$\ddots$};
    \node[font=\scriptsize,scale=0.75] at (6.700,2.525) {$M_{r-1}(\mathcal Z)$};
    \node[font=\scriptsize] at (8.300,2.325) {$L_{r-1}(\mathcal Z)$};
    \node at (8.475,0.725) {$M_r(\mathcal Z)$};
\end{tikzpicture}
\caption{Schematic support decomposition in condition (Z0). Each $L_i(\mathcal Z)$ is shown as one connected shaded layer and is labeled only once. The small gaps between successive diagonal regions are only a visual representation of the connection between the two symmetric arms of each $L_i(\mathcal Z)$.}
\label{fig:Mi-Li-decomposition}
\end{figure}

The roles of these conditions can be summarized as follows.
\begin{itemize}
    \item[(Z0)] The direct-sum decomposition is represented schematically in Figure~\ref{fig:Mi-Li-decomposition}. It means that every $x\in\mathcal Z$ decomposes uniquely into diagonal components from $M_i(\mathcal Z)$ and off-diagonal components from $L_i(\mathcal Z)$. Thus, the linear constraints defining $\mathcal Z$ may act within an individual component, but do not couple distinct $M_i(\mathcal Z)$- and $L_j(\mathcal Z)$-components. The additional requirement $I_p\in\mathcal Z$, which is not depicted in the figure, implies, together with this decomposition, that each $M_i(\mathcal Z)$ contains the identity matrix on its $i$th diagonal block. This separation also provides the recursive structure used in the generalized Cholesky decomposition: the first layer $M_1(\mathcal Z)\oplus L_1(\mathcal Z)$ can be handled first, after which the same argument is inductively applied to the remaining $r-1$ block levels.
    \item[(Z1)] This is the closure property needed for the generalized Cholesky decomposition developed in Section~\ref{sec:gen_cholesky}. It ensures that the Gram product of the block lower-triangular part remains in $\mathcal Z$, so that successive block elimination can be carried out without leaving the structured space.
    \item[(Z2)] This condition implies that the diagonal parts commute, explaining the terminology; see Corollary~\ref{lemma:commutativity} below.
\end{itemize}
\begin{corol} \label{lemma:commutativity}
Assume that $\mathcal{Z}$ satisfies (Z2). Then, for $x, y \in \mathcal{Z}$, one has 
\[
\bdiag(x) \cdot \bdiag(y) = \bdiag(y) \cdot \bdiag(x).
\]
\end{corol}

\begin{theorem}\label{thm:ZGCisNice}
Let $\mathcal{C}=(\mathcal{V},\mathcal{E})$ be a coloring
of $G=(V,E)$.
After a suitable ordering of the vertex-color classes and
corresponding relabeling of the vertices, the following hold
with respect to the induced block decomposition:
\begin{itemize}
    \item[(i)]
    If $(G,\mathcal{C})$ is a CER graph, then
    $\coloredZ$ is a \BC-space.
    \item[(ii)]
    If $(G,\mathcal{C})$ is a symmetric CER graph, then
    $\coloredZ$ is a \cBC-space.
\end{itemize}
\end{theorem}

In view of Proposition \ref{prop:RCOP}, we immediately obtain the following result. 
\begin{corol}
Assume that $G$ is decomposable and that $\Gamma$ is a subgroup of
$\mathrm{Aut}(G)$. Then there exists a cpeo. Fix any such cpeo and
relabel the vertices accordingly, as in
Theorem~\ref{thm:ZGCisNice}. Then the following statements hold:
\begin{enumerate}
    \item[(i)] $\mathcal{Z}_G^\Gamma$ is a \BC-space.
    \item[(ii)] If, in addition, $\Gamma$ is cyclic or orbitwise generously transitive, then $\mathcal{Z}_G^\Gamma$ is a \cBC-space.
\end{enumerate}
\end{corol}

\begin{example}[The running example: a \cBC-space without a colored-graph representation]\label{ex:notgraph} There are \BC- and \cBC-spaces, which cannot be written in the form of $\mathcal{Z}_G^{\mathcal{C}}$ for any colored graph $(G,\mathcal{C})$. 
One checks that the following space satisfies (Z0), (Z1) and (Z2),
\[
\mathcal{Z} =
\left\{
\left[
\begin{array}{cc|cc}
 a & 0 & b & c \\ 
0 & a & -c & b \\
\hline
b & -c & d & e \\
c & b & e & d 
\end{array}
\right]\colon 
a, b,c,d,e \in \R\right\},
\]
with $r=2$, $n_1 = n_2 = 2$ and $p=n_1+n_2=4$. The space \(\mathcal Z\) cannot be represented by a colored graph, because colored graph spaces impose only zero constraints or equalities between matrix entries, whereas \(\mathcal Z\) has a nontrivial signed constraint \(x_{14}=-x_{23}\).
We use this \cBC-space as a running example throughout the remainder of
this section: we compute its structure constants in
Example~\ref{ex:notgraph-structure-constants}, describe its generalized
Cholesky parametrization in Example~\ref{ex:notgraph-cholesky}, and
evaluate its gamma-like integral in Example~\ref{ex:notgraph-integral}.
\end{example}

\begin{example}[A \BC-space that is not a \cBC-space]\label{ex:Z1}
Recall the linear subspace $\Zsp= \Zsp_{G}^{\mathcal{C}}$
    from Example \ref{exp:not_nice}. 
It satisfies (Z0) and (Z1), but not (Z2).
Here we have $r = 1$ and $n_1 = p = 6$.
\end{example}

\subsection{Structure constants of \texorpdfstring{$\mathcal{Z}$}{Z}}\label{sec:structure_constants}

Assume that $\mathcal{Z}$ is a \BC-space. On $\mathcal{Z}$ we consider the Lebesgue measure $\dd x$ normalized with respect to the trace inner product.
Recall the gamma-like integral \eqref{eq:general-gamma-integral},
\[
\int_{\mathcal{Z} \cap \mathrm{Sym}^+(p)}(\det x)^s  e^{- \tr(Ax) } \,\dd x,
\]
where $s\in\R$ and $A\in\mathrm{Sym}^+(p)$. Our formulas (see Theorem \ref{thm:integral}) hold on any $\BC$-space and they depend on objects that we collectively call structure constants. Obtaining these structure constants in general is a difficult task. However, if $\Zsp=\coloredZ$ is a \cBC-space (i.e. $(G,\mathcal{C})$ is a symmetric CER graph), then we offer an efficient method for calculating all necessary ingredients. 

 \begin{remark}\label{rem:Leb} 
 Every $x\in\mathcal{Z}$ can be written  uniquely as 
\[
x = \sum_{k} \alpha_k J_k,
\]
where $(J_k)_k$ is an orthonormal basis of $\mathcal{Z}$. Then we put
\[
\dd x = \prod_k \dd \alpha_k,
\]
where each $\dd\alpha_k$ denotes the one-dimensional Lebesgue measure on $\R$.
In the case of an uncolored graph $\mathcal{Z}=\mathcal{Z}_G$, this implies that, $x=(k_{ij})_{i,j\in V}$,
\[
\dd x = 2^{|E|/2} \prod_{v\in V} \dd k_{vv}\prod_{\{i,j\}\in E} \dd k_{ij},
\]
which differs by the multiplicative constant $2^{|E|/2}$ from the usual normalization \citep{A-KM05,ULR18,MML23}.

Consider now a colored space $\mathcal{Z}=\coloredZ$ with a single vertex color and a single edge color. Without loss of generality, assume that $\{1,2\}\in E$. Then 
\[
(J_1,J_2) = \left(\frac{1}{\sqrt{p}} I_p, \frac{1}{\sqrt{2|E|}}A\right),
\]
where $A$ is the adjacency matrix of $G=(V,E)$, is an orthonormal basis of $\coloredZ$. Then, 
\[
x = ( k_{ij})_{i,j\in V} = k_{1,1}\sqrt{p} J_1 + k_{1,2} \sqrt{2|E|} J_2, 
\]
which implies that 
\[
\dd x = \sqrt{2 p |E|} \dd k_{1,1} \dd k_{1,2}. 
\]
We note, however, that the ultimate goal is to compute ratios of the normalizing constants, causing these normalizations to eventually cancel out.
\end{remark}

Then each $M_i(\mathcal{Z})$ is a subalgebra of the Euclidean Jordan algebra 
 $\mathrm{Sym}(p)$. Indeed, (Z1) implies that $M_i(\mathcal{Z})$ is closed under the Jordan product $x\circ y=(xy+yx)/2$. 
Let $d_i$ be the rank of $M_i(\mathcal{Z})$ as a Jordan algebra,
 and $\{ c^{(i)}_{\alpha} \}_{\alpha\in[d_i]}$ be a Jordan frame of $M_i(\mathcal{Z})$,
 that is, a maximal family of mutually orthogonal projection matrices in $M_i(\mathcal{Z})$, i.e.,
 \[
 \sum_{\alpha\in[d_i]} c_\alpha^{(i)}= e^{(i)}\qquad\mbox{and}\qquad c_\alpha^{(i)}c_\beta^{(i)}=\delta_{\alpha\beta} c_\alpha^{(i)},
 \]
 where $e^{(i)}$ is the identity element of $M_i(\mathcal{Z})$: the $(i,i)$-block component of $e^{(i)}$ is the identity matrix $I_{n_i}$,
 and the other block components are zero.

 \begin{remark}
In Section~\ref{sec:intersection}, the Jordan frame elements $c_\alpha^{(i)}$
were defined as matrices in $\mathcal{Z}_i\subset \mathrm{Sym}(n_i)$. Throughout this section,
to simplify the notation for $p\times p$ matrices,
we identify $c_\alpha^{(i)}$ with its canonical embedding into
$\mathrm{Sym}(p)$: the $p \times p$ matrix whose $(i,i)$-block is
the original $n_i \times n_i$ idempotent and all other blocks are zero.
\end{remark}
 
Let $\mu_{i,\alpha}$ be the rank of the projection matrix
 $c^{(i)}_\alpha \in \mathrm{Sym}(p)$.
Put $\tilde{e}^{(i)}_\alpha :=  c^{(i)}_\alpha/\sqrt{\mu_{i,\alpha}}$
 so that $\| \tilde{e}^{(i)}_\alpha \|_F = 1$.
For each $i$, fix an arbitrary ordering $c_1^{(i)},\ldots,c_{d_i}^{(i)}$ of the Jordan frame. The relation $\beta>\alpha$ below refers only to this ordering and is unrelated to the cpeo or to any ordering of graph colors.
Putting $c^{(i)}_{>\alpha} := \sum_{\beta > \alpha} c^{(i)}_\beta$,
 we take an orthonormal basis $\{ \tilde{f}^{(i)}_\gamma \}_{\gamma \in J_{i,\alpha}}$
 of a linear space
 \begin{align}\label{eq:defhh}
    \hh_{i,\alpha} := c^{(i)}_{>\alpha} M_i(\mathcal{Z}) c^{(i)}_\alpha 
 \oplus L_i(\mathcal{Z}) c^{(i)}_\alpha \subset \R^{p \times p} 
  \end{align} 
 with respect to the trace inner product,
 where $J_{i,\alpha}$ is an index set, which can be empty.

 Under (Z2), by commutativity, we have $ c^{(i)}_{>\alpha} M_i(\mathcal{Z}) c^{(i)}_\alpha =  c^{(i)}_{>\alpha} c^{(i)}_\alpha  M_i(\mathcal{Z}) =\{0\}$, which implies that under (Z2) we have
 \begin{align}\label{eq:defhh_simple}
    \hh_{i,\alpha} = L_i(\mathcal{Z}) c^{(i)}_\alpha.
  \end{align} 
 
Let $m_{i,\alpha} := | J_{i,\alpha} | = \dim \hh_{i,\alpha}$. 
For $A \in \mathrm{Sym}(p)$,
 define 
 $\lambda_{i,\alpha}(A)
 := \tr(A \tilde{e}^{(i)}_{\alpha} \tilde{e}^{(i)}_{\alpha})$.
If $m_{i,\alpha} > 0$, 
 define the row vector $v_{i,\alpha}(A) \in \R^{1 \times m_{i,\alpha} }$
 and $\psi_{i,\alpha}(A) \in \mathrm{Sym}(m_{i,\alpha})$ by
 $(v_{i,\alpha}(A))_{1\gamma}
:= \tr (A \tilde{e}^{(i)}_{\alpha}  \mytransp{(\tilde{f}^{(i)}_\gamma)} )$
and
 $\psi_{i,\alpha}(A)_{\gamma \gamma'}
:= \tr\, (A\tilde{f}^{(i)}_{\gamma} \mytransp{ (\tilde{f}^{(i)}_{\gamma'}) })$
 respectively for $\gamma, \gamma' \in J_{i,\alpha}$.
Finally, 
 define $\phi_{i,\alpha}(A) \in \mathrm{Sym}(1 + m_{i,\alpha})$ by
\[
\phi_{i,\alpha}(A) := \begin{cases} 
\begin{bmatrix} \lambda_{i,\alpha}(A) & v_{i,\alpha}(A) \\
 \mytransp{v_{i,\alpha}(A)} & \psi_{i,\alpha}(A) \end{bmatrix}  & \mbox{if } m_{i,\alpha} >0,\\
\lambda_{i,\alpha}(A) & \mbox{if } m_{i,\alpha} = 0. 
\end{cases}
\]

\begin{example}[Structure constants for the running example]\label{ex:notgraph-structure-constants}
 
For the running \cBC-space introduced in Example~\ref{ex:notgraph}, we
have $r=2$, $d_1=1$, $d_2=2$, $\mu_{1,1}=2$,
$\mu_{2,1}=\mu_{2,2}=1$ and
\begin{gather*}
\tilde{e}^{(1)}_1 = \frac{1}{\sqrt{2}} \begin{bmatrix} 
1 & 0 & 0 & 0 \\ 0 & 1 & 0 & 0 \\ 0 & 0 & 0 & 0 \\ 0 & 0 & 0 & 0
\end{bmatrix}, \quad
\tilde{e}^{(2)}_1 = \frac{1}{2} \begin{bmatrix} 
0 & 0 & 0 & 0 \\ 0 & 0 & 0 & 0 \\ 0 & 0 & 1 & 1 \\ 0 & 0 & 1 & 1
\end{bmatrix}, \quad
\tilde{e}^{(2)}_2 = \frac{1}{2} \begin{bmatrix} 
0 & 0 & 0 & 0 \\ 0 & 0 & 0 & 0 \\ 0 & 0 & 1 & -1 \\ 0 & 0 & -1 & 1
\end{bmatrix}
\end{gather*}
Moreover,
\[
\tilde{f}^{(1)}_1 = \frac{1}{\sqrt{2}} \begin{bmatrix} 
0 & 0 & 0 & 0 \\ 
0 & 0 & 0 & 0 \\ 
1 & 0 & 0 & 0 \\ 
0 & 1 & 0 & 0
\end{bmatrix}, \quad
\tilde{f}^{(1)}_2 = \frac{1}{\sqrt{2}} \begin{bmatrix} 
0 & 0 & 0 & 0 \\ 
0 & 0 & 0 & 0 \\ 
0 & -1 & 0 & 0 \\ 
1 & 0 & 0 & 0
\end{bmatrix}
\]
forms an orthonormal basis of $\hh_{1,1}=L_1(\mathcal{Z})c_1^{(1)}$, while $\hh_{2,\alpha}=\{0\}$ so that $J_{2,\alpha}=\emptyset$ for $\alpha=1,2$.
Thus, 
$m_{1,1}=2$ and $m_{2,1}=m_{2,2}=0$.

With $A=(a_{ij})_{i,j}\in\mathrm{Sym}(4)$, we have 
\[
\phi_{1,1}(A) =   \frac12 \begin{bmatrix}
  a_{11}+a_{22} & a_{13}+a_{24} &  a_{14}-a_{23} \\ 
  a_{13}+a_{24} & a_{33}+a_{44} & 0 \\
  a_{14}-a_{23} & 0 & a_{33}+a_{44}
\end{bmatrix},\quad \begin{array}{ll} \phi_{2,1}(A) = \begin{bmatrix}\frac{a_{33}+a_{44}}{2}+a_{34}\end{bmatrix},\vspace{1mm}\\ \vspace{1mm} \phi_{2,2}(A) = \begin{bmatrix}\frac{a_{33}+a_{44}}{2}-a_{34}\end{bmatrix}.\end{array}
\]
\end{example}

\begin{definition}
    The \textbf{structure constants} of \BC-space $\mathcal{Z}$ are $(d_i, (\mu_{i,\alpha}, m_{i,\alpha})_{\alpha\in[d_i]})_{i\in[r]}$, where $d_i$ is the Jordan rank of $M_i(\mathcal{Z})$, $\mu_{i,\alpha}$ is the rank of the projection $c_{\alpha}^{(i)}$ and $m_{i,\alpha}$ is the dimension of the linear space $\hh_{i,\alpha}$ defined in \eqref{eq:defhh}.
\end{definition}

\subsection{Generalized Cholesky decomposition}\label{sec:gen_cholesky}

We now introduce a generalized Cholesky decomposition on $\mathcal{P}=\Zsp\cap\mathrm{Sym}^+(p)$, where $\Zsp$ is a Block-Cholesky space.

Recall that the linear space $\hh_{i,\alpha} \subset \R^{p \times p}$ is
 spanned by $\{ \tilde{f}^{(i)}_\gamma \}_{\gamma \in J_{i,\alpha}}$.
Let 
\[
\hh=\bigoplus_{i\in[r]} \bigoplus_{\alpha\in[d_i]}(\R c^{(i)}_\alpha \oplus \hh_{i,\alpha}),
\]
 and
 $\hh^+$ a subset of $\hh$ consisting of elements $T$ 
 of the form
$T = \sum_{i\in[r]} \sum_{\alpha\in[d_i]} (t_{i,\alpha} \tilde{e}^{(i)}_\alpha + \sum_{\gamma \in J_{i,\alpha}} \tau_{i,\gamma} \tilde{f}^{(i)}_\gamma )$
 with $t_{i,\alpha}>0$ and $\tau_{i,\gamma} \in \R$.

\begin{remark}\label{rem:Tri+}
Let \(\tri(\mathcal{Z})^{+}\) denote the cone of block lower triangular
matrices whose diagonal blocks are positive definite.
\begin{enumerate}
\item[(i)]
In general, if at least one of the Jordan algebras \(M_i(\mathcal Z)\)
is noncommutative, then \(\hh^+\) and \(\tri(\mathcal Z)^+\) are
distinct cones in \(\mathbb R^{p\times p}\). The cone \(\hh^+\) is
adapted to the Peirce decomposition associated with the Jordan frames
\(\{c_\alpha^{(i)}\}\). In an appropriate orthonormal basis of
\(\mathbb R^p\), its elements become lower triangular, whereas the
elements of \(\tri(\mathcal Z)^+\) are in general only block lower
triangular.

\item[(ii)]
Under (Z2), each \(M_i(\mathcal Z)\) is a commutative Euclidean Jordan
algebra, and the two cones \(\hh^+\) and \(\tri(\mathcal Z)^+\) coincide.

\item[(iii)]
When \(\mathcal Z=\mathcal Z_G^{\mathcal C}\) arises from a CER graph,
Theorem~\ref{thm:group} gives a necessary and sufficient condition for
\(\tri(\mathcal Z)^+\) to form a group.
\end{enumerate}
\end{remark}

\begin{example}[Generalized Cholesky parametrization for the running example]\label{ex:notgraph-cholesky}
For the running \cBC-space introduced in Example~\ref{ex:notgraph}, we have
\[
    \sum_{i\in[r]} \sum_{\alpha\in[d_i]} (t_{i,\alpha} \tilde{e}^{(i)}_\alpha + \sum_{\gamma \in J_{i,\alpha}} \tau_{i,\gamma} \tilde{f}^{(i)}_\gamma )=\begin{bmatrix}
        \frac{t_{1,1}}{\sqrt{2}} & 0 & 0 & 0 \\
        0 & \frac{t_{1,1}}{\sqrt{2}} & 0 & 0 \\
        \frac{\tau_{1,1}}{\sqrt{2}} & - \frac{\tau_{1,2}}{\sqrt{2}} & \frac{t_{2,1}+t_{2,2}}{2} &\frac{t_{2,1}-t_{2,2}}{2} \\
     \frac{\tau_{1,2}}{\sqrt{2}} & \frac{\tau_{1,1}}{\sqrt{2}} & \frac{t_{2,1}-t_{2,2}}{2} &\frac{t_{2,1}+t_{2,2}}{2} 
    \end{bmatrix}\in\hh.
\]
\end{example}

\begin{theorem}\label{thm:new_Cholesky}
    Assume (Z0) and (Z1).  
Let $x \in \mathcal{Z} \cap \mathrm{Sym}^+(p)$.

Then, there exists a unique $T \in \hh^+$ for which $x = T \mytransp{T}$.
Writing this element as
\[
T = \sum_{i\in[r]} \sum_{\alpha\in[d_i]} (t_{i,\alpha} \tilde{e}^{(i)}_\alpha + \sum_{\gamma \in J_{i,\alpha}} \tau_{i,\gamma} \tilde{f}^{(i)}_\gamma ),
\]
we have 
\begin{enumerate}
    \item[(i)]    
\begin{align*}
  \det x
 = \prod_{i\in[r]} \prod_{\alpha\in[d_i]} 
 \Bigl( \frac{t_{i,\alpha}^2} {\mu_{i,\alpha} } \Bigr)^{\mu_{i,\alpha}};
    \end{align*}
    \item[(ii)] for $A\in\mathrm{Sym}(p)$, 
\begin{equation*}
 \tr(xA) =
\sum_{i\in[r]} \sum_{\alpha\in[d_i]} 
\begin{bmatrix} t_{i,\alpha} & (\tau_{\alpha}^{(i)})^\top \end{bmatrix}
\phi_{i,\alpha}(A)
\begin{bmatrix} t_{i,\alpha} \\ \tau_{\alpha}^{(i)} \end{bmatrix},
\end{equation*}
where $\tau_{\alpha}^{(i)}=(\tau_{i,\gamma})_{\gamma\in J_{i,\alpha}}\in \R^{m_{i,\alpha}}$.
If $m_{i,\alpha}=0$, we interpret the corresponding term in the sum as $\lambda_{i,\alpha}(A) t_{i,\alpha}^2$. 
 \item[(iii)] \begin{equation*} 
\begin{aligned}
\dd x = \prod_{i\in[r]} \prod_{\alpha\in[d_i]}
 \Bigl\{ 
2^{1 + m_{i,\alpha}/2} \mu_{i,\alpha}^{-(1+m_{i,\alpha})/2} t_{i,\alpha}^{1 + m_{i,\alpha}}
 \, \dd t_{i,\alpha} \dd\tau^{(i)}_\alpha
 \Bigr\},
\end{aligned}
\end{equation*}
where $\dd\tau^{(i)}_\alpha$ denotes the Lebesgue measure on $\R^{m_{i,\alpha}}$ (see Remark \ref{rem:Leb}). 
\end{enumerate}
\end{theorem}

\subsection{Integral formula}\label{sec:integral}

Theorem~\ref{thm:integral} is our main analytic result. It gives a
convergence criterion and product formula for gamma-like integrals over
\BC-spaces. Appendix~\ref{subsec:random_method_example} provides a fully
worked five-vertex example in which all the structure constants and
determinant factors are computed explicitly and then substituted into the
integral formula below.

We normalize the Lebesgue measure $\dd x$ on the vector space $\mathcal{Z}$  with respect to the trace inner product (recall Remark \ref{rem:Leb}).

\begin{theorem} \label{thm:integral}
Assume (Z0) and (Z1). 

For $s\in\mathbb{R}$ and $A \in \mathrm{Sym}^+(p)$, 
 the integral 
\[
I(s,A)
:=\int_{\mathcal{Z} \cap \mathrm{Sym}^+(p)}(\det x)^s  e^{- \tr(Ax) } \,\dd x
\]
converges
if and only if 
\[
s > \max\left\{-\frac{2+m_{i,\alpha}}{2\,\mu_{i,\alpha}}\colon i\in[r], \alpha\in[d_i]\right\}.
\]
In that case,
\[
I(s,A)=e^{-p_{\mathcal{Z}}s-q_{\mathcal{Z}}}\!
  \prod_{\substack{i\in[r]\\ \alpha\in[d_i]}}\!
  \Bigl\{\Gamma\Bigl(\mu_{i,\alpha}s + \frac{m_{i,\alpha}+2}{2}\Bigr)
 \Bigl( \frac{\det \phi_{i,\alpha}(A) }{ \det \psi_{i,\alpha}(A) }
 \Bigr)^{-\mu_{i,\alpha}s - \frac{m_{i,\alpha}+2}{2}}
 (\det \psi_{i,\alpha}(A) )^{-\frac12}
\Bigr\},
\]
where 
\begin{align*}
p_{\mathcal{Z}} 
&=\sum_{i\in[r]} \sum_{\alpha\in[d_i]} \mu_{i,\alpha} \log \mu_{i,\alpha},\\
q_{\mathcal{Z}}
&= \frac12 \sum_{i\in[r]} \sum_{\alpha\in[d_i]} 
 \{ (\log \mu_{i,\alpha} - \log(2\pi)) m_{i,\alpha} +\log \mu_{i,\alpha} \},  
\end{align*} 
 and we interpret $\det \psi_{i,\alpha}(A) = 1$ if $m_{i,\alpha} = 0$.
\end{theorem}

\begin{example}[Integral for the running example]\label{ex:notgraph-integral}
We now evaluate the integral for the \cBC-space introduced in
Example~\ref{ex:notgraph}. Let $A=(a_{ij})_{i,j}\in\mathrm{Sym}^+(4)$.
The matrices computed in Example~\ref{ex:notgraph-structure-constants}
satisfy
\[
\det\psi_{1,1}(A)=\frac{(a_{33}+a_{44})^2}{4},
\]
\[
\frac{\det\phi_{1,1}(A)}{\det\psi_{1,1}(A)}
=\frac{
(a_{11}+a_{22})(a_{33}+a_{44})
-(a_{13}+a_{24})^2-(a_{14}-a_{23})^2
}{2(a_{33}+a_{44})},
\]
and
\[
\phi_{2,1}(A)\phi_{2,2}(A)
=\left(\frac{a_{33}+a_{44}}{2}+a_{34}\right)
 \left(\frac{a_{33}+a_{44}}{2}-a_{34}\right)
=\frac{(a_{33}+a_{44})^2}{4}-a_{34}^2.
\]
Moreover, the structure constants from that example give
\[
p_{\mathcal Z}=2\log 2,
\qquad
q_{\mathcal Z}=\frac12\log 2-\log\pi,
\]
and the convergence condition in Theorem~\ref{thm:integral} reduces to
$s>-1$. Consequently,
\[
\begin{aligned}
I(s,A)
&=\int_{\mathcal Z\cap\mathrm{Sym}^+(4)}
(\det x)^s e^{-\tr(Ax)}\,\dd x\\
&=\frac{\pi\,2^{1/2-2s}}{a_{33}+a_{44}}\,
\Gamma(2s+2)\Gamma(s+1)^2\\
&\qquad\times
\left(
\frac{
(a_{11}+a_{22})(a_{33}+a_{44})
-(a_{13}+a_{24})^2-(a_{14}-a_{23})^2
}{2(a_{33}+a_{44})}
\right)^{-2s-2}\\
&\qquad\times
\left(\frac{(a_{33}+a_{44})^2}{4}-a_{34}^2\right)^{-s-1}.
\end{aligned}
\]
For example, for $A=I_4$ we obtain
\[
I(s,I_4)
=\pi\,2^{-2s-1/2}\,
\Gamma(2s+2)\Gamma(s+1)^2,
\qquad s>-1.
\]
\end{example}

\section{Computation of structure constants and determinant factors}\label{sec:computation}

In this section we present two algorithms that together efficiently compute all
quantities needed to evaluate the integral from
Theorem~\ref{thm:integral}. Algorithm~\ref{alg:random_method} computes
the Jordan frames and their ranks, whereas
Algorithm~\ref{alg:gram_cholesky} computes the remaining structure
constants and determinant factors. Appendix~\ref{subsec:random_method_example}
illustrates this complete computational pipeline on a concrete colored graph.
Throughout we assume (Z2), so that $\mathcal Z$ is a \cBC-space.

Since our main application is the computation of normalizing constants
in the Bayesian model associated with a symmetric CER graph
$(G,\mathcal C)$, the construction of Jordan frames in
Subsection~\ref{sec:random_method} is presented for
$\mathcal Z=\mathcal Z_G^{\mathcal C}$, that is, for matrix spaces in
\cBC\,$\cap$\,RCON, and exploits the corresponding graphical
intersection numbers. The Gram--Cholesky method in
Subsection~\ref{sec:gram_cholesky}, however, applies to an arbitrary
\cBC-space once its Jordan frames and bases of the spaces
$L_i(\mathcal Z)$ are available, even if $\mathcal Z$ does not arise
from a colored graph.

Under (Z2), Corollary~\ref{lemma:commutativity} implies that
$M_i(\mathcal{Z})$ is a commutative matrix algebra for each $i\in[r]$.
Hence $M_i(\mathcal{Z})$ is spanned by a Jordan frame
$\{c^{(i)}_\alpha\}_{\alpha\in[d_i]}$.
We recall that
\[
  \mu_{i,\alpha} = \mathrm{rank}(c^{(i)}_\alpha),\qquad
  c^{(i)}_\alpha c^{(i)}_\beta = \delta_{\alpha\beta} c^{(i)}_\alpha.
\]

\subsection{Computation of Jordan frames}
\label{sec:random_method}

Throughout this subsection, we work in the graphical setting
$\mathcal Z=\mathcal Z_G^{\mathcal C}$ described above.
Recall that $d_i=\dim\Zsp_i=|F_i|\le n_i$, where the inequality
follows from simultaneous diagonalizability under (Z2).
Given the intersection matrices,
Algorithm~\ref{alg:random_method} computes the coefficients of the
Jordan-frame projections in the color-matrix bases, together with
their ranks, in $O(\sum_{i\in[r]}d_i^3)$ operations.
Its randomized step, called the Random Method, computes the common
eigenbasis $V^{(i)}$ of the intersection matrices
$\{L_n^{(i)}\}_{n\in[d_i]}$.
Compared with direct diagonalization of the $n_i\times n_i$
blocks, whose total cost is $O(\sum_{i\in[r]}n_i^3)$,
this can be substantially more efficient when $d_i\ll n_i$.

Recall the matrices $A_k^{(i)}$, $k\in F_i$, defined in \eqref{eq:diagonal-color-matrices}. Choose an enumeration $F_i=\{k_1,\ldots,k_{d_i}\}$ and, for $m\in[d_i]$, set $A_m^{(i)}:=A_{k_m}^{(i)}$. Using the same enumeration, reindex the intersection numbers defined in \eqref{eq:intersection-numbers} by setting
\[
\eta_{nm}^{(i),\ell}:=\eta_{k_nk_m}^{k_\ell},
\qquad n,m,\ell\in[d_i].
\]
By \eqref{eq:intersection-product}, these coefficients satisfy
\begin{equation}\label{eq:local-intersection-product}
A_n^{(i)}A_m^{(i)}=\sum_{\ell\in[d_i]}\eta_{nm}^{(i),\ell}A_\ell^{(i)}.
\end{equation}

For completeness, we recall the standard argument for the
simultaneous diagonalization of the intersection matrices.

Define $d_i\times d_i$ matrices $P^{(i)}$ and $Q^{(i)}$ (called the first and second eigenmatrix in association-scheme convention) by 
\begin{align}\label{eq:APQ}
A_m^{(i)} = \sum_{\alpha\in[d_i]} P_{m\alpha}^{(i)}\,c_\alpha^{(i)}\qquad\mbox{and}\qquad c_\alpha^{(i)} = \frac1{n_i}\sum_{m\in[d_i]} Q_{\alpha m}^{(i)} A_m^{(i)}.
\end{align}
Our aim is to find matrix $Q^{(i)}$.

Writing
\[
A_m^{(i)} = \frac1{n_i}  \sum_{\alpha\in[d_i]} \sum_{n\in[d_i]} P_{m\alpha}^{(i)} Q_{\alpha n}^{(i)} A_n^{(i)},
\]
we obtain
\[
\sum_{\alpha=1}^{d_i} P_{m\alpha}^{(i)} Q_{\alpha n}^{(i)} = n_i \,\delta_{mn},
\]
which is equivalent to  $Q^{(i)} P^{(i)} = P^{(i)} Q^{(i)} = n_i \,I_{d_i}$. 
Moreover, using the product formula \eqref{eq:local-intersection-product}, we have
\begin{align*}
A_n^{(i)} c_\alpha^{(i)} &=  A_n^{(i)} \left(\frac1{n_i}\sum_{m\in[d_i]} Q_{\alpha m}^{(i)} A_m^{(i)}\right) = 
\frac1{n_i}\sum_{m\in[d_i]} Q_{\alpha m}^{(i)} \sum_{\ell\in[d_i]} \eta_{nm}^{(i),\ell} A_\ell^{(i)} \\
&=\frac1{n_i} \sum_{\ell\in[d_i]} \left(\sum_{m\in[d_i]} \eta_{nm}^{(i),\ell} Q_{\alpha m}^{(i)}\right) A_\ell^{(i)}.
\end{align*}
On the other hand,
\[
A_n^{(i)} c_\alpha^{(i)} = \sum_{\beta\in[d_i]} P_{n\beta}^{(i)}c_\beta^{(i)} c_\alpha^{(i)} =  P_{n\alpha}^{(i)}c_\alpha^{(i)} = \frac1{n_i}\sum_{\ell\in[d_i]} P_{n\alpha}^{(i)} Q_{\alpha \ell}^{(i)} A_\ell^{(i)}, 
\]
which by comparing the $A_\ell^{(i)}$-coefficients implies that 
\begin{align}\label{eq:eigen}
\sum_{m\in[d_i]} \eta_{nm}^{(i),\ell} Q_{\alpha m}^{(i)} = P_{n\alpha} ^{(i)}Q_{\alpha\ell}^{(i)},\qquad \alpha,\ell,n\in[d_i].
\end{align}
For $n\in [d_i]$ we define the $d_i\times d_i$ (intersection) matrix $L^{(i)}_n$ by
\[
(L^{(i)}_{n})_{\ell m} = \eta_{nm}^{(i),\ell}.
\]
Note that $\Zsp_i=\mathrm{span}\{A_m^{(i)}\colon m\in[d_i]\}$, and $L^{(i)}_n$ is the matrix of the linear map
$\Zsp_i\ni X\mapsto A_n^{(i)}X\in\Zsp_i$ with respect to the ordered basis
$(A_1^{(i)},\ldots,A_{d_i}^{(i)})$.

Then, \eqref{eq:eigen} can be rewritten as 
\[
L^{(i)}_n\, q_\alpha^{(i)} = P_{n\alpha}^{(i)}\, q_\alpha^{(i)},\qquad \alpha,n\in[d_i],
\]
where $q_\alpha^{(i)} = (Q_{\alpha1}^{(i)},\ldots,Q_{\alpha d_i}^{(i)})^\top$. This implies that each $q_\alpha^{(i)}$ is an eigenvector of any matrix $L^{(i)}_n$ and the corresponding eigenvalue is $P_{n\alpha}^{(i)}$. Thus, $(L^{(i)}_n)_{n=1}^{d_i}$ are simultaneously diagonalizable, i.e., there exists a matrix $V^{(i)}$ such that 
\[
(V^{(i)})^{-1} L^{(i)}_n V^{(i)} = \mathrm{diag}(P_{n1}^{(i)},\ldots,P_{nd_i}^{(i)}). 
\]
The common eigenbasis $V^{(i)}$ can be obtained in $O(d_i^3)$-time by diagonalizing a generic linear combination
\[
L^{(i)}(x) = \sum_{n\in[d_i]} x_n L^{(i)}_n,\qquad x\in \R^{d_i}. 
\]
The eigenvalues of 
$L^{(i)}(x)$ are the linear forms 
\[
\lambda_\alpha(x) = \sum_{n\in[d_i]} x_n P_{n\alpha}^{(i)}.
\]
To find the common eigenbasis $V^{(i)}$, we use the following randomized step.

\noindent\textbf{Random Method: computation of \(V^{(i)}\)}\\
Since \(Q^{(i)}P^{(i)}=n_iI_{d_i}\), the columns of \(P^{(i)}\) are pairwise distinct. Hence, $L^{(i)}(x)$ has a repeated eigenvalue only when
$x$ lies in a finite union of proper hyperplanes,
a set of Lebesgue measure zero. Thus, sampling \(x\) from any absolutely continuous distribution on \(\R^{d_i}\) yields a simple spectrum with probability one. Since every \(L_n^{(i)}\) commutes with \(L^{(i)}(x)\), diagonalizing \(L^{(i)}(x)\) yields a common eigenbasis \(V^{(i)}\) for \(\{L_n^{(i)}\}_{n\in[d_i]}\) in $O(d_i^3)$ time.

\noindent\textbf{Recovery of the Jordan frame and its ranks}\\
Let \(v_\alpha^{(i)}\) denote the \(\alpha\)-th column of \(V^{(i)}\). One could recover \(P^{(i)}\) by computing
\[
(V^{(i)})^{-1}L_n^{(i)}V^{(i)}
=
\operatorname{diag}\bigl(P_{n1}^{(i)},\ldots,P_{nd_i}^{(i)}\bigr)
\]
for every \(n\in[d_i]\), but these \(d_i\) similarity transformations would require \(O(d_i^4)\) operations. Instead, as detailed in Algorithm~\ref{alg:random_method}, all entries of \(P^{(i)}\) can be recovered directly from
\[
L_n^{(i)}v_\alpha^{(i)}
=
P_{n\alpha}^{(i)}v_\alpha^{(i)}
\]
in \(O(d_i^3)\) operations.

We could recover \(Q^{(i)}\) in \(O(d_i^3)\) time by inverting \(P^{(i)}\), namely \(Q^{(i)}=n_i(P^{(i)})^{-1}\). However, this can be done in \(O(d_i^2)\) time. Since the
columns of \(V^{(i)}\) are proportional to the vectors
\(q_\alpha^{(i)}\), there exist nonzero scalars
\(\gamma_1^{(i)},\ldots,\gamma_{d_i}^{(i)}\) such that
\[
(V^{(i)})^\top P^{(i)}
=n_i\,\operatorname{diag}\bigl(\gamma_1^{(i)},\ldots,
\gamma_{d_i}^{(i)}\bigr).
\]
Using \(Q^{(i)}P^{(i)}=n_iI_{d_i}\), we then obtain
\[
Q^{(i)}
=\operatorname{diag}\bigl((\gamma_1^{(i)})^{-1},\ldots,
(\gamma_{d_i}^{(i)})^{-1}\bigr)(V^{(i)})^\top,
\]
also in \(O(d_i^2)\) operations. By the second equation of
\eqref{eq:APQ}, \(Q^{(i)}/n_i\) gives the coefficients of the Jordan frame
\(\{c_\alpha^{(i)}\}_{\alpha\in[d_i]}\) in the color-matrix basis.
Thus, given the intersection matrices, this coefficient representation
is obtained in \(O(d_i^3)\) operations.

\smallskip
The ranks $\mu_{i,\alpha}=\operatorname{rank}(c_\alpha^{(i)})$ can also be recovered directly from $Q^{(i)}$. Since each $c_\alpha^{(i)}$ is a projection matrix, its rank equals its trace. Let $s_i\in[d_i]$ be the index of the loop color, so that $A_{s_i}^{(i)}=I_{n_i}$, while all the remaining matrices $A_m^{(i)}$, $m\neq s_i$, have zero diagonal. Hence
\[
\tr(A_m^{(i)})=n_i\delta_{m,s_i}.
\]
Using the second identity in \eqref{eq:APQ}, we therefore obtain
\[
\mu_{i,\alpha}
=\operatorname{rank}(c_\alpha^{(i)})
=\tr(c_\alpha^{(i)})
=\frac1{n_i}\sum_{m\in[d_i]}Q_{\alpha m}^{(i)}\tr(A_m^{(i)})
=Q_{\alpha s_i}^{(i)}.
\]

\clearpage
\begin{algorithm}[H]
\caption{Computation of Jordan frames using the Random Method}
\label{alg:random_method}
\SetAlgoLined
\KwIn{Block sizes \(n_i\), color sets \(F_i\), color matrices
\(\{A_k^{(i)}\}_{k\in F_i}\), and intersection numbers
\(\{\eta_{kh}^{c}\}_{k,h,c\in F_i}\), for \(i\in[r]\)}
\KwOut{Jordan frames
\(\{c_\alpha^{(i)}\}_{\alpha\in[d_i]}\) and their ranks
\(\{\mu_{i,\alpha}\}_{\alpha\in[d_i]}\), for \(i\in[r]\)}

\For{\(i=1\) \KwTo \(r\)}{
    Choose an enumeration \(F_i=\{k_1,\ldots,k_{d_i}\}\)\;
    Set \(A_m^{(i)}\gets A_{k_m}^{(i)}\) for \(m\in[d_i]\)\;
    Let \(s_i\in[d_i]\) satisfy \(A_{s_i}^{(i)}=I_{n_i}\)\;
    Construct \(L_n^{(i)}\), \(n\in[d_i]\), by
    \((L_n^{(i)})_{\ell m}\gets\eta_{k_nk_m}^{k_\ell}\),
    \(\ell,m\in[d_i]\)\;

    \Repeat{\(\lambda_1,\ldots,\lambda_{d_i}\) are pairwise distinct}{
        Sample \(x\sim\mathrm{Unif}([0,1]^{d_i})\)\;
        \(L^{(i)}(x) \gets \sum_{n\in[d_i]}x_nL_n^{(i)}\)\;
        Compute eigenpairs
        \(\{(\lambda_\alpha,v_\alpha^{(i)})\}_{\alpha\in[d_i]}\)
        of \(L^{(i)}(x)\)\;
    }(\tcp*[f]{With probability one, this loop is executed exactly once.})

    \For{\(\alpha=1\) \KwTo \(d_i\)}{
        Choose
        \(\ell_\alpha\in\operatorname*{arg\,max}_{\ell\in[d_i]}
        |(v_\alpha^{(i)})_\ell|\)\;
        \For{\(n=1\) \KwTo \(d_i\)}{
            \(P_{n\alpha}^{(i)} \gets
            \displaystyle
            \frac{\sum_{m\in[d_i]}(L_n^{(i)})_{\ell_\alpha m}
            (v_\alpha^{(i)})_m}
            {(v_\alpha^{(i)})_{\ell_\alpha}}\)\;
        }
    }

    \For{\(\alpha=1\) \KwTo \(d_i\)}{
        \(\gamma_\alpha^{(i)}\gets
        \displaystyle\frac1{n_i}\sum_{m\in[d_i]}
        (v_\alpha^{(i)})_mP_{m\alpha}^{(i)}\)\;
        \For{\(m=1\) \KwTo \(d_i\)}{
            \(Q_{\alpha m}^{(i)}\gets
            \displaystyle\frac{(v_\alpha^{(i)})_m}
            {\gamma_\alpha^{(i)}}\)\;
        }
    }
    \For{\(\alpha=1\) \KwTo \(d_i\)}{
        \(c_\alpha^{(i)}\gets
        \displaystyle\frac1{n_i}\sum_{m\in[d_i]}
        Q_{\alpha m}^{(i)}A_m^{(i)}\)\;
        \(\mu_{i,\alpha}\gets Q_{\alpha s_i}^{(i)}\)\;
    }
}
\Return
\(\bigl\{\{(c_\alpha^{(i)},\mu_{i,\alpha})\}_{\alpha\in[d_i]}\bigr\}_{i\in[r]}\)\;
\end{algorithm}

\subsection{Computation of the remaining constants: the Gram--Cholesky method}
\label{sec:gram_cholesky}

The definitions of $v_{i,\alpha}(A)$, $\psi_{i,\alpha}(A)$ and,
consequently, $\phi_{i,\alpha}(A)$ in
Section~\ref{sec:structure_constants} are expressed in terms of an
orthonormal basis
$\{\tilde f_\gamma^{(i)}\}_{\gamma\in J_{i,\alpha}}$
of the space $\hh_{i,\alpha}$.
Such a basis could be constructed explicitly, for example by a QR or
singular-value decomposition. This orthonormalization is, however, not
necessary for computing the determinant quantities appearing in the
normalizing constant.

In this subsection we describe a practical procedure, which we call the
\textbf{Gram--Cholesky method}, for computing the remaining structure
constants $m_{i,\alpha}$ and, for a given $A\in\mathrm{Sym}^+(p)$,
the quantities
\[
  \det\psi_{i,\alpha}(A)
  \quad\mbox{and}\quad
  \frac{\det\phi_{i,\alpha}(A)}
       {\det\psi_{i,\alpha}(A)}
\]
without explicitly constructing these orthonormal bases.
The method accounts for non-orthonormality through Gram matrices
and uses the factorization
$c_\alpha^{(i)}=W_{i,\alpha}W_{i,\alpha}^{\top}$,
with $W_{i,\alpha}\in\R^{p\times\mu_{i,\alpha}}$,
to work with compressed factors instead of full $p\times p$
matrix products.
The complete procedure is summarized in
Algorithm~\ref{alg:gram_cholesky}.

For numerical implementations, the determinants may be evaluated
through their logarithms; for clarity, we state the formulas below
directly in terms of determinants.

We assume throughout that the Jordan frames
$\{c_\alpha^{(i)}\}_{\alpha\in[d_i]}$ are already known.
Steps~1 and~2 depend only on the space $\mathcal Z$, whereas
Steps~3 and~4 additionally depend on the chosen matrix
$A\in\mathrm{Sym}^+(p)$.

\noindent\textbf{Step 1: Factorization of the projections.}\\
Fix $i\in[r]$ and $\alpha\in[d_i]$.
Since $c_\alpha^{(i)}$ is supported only on its $(i,i)$ block,
compute $\widehat W_{i,\alpha}\in\R^{n_i\times\mu_{i,\alpha}}$
whose columns form an orthonormal basis of the eigenspace of
$(c_\alpha^{(i)})_{ii}$ corresponding to the eigenvalue~$1$.
Let $W_{i,\alpha}\in\R^{p\times\mu_{i,\alpha}}$ be obtained by
placing $\widehat W_{i,\alpha}$ in the rows corresponding to
the $i$th block and filling the remaining rows with zeros.
Then
\[
  W_{i,\alpha}^\top W_{i,\alpha}=I_{\mu_{i,\alpha}},
  \qquad
  c_\alpha^{(i)}=W_{i,\alpha}W_{i,\alpha}^\top.
\]
In computations, it is enough to store the low-dimensional factor
$\widehat W_{i,\alpha}$.

\noindent\textbf{Step 2: Constructing the spaces $\hh_{i,\alpha}$.}\\
Let $q_i:=\dim L_i(\mathcal Z)$, and let
$\{B_1^{(i)},\dots,B_{q_i}^{(i)}\}$ be a basis of
$L_i(\mathcal Z)$.
\par\noindent
By Lemma~\ref{lem:L_i}, each off-diagonal edge-color class has a fixed
earlier endpoint color. Hence, in the graphical setting considered
here, where $\mathcal{Z}=\mathcal{Z}_G^{\mathcal{C}}$, a natural choice of
basis of $L_i(\mathcal{Z})$ is given by the adjacency matrices of the
edge-color classes whose earlier endpoint color is $i$.
Since we are in a \cBC-space, 
\[
  \hh_{i,\alpha} = L_i(\mathcal{Z})\,c_\alpha^{(i)}
  = \mathrm{span}\{B_t^{(i)} c_\alpha^{(i)} \colon t\in[q_i]\}.
\]
To exploit the low rank of $c_\alpha^{(i)}$, we define
\[
  \widehat U_t := B_t^{(i)} W_{i,\alpha}\in\R^{p\times \mu_{i,\alpha}},
  \qquad t\in[q_i].
\]
Then $B_t^{(i)}c_\alpha^{(i)} = \widehat U_t W_{i,\alpha}^\top$, and
\[
  \langle B_t^{(i)}c_\alpha^{(i)}, B_s^{(i)}c_\alpha^{(i)}\rangle_F
  = \tr( B_t^{(i)}c_{\alpha}^{(i)}(B_s^{(i)}c_{\alpha}^{(i)})^\top)
  = \tr(\widehat U_t \widehat U_s^\top).
\]

If all $\widehat U_t$ are zero, then $\hh_{i,\alpha} = \{0\}$ and
we set $m_{i,\alpha}=0$.
Otherwise, we form the Gram matrix
\[
  G_{i,\alpha}\in\R^{q_i\times q_i},\qquad
  (G_{i,\alpha})_{ts}
  = \tr(\widehat U_t \widehat U_s^\top),\quad t,s\in[q_i].
\]
A Cholesky decomposition with complete pivoting of $G_{i,\alpha}$
determines its rank
$m_{i,\alpha}=\mathrm{rank}(G_{i,\alpha})$
and a set of pivot indices
\[
  \mathcal I_{i,\alpha}
  =\{t_1,\ldots,t_{m_{i,\alpha}}\}\subset[q_i].
\]
The selected matrices
$\{B_{t_k}^{(i)}c_\alpha^{(i)}\}_{k=1}^{m_{i,\alpha}}$
form a basis of $\hh_{i,\alpha}$, with positive-definite Gram matrix
\[
  \widetilde G_{i,\alpha}
  =G_{i,\alpha}[\mathcal I_{i,\alpha},\mathcal I_{i,\alpha}].
\]
Since $\widetilde G_{i,\alpha}$ is positive definite, let
$\widetilde G_{i,\alpha}=L_G L_G^\top$ be its Cholesky factorization,
where $L_G$ is lower triangular. We compute at this stage
\[
  \det\widetilde G_{i,\alpha}
  = \prod_{j\in[m_{i,\alpha}]}(L_G)_{jj}^2,
\]
and retain this value for use in Step~4.

\noindent\textbf{Step 3: $A$-weighted Gram matrices and mixed terms.}\\
Fix $A\in\mathrm{Sym}^+(p)$.
For notational convenience, we write
\[
  U_k := B_{t_k}^{(i)} c_\alpha^{(i)}
       = \widehat U_{t_k} W_{i,\alpha}^\top,
  \qquad k\in[m_{i,\alpha}],
\]
and reuse $\widehat U_{t_k}$ for the compressed factors.

The $A$-weighted Gram matrix on the basis $\{U_k\}_{k=1}^{m_{i,\alpha}}$ is
\[
  \Psi_{i,\alpha}(A) \in \R^{m_{i,\alpha}\times m_{i,\alpha}},
  \qquad
  (\Psi_{i,\alpha}(A))_{kl}
  = \tr(A U_k U_l^\top)
  = \tr(A \widehat U_{t_k} \widehat U_{t_l}^\top).
\]
We also define the mixed vector
$V_{i,\alpha}(A)\in\R^{1\times m_{i,\alpha}}$ by
\[
  (V_{i,\alpha}(A))_{1k}
  = \frac{1}{\sqrt{\mu_{i,\alpha}}}\,\tr(A c_\alpha^{(i)} U_k^\top)
  = \frac{1}{\sqrt{\mu_{i,\alpha}}}\,
    \tr(A W_{i,\alpha} \widehat U_{t_k}^\top),
  \qquad k\in[m_{i,\alpha}].
\]
Here $V_{i,\alpha}(A)$ is the coordinate vector of mixed terms in the basis $\{U_k\}$; it differs from the vector $v_{i,\alpha}(A)$ defined in Section~\ref{sec:structure_constants}, which is expressed in the orthonormal basis $\{\tilde f_\gamma^{(i)}\}_{\gamma\in J_{i,\alpha}}$.

Finally, recalling $\tilde e^{(i)}_\alpha = c_\alpha^{(i)}/\sqrt{\mu_{i,\alpha}}$, we have
\[
  \lambda_{i,\alpha}(A)
  = \tr(A \tilde e^{(i)}_\alpha (\tilde e^{(i)}_\alpha)^\top)
  = \frac{1}{\mu_{i,\alpha}}\,\tr(Ac_\alpha^{(i)}).
\]
Since $c_\alpha^{(i)}$ is supported only on block $(i,i)$, this can be computed as
\[
  \lambda_{i,\alpha}(A)
  = \frac{1}{\mu_{i,\alpha}}\,\tr(A_{ii} (c_\alpha^{(i)})_{ii})
  = \frac{1}{\mu_{i,\alpha}}\,\tr(\widehat W_{i,\alpha}^\top A_{ii}
                             \widehat W_{i,\alpha}),
\]
where $A_{ii}$ is the $(i,i)$-block of $A$.

\noindent\textbf{Step 4: Determinant formulas.}\\
If $m_{i,\alpha}=0$, then we set
\[
  \det \psi_{i,\alpha}(A) = 1,
  \qquad
  \frac{\det \phi_{i,\alpha}(A)}{\det \psi_{i,\alpha}(A)}
  = \lambda_{i,\alpha}(A).
\]

Assume now that $m_{i,\alpha}>0$.
Let $\widetilde G_{i,\alpha}$ be as in Step~2, i.e. the Gram matrix of
$\{U_k\}_{k=1}^{m_{i,\alpha}}$ with respect to the Frobenius inner product.
Its determinant, which appears in the denominator below, was computed in
Step~2. Then:
\begin{align*}
  \det \psi_{i,\alpha}(A)
  &= \frac{\det \Psi_{i,\alpha}(A)}
          {\det \widetilde G_{i,\alpha}},\\[0.4em]
  \frac{\det \phi_{i,\alpha}(A)}{\det \psi_{i,\alpha}(A)}
  &= \lambda_{i,\alpha}(A)
     - V_{i,\alpha}(A)
       (\Psi_{i,\alpha}(A))^{-1}
       V_{i,\alpha}(A)^\top.
\end{align*}

In practice one computes $V_{i,\alpha}(A)\Psi_{i,\alpha}(A)^{-1}
V_{i,\alpha}(A)^\top$ and $\det \Psi_{i,\alpha}(A)$ via a Cholesky factorization. Write
$\Psi_{i,\alpha}(A) = L L^\top$, with $L$ lower triangular, and solve
\[
  L\, y = V_{i,\alpha}(A)^\top
\]
for $y$ by forward substitution. Then
\[
  V_{i,\alpha}(A)\Psi_{i,\alpha}(A)^{-1}
  V_{i,\alpha}(A)^\top
  =y^\top y.
\]
The determinant $\det\Psi_{i,\alpha}(A)$ is given by
$\det\Psi_{i,\alpha}(A) = \prod_{j\in[m_{i,\alpha}]} L_{jj}^2$.

\begin{theorem}\label{thm:GC}
The Gram--Cholesky method described above and summarized in
Algorithm~\ref{alg:gram_cholesky} correctly computes the
structure constants $m_{i,\alpha}$ and, for every $A\in\mathrm{Sym}^+(p)$,
the determinants $\det\psi_{i,\alpha}(A)$ and
$\tfrac{\det\phi_{i,\alpha}(A)}{\det\psi_{i,\alpha}(A)}$ when
$\mathcal Z$ is a \cBC-space.
\end{theorem}

\clearpage
\begin{algorithm}[H]
\small
\caption{Gram--Cholesky computation of the remaining constants}
\label{alg:gram_cholesky}
\SetAlgoLined
\KwIn{Bases $\{B_t^{(i)}\}_{t\in[q_i]}$ of $L_i(\mathcal Z)$,
Jordan frames $\{c_\alpha^{(i)}\}_{\alpha\in[d_i]}$ and their ranks
$\{\mu_{i,\alpha}\}_{\alpha\in[d_i]}$, for $i\in[r]$, and
$A\in\mathrm{Sym}^+(p)$}
\KwOut{The constants $m_{i,\alpha}$ and the quantities
$\det\psi_{i,\alpha}(A)$ and
$\det\phi_{i,\alpha}(A)/\det\psi_{i,\alpha}(A)$}

\textbf{Preprocessing (performed once for $\mathcal Z$)}\;
\ForEach{$i\in[r]$ and $\alpha\in[d_i]$}{
    Compute an eigendecomposition of $(c_\alpha^{(i)})_{ii}$ and let
    $\widehat W_{i,\alpha}$ contain its $\mu_{i,\alpha}$ orthonormal
    eigenvectors with eigenvalue~$1$\;
    Let $W_{i,\alpha}$ denote the zero-padded embedding of
    $\widehat W_{i,\alpha}$\;
    \For{$t=1$ \KwTo $q_i$}{
        $\widehat U_t\gets B_t^{(i)}W_{i,\alpha}$\;
    }
    \eIf{$q_i=0$ or all $\widehat U_t$ are zero}{
        $m_{i,\alpha}\gets0$\;
    }{
        Form $G_{i,\alpha}$ with
        $(G_{i,\alpha})_{ts}\gets
        \tr(\widehat U_t\widehat U_s^\top)$\;
        Apply pivoted Cholesky to $G_{i,\alpha}$ to obtain
        $m_{i,\alpha}$ and
        $\mathcal I_{i,\alpha}=\{t_1,\ldots,t_{m_{i,\alpha}}\}$\;
        $\widetilde G_{i,\alpha}\gets
        G_{i,\alpha}[\mathcal I_{i,\alpha},\mathcal I_{i,\alpha}]$\;
        Compute $\widetilde G_{i,\alpha}=L_G L_G^\top$ and set
        $g_{i,\alpha}\gets\prod_{j\in[m_{i,\alpha}]}(L_G)_{jj}^2$\;
    }
    Retain $m_{i,\alpha}$ and $\widehat W_{i,\alpha}$; if
    $m_{i,\alpha}>0$, also retain $g_{i,\alpha}$ and
    $\{\widehat U_{t_k}\}_{k\in[m_{i,\alpha}]}$ for the evaluation phase\;
}

\BlankLine
\textbf{Evaluation for the chosen matrix $A$}\;
\ForEach{$i\in[r]$ and $\alpha\in[d_i]$}{
    $\lambda_{i,\alpha}(A)\gets
    \tr(\widehat W_{i,\alpha}^\top A_{ii}\widehat W_{i,\alpha})
    /\mu_{i,\alpha}$\;
    \eIf{$m_{i,\alpha}=0$}{
        $\det\psi_{i,\alpha}(A)\gets1$\;
        $\det\phi_{i,\alpha}(A)/\det\psi_{i,\alpha}(A)
        \gets\lambda_{i,\alpha}(A)$\;
    }{
        Let $W_{i,\alpha}$ be the zero-padded embedding of the retained
        $\widehat W_{i,\alpha}$\;
        $(\Psi_{i,\alpha}(A))_{kl}\gets
        \tr(A\widehat U_{t_k}\widehat U_{t_l}^{\top})$ for
        $k,l\in[m_{i,\alpha}]$\;
        $(V_{i,\alpha}(A))_{1k}\gets
        \tr(AW_{i,\alpha}\widehat U_{t_k}^{\top})/
        \sqrt{\mu_{i,\alpha}}$ for $k\in[m_{i,\alpha}]$\;
        Compute $\Psi_{i,\alpha}(A)=LL^\top$\;
        Solve $Ly=V_{i,\alpha}(A)^\top$\;
        $\det\psi_{i,\alpha}(A)\gets
        \displaystyle\frac{\prod_{j\in[m_{i,\alpha}]}L_{jj}^2}
        {g_{i,\alpha}}$\;
        $\det\phi_{i,\alpha}(A)/\det\psi_{i,\alpha}(A)
        \gets\lambda_{i,\alpha}(A)-y^\top y$\;
    }
}
\end{algorithm}

\section{Relations to RDAG models}
\label{sec:relations}

In this section we explain how our notions of color perfect elimination ordering (cpeo), $2$-path regularity, CER graphs, and \BC/\cBC-spaces relate to the restricted DAG (RDAG) models of \citet{Seigal23}. 

A directed graph is a pair $\mathcal{G}=(V,\vec{E})$ with $p=|V|<\infty$ and $\vec{E}\subset V\times V$; we write $j\to i$ for $(j,i)\in\vec{E}$ and $j\nrightarrow i$ otherwise. A directed acyclic graph (DAG) has no directed cycles. A coloring of a DAG is a map
\[
c\colon V\cup\vec{E}\to [r+R],
\]
where $r=|c(V)|$ and $R=|c(\vec E)|$. 
According to \citet[Definition~2.1]{Seigal23}, the restricted DAG model $\mathrm{RDAG}_{\mathcal{G}}^c$ on $(\mathcal{G},c)$ consists of positive definite matrices of the form $\Psi=(I_p-\Lambda)^\top\Omega^{-1}(I_p-\Lambda)$, where $\Omega=(\omega_{ij})_{i,j}\in\R^{p\times p}$ is diagonal with positive entries, $\Lambda=(\lambda_{ij})_{i,j}\in\R^{p\times p}$ and
\begin{itemize}
    \item $\lambda_{ij}=0$ unless $j\to i$ in $\mathcal{G}$,
    \item $c(j\to i)=c(l\to k)\implies \lambda_{ij}=\lambda_{kl}$,
    \item $c(i)=c(j)\implies \omega_{ii}=\omega_{jj}$.
\end{itemize}
Following \citet[Definition~2.5]{Seigal23}, a coloring $c$ is compatible if \mbox{$c(V)\cap c(\vec{E})=\emptyset$} and
\[
c(j\to i)=c(l\to k)\implies c(i)=c(k).
\]
Under compatibility, \citet[Proposition~2.7]{Seigal23} shows that
\[
\mathrm{RDAG}_{\mathcal{G}}^c = \{a^\top a \colon a\in A(\mathcal{G},c)\},
\]
where
\[
A(\mathcal{G},c)
= \left\{ a\in\mathrm{GL}_p \colon
\begin{array}{@{}ccl@{}}
i\neq j \ \text{and}\ j\nrightarrow i &\Longrightarrow& a_{ij}=0\\
c(i)=c(j) &\Longrightarrow& a_{ii}=a_{jj}\\
c(j\to i)=c(l\to k) &\Longrightarrow& a_{ij}=a_{kl}
\end{array}
\right\}.
\]

Let $\ch(i) = \{k\in V\colon i\to k\in\mathcal{G}\}$. Given a vertex $i\in V$, let $(\mathcal{G}_i,c)$ be the colored subgraph on $\{i\}\cup\ch(i)$, with edges $i\to k$ for $k\in\ch(i)$ and inherited colors. For an edge $j\to i$, let $(\mathcal{G}_{(j\to i)},c)$ be the colored directed multigraph on $\{i\}\cup (\ch(i)\cap\ch(j))$ having, for each $k\in\ch(i)\cap\ch(j)$, two parallel edges $i\to k$ colored $c(i\to k)$ and $c(j\to k)$. Writing $(\mathcal{G}_1,c_1)\cong(\mathcal{G}_2,c_2)$ for isomorphism of colored graphs, \citet[Theorem~3.4]{Seigal23} states that, under compatibility,
$\mathrm{RDAG}_{\mathcal{G}}^c$ coincides with an RCON model if and only if
\begin{align}\label{eq:i-iii}
\begin{split}
(i)\ & \ \mbox{there is no induced subgraph $i\to k\leftarrow j$ with $i$ and $j$ non-adjacent;}\\
(ii)\ & \ c(i)=c(j)\implies (\mathcal{G}_i,c)\cong (\mathcal{G}_j,c);\\
(iii)\ & \ c(j\to i)=c(l\to k)\implies (\mathcal{G}_{(j\to i)},c)\cong (\mathcal{G}_{(l\to k)},c).
\end{split}
\end{align}
The undirected graph $G=(V,E)$  underlying this RCON model is the skeleton of $\mathcal{G}=(V,\vec{E})$ (so that $\{i,j\}\in E$ if and only if $(i\to j)\in\vec{E}$ or $(j\to i)\in\vec{E}$) and the coloring $\mathcal{C}$ of $G$ is inherited from $(\mathcal{G},c)$. 
We now characterize the colored graphs that admit such an orientation.

\begin{theorem}\label{thm:RDAG}
Let $(G,\mathcal C)$ be a colored graph.
Then $(G,\mathcal C)$ admits an acyclic orientation whose inherited
coloring is compatible and for which (i)--(iii) in \eqref{eq:i-iii}
hold if and only if $(G,\mathcal C)$ is a symmetric CER graph and
$G$ has no edges between vertices of the same color.
\end{theorem}

For RDAG models with compatible coloring, we therefore obtain
the following strict inclusion of models
\[
\mathrm{RDAG}\cap \mathrm{RCON} \subsetneq \mathrm{\cBC}\cap \mathrm{RCON}.
\]
Indeed, note that the RCON space
\[
\Zsp = \left\{ \begin{bmatrix}
    a & b \\ b & a
\end{bmatrix}\colon a,b\in\R \right\}
\]
is a \cBC-space, but its associated RCON model is not an RDAG model.
Indeed, a two-vertex RDAG allowing a nonzero off-diagonal entry has
one directed edge. Distinct vertex colors yield the full positive
definite cone, whereas equal vertex colors force the two diagonal
entries to be unequal whenever the edge parameter is nonzero.

\citet[Proposition~B.2]{Seigal23} also formulates the butterfly criterion, a condition under which the set $A(\mathcal{G},c)$ forms a group.
Concretely, for a directed edge $j\to i$ define the set of vertices
\[
b(j\to i) = \{k\in V \colon j\to k\to i \text{ is a directed path in }\mathcal{G}\}.
\]
The butterfly graph is the subgraph of $\mathcal{G}$ on $\{i,j\}\cup b(j\to i)$ with edges $j\to k$ and $k\to i$ for $k\in b(j\to i)$, inheriting all colors from $(\mathcal{G},c)$. The butterfly criterion asserts that, for a compatible coloring $c$, the set $A(\mathcal{G},c)$ is a group if and only if $\mathcal{G}$ is transitive (i.e.\ $k\to j$ and $j\to i$ imply $k\to i$) and any two directed edges of the same color have isomorphic butterfly graphs. When $A(\mathcal G,c)$ is a group, $\mathrm{RDAG}_{\mathcal G}^c$ is a Gaussian group model in the sense of \citet{GGM}. 

We now state the analogous criterion for $\tri(\Zsp)^+$ (recall Remark \ref{rem:Tri+}), the cone of block lower triangular matrices whose diagonal blocks are positive definite.

\begin{theorem}\label{thm:group}
Let $(G,\mathcal C)$ be a CER graph. Fix a vertex-color ordering from the definition of CER, and relabel the vertex-color classes so that this ordering is $(V_1,\ldots,V_r)$. Let $\Zsp=\coloredZ$, with its block decomposition induced by this ordering.

For $v,w\in V$ and $k,h\in[r+R]$, define
\[
\widetilde m_{v\to w}(k,h)
=\Bigl|\bigl\{u\in V:
c(w)\leq c(u)\leq c(v),\ c(v,u)=k,\ c(u,w)=h
\bigr\}\Bigr|.
\]
This function counts colored $2$-paths whose intermediate vertex belongs to a color class between those of the endpoints, including the endpoint color classes.

Then $\tri(\Zsp)^+$ forms a group under matrix multiplication if and only if both of the following conditions hold:
\begin{itemize}
\item[(M3a)] For all distinct $v,w\in V$ with $v\nsim w$,
\[
c(w)\leq c(v)
\quad\Longrightarrow\quad
\widetilde m_{v\to w}\equiv0.
\]
\item[(M3b)] For all $v,w,v',w'\in V$ with $\{v,w\},\{v',w'\}\in\tilde E$,
\[
\begin{gathered}
c(v,w)=c(v',w')
\ \land\ c(w)\leq c(v)\ \land\ c(w')\leq c(v')
\\
\Longrightarrow\quad
\widetilde m_{v\to w}=\widetilde m_{v'\to w'}.
\end{gathered}
\]
\end{itemize}
Condition (M3a) ensures that multiplication preserves the prescribed zeros, while (M3b) ensures that it preserves the equality constraints imposed by the coloring.
\end{theorem}

\section{Conclusion}\label{sec:conclusion}

In this paper, we introduced a new subclass of RCON models \citep{HL08} defined by Color Elimination-Regular (CER) and symmetric CER graphs. On the algebraic side, these models are supported on Block-Cholesky (\BC) and Diagonally Commutative Block-Cholesky (\cBC) spaces, which provide a natural framework for expressing the Diaconis--Ylvisaker normalizing constants as gamma-type integrals over cones of structured precision matrices. For \BC-spaces, we derived an explicit closed-form formula
for these integrals in terms of structure constants.
For RCON models whose matrix spaces are \cBC-spaces, we developed
an efficient methodology, combining the Random Method for diagonal
blocks with the Gram--Cholesky procedure for off-diagonal interactions,
to compute all ingredients needed for the evaluation of the
normalizing constants.

Our theoretical analysis highlights connections between these undirected colored models and existing structures in algebraic statistics. We demonstrated that the combinatorial conditions defining CER graphs, particularly $2$-path regularity, naturally extend the structural constraints found in RDAG models \citep{Seigal23}, offering a unified perspective on tractability in both directed and undirected colored graphs. Additionally, in the symmetric case, our framework reveals that the diagonal-block spaces, when restricted to connected components of the corresponding graphs $G[V_i]$, form Bose--Mesner algebras, linking statistical inference to the theory of association schemes.

Several directions for future work follow naturally from our results. A first step is to design Markov chain Monte Carlo algorithms whose state space is restricted to CER or symmetric CER graphs, thereby enabling Bayesian structure learning within this tractable subclass. Beyond this, it would be desirable to develop Monte Carlo methods for approximating normalizing constants outside our class (for example for general $2$-path regular graphs), extending the approach of \citet{A-KM05}, as well as to derive analytic approximations in the spirit of \citet{MML23} for more general colored graphical models. Finally, a deeper investigation of the interplay between \BC-/\cBC-spaces, association schemes, and RDAG models may reveal further algebraic structure that can be exploited for efficient Bayesian inference in high-dimensional settings, thereby removing a major bottleneck for this class of models.

\section{Proofs}\label{sec:proofs}

\begin{lemma}\label{lem:edge_regularity_on_diagonal}
Assume that $(G,\mathcal{C})$ satisfies (M1) with respect to the fixed ordering of its vertex-color classes. If $c(v,w)=k\in F_i$, then $c(v)=c(w)=i$.
\end{lemma}
\begin{proof}
Let $i\in [r]$ and $k\in F_i$. Then, there exists $v,w\in V_i$ such that $c(v,w)=k$. We have to show that for any $v',w'\in V$ such that $c(v',w') = k$ one has $c(v') = c(w') = i$.

Consider $2$-paths $v-u-w$ for $u\in V_{\le i}$. The only such path with $u-w$ edge being of color $i$ is for $u=w$ (as $i$ is a loop color). The color of the first edge in the path $v-w-w$, is $k$. So, $m_{v\to w}(k, i) =\Bigl| \{w \} \Bigr|= 1$. Similarly, $m_{v\to w}(i, k) =\Bigl| \{v \} \Bigr|= 1$. Therefore, we have $m_{v\leftrightarrow w}(i, k) = m_{v\to w}(k, i) + m_{v\to w}(i, k) = 2$.

Now, let $\{v',w'\}\in E_k$. We must show that $c(v')=c(w')=i$.
By (M1), we have $m_{v\leftrightarrow w} = m_{v'\leftrightarrow w'}$. In particular, $m_{v'\leftrightarrow w'}(k,i) = m_{v\leftrightarrow w}(k,i) = 2$. Hence, there must be exactly two $2$-paths from $v'$ to $w'$ using colors $k$ and $i$, which gives:
\begin{align*}
 m_{v'\leftrightarrow w'}(k, i) =2.
\end{align*}
Again, since $i$ is a loop color, the only possible choice for intermediate vertex $u$ is $u=w'$ in the first case and $u=v'$ in the second, which forces $c(w') = i$ and $c(v') = i$, respectively, completing the proof of the claim.
\end{proof}

\subsection{Proofs for Section~\ref{sec:newclass}}
\begin{proof}[Proof of Proposition~\ref{prop:regularity-within-color-classes}]
Fix a cpeo from the definition of CER and relabel the vertex-color classes
according to this ordering.

Fix $i\in[r]$. Edge-regularity is immediate, since all vertices of
$G[V_i]$ have the same color. For $k\in F_i\setminus\{i\}$ and $v\in V_i$, let
\[
d_k(v)=|\{u\in V_i\colon c(v,u)=k\}|.
\]
By Lemma~\ref{lem:edge_regularity_on_diagonal}, every edge of color $k$
has both endpoints in $V_i$. Thus, using that $G$ is undirected, we obtain
\[
m_{v\leftrightarrow v}(k,k)=2m_{v\to v}(k,k)=2d_k(v).
\]
All loops at vertices of $V_i$ have color $i$, so (M1) implies that $d_k(v)$
is independent of $v\in V_i$ for every $k\in F_i\setminus\{i\}$.
This proves vertex-regularity of $(G[V_i],\mathcal{C}_i)$.
Summing over the edge colors shows that all vertices of $G[V_i]$ have the
same degree. By Remark~\ref{rem:cliques}, its connected components are
cliques, and hence they all have the same size.

For the final assertion, take $G$ to be the disjoint union
of two copies of $K_{16}$, with a single vertex color.
In the first component, color the edges of the Shrikhande graph
blue; in the second, color the edges of the $4\times4$ rook graph blue.
Color all remaining edges within each component red.
Both blue graphs are strongly regular with parameters $(16,6,2,2)$;
see Appendix~\ref{subsec:strongly-regular-graph-models}
for the definition.
The rook graph contains a $K_4$, whereas the Shrikhande graph
does not \citep[Section~10.6]{srg}.
We verify the $2$-path conditions using products of color adjacency matrices,
whose entries count colored $2$-paths; see Section~\ref{sec:intersection}
for this interpretation.
In either component, the blue and red adjacency matrices $B$ and $R$
satisfy
\[
B^2=6I+2B+2R,\qquad
BR=RB=3B+4R,\qquad
R^2=9I+6B+4R.
\]
Together with multiplication by $I$, these identities show that the colored
$2$-path counts depend only on the color of the endpoint pair and agree
across the two components, giving (M1). The color adjacency matrices
commute, giving (M2); counts between distinct components are zero.
The unique vertex-color ordering is a cpeo because every vertex of $G$
is simplicial. Thus, the coloring is symmetric CER, but its two components
are not isomorphic as colored graphs: a color-preserving isomorphism
would induce an isomorphism between their blue graphs.
\end{proof}

\begin{proof}[Proof of Lemma \ref{lem:GG}]
We have $c(v)=i$ if and only if $v\in V_{\eta_i}$. 
(i) follows from the fact that $\sigma$ is an automorphism of $(G,\mathcal{C})$. Indeed, we have 
\begin{align*}
 &m_{\sigma(v)\to \sigma(w)}(k,h)\\
 &=
\Bigl| \Big\{ u \in V_{\le c(\sigma(v)) \wedge c(\sigma(w))} \colon c(\sigma(v),u)=k\mbox{ and } c(u,\sigma(w))=h\Big\} \Bigr| \\
 &=
\Bigl| \Big\{ \sigma(u') \in V_{\le c(\sigma(v)) \wedge c(\sigma(w))} \colon c(\sigma(v),\sigma(u'))=k\mbox{ and } c(\sigma(u'),\sigma(w))=h\Big\} \Bigr| \\
 &=
\Bigl| \Big\{ u' \in \sigma^{-1}(V_{\le c(v) \wedge c(w)}) \colon c(v,u')=k\mbox{ and } c(u',w)=h\Big\} \Bigr|,
\end{align*}
where we have used the fact that $c(\sigma(v))=c(v)$ and $c(\sigma(v),\sigma(w))=c(v,w)$ for any $v,w\in V$ and $\sigma\in\Gamma$. 
By the first fact and the definition of $V_{\leq \cdot}$, we have 
\begin{align*}
V_{\le c(v) \wedge c(w)} &= \{u\in V\colon c(u)\leq c(v)\wedge c(w)\} = \{u\in V\colon c(\sigma(u))\leq c(v)\wedge c(w)\}  \\
&=  \sigma^{-1}(V_{\le c(v) \wedge c(w)}),
\end{align*}
which completes the proof of (i).

(ii) If $\Gamma$ is orbitwise generously transitive, then (ii) follows
directly from (i). Indeed, for any $v,w$ in the same $\Gamma$-orbit, there exists
$\sigma\in\Gamma$ such that
$(\sigma(v),\sigma(w))=(w,v)$.

For $v,w\in V_i$ (i.e., $v,w\in V$ with  $c(v)=c(w)=i$), we have for $k,h\in F_i$, 
\[
m_{v\to w}(k,h) = |\{ u\in V_i\colon c(v,u)=k\mbox{ and }c(u,w)=h\}|. 
\]
If $\Gamma$ is cyclic, then its restriction to each orbit $V_i$ is cyclic and transitive. Thus, there exists $\tau\in\Gamma$ for which $V_i= \{\tau^t(v)\}_{t=0}^{|V_i|-1}$ and there exists $d\in\{0,\ldots,|V_i|-1\}$ such that 
$w=\tau^d(v)$. Denote $n_i=|V_i|$. Then, 
\[
m_{v\to w}(k,h) = |\{ t\in \{0,\ldots,n_i-1\}\colon c(v,\tau^t(v))=k\mbox{ and }c(\tau^t(v),\tau^d(v))=h\}|. 
\]
Since $\tau\in\mathrm{Aut}(G,\mathcal{C})$, we have 
\[
c(v,\tau^t(v))=k \iff c(\tau^{d-t}(v),\tau^d(v))=k 
\]
and similarly
\[
c(\tau^t(v),\tau^d(v))=h\iff c(v,\tau^{d-t}(v))=h. 
\]
Thus,
\begin{align*}
m_{v\to w}(k,h) &= |\{ t\in \{0,\ldots,n_i-1\}\colon c(\tau^d(v),\tau^{d-t}(v))=k\mbox{ and }c(\tau^{d-t}(v),v)=h\}|\\
&=m_{w\to v}(k,h). 
\end{align*}

\end{proof}

\begin{proof}[Proof of Proposition \ref{prop:RCOP}]
\begin{itemize}
    \item[(i)] Assume that $c(v,w)=c(v',w')$. Then, there exists $\sigma\in \Gamma$ such that $\{\sigma(v),\sigma(w)\}=\{v',w'\}$. We have two cases. Suppose that $\sigma(v) = v'$. Then, $\sigma(w)=w'$ and, by Lemma \ref{lem:GG} (i), 
\begin{align*}
m_{v\to w}(k, h) = m_{v'\to w'}(k, h). 
\end{align*}
On the other hand, if  $\sigma(v) = w'$ and $\sigma(w)=v'$, we arrive at 
\begin{align*}
m_{v\to w}(k, h) = m_{w'\to v'}(k, h). 
\end{align*}
In both cases, 
\begin{align*}
m_{v\leftrightarrow w}(k, h) &= m_{v\to w}(k,h)+m_{w\to v}(k,h) \\
&= m_{v'\to w'}(k,h)+m_{w'\to v'}(k,h) = m_{v'\leftrightarrow w'}(k, h),
\end{align*}
which implies $2$-path regularity. 
    \item[(ii)] The validity of this construction follows from the fact that $\Gamma$ is a subgroup of $\mathrm{Aut}(G)$. Thus, if $v_{\sigma_1}$ is simplicial in $G$, then all vertices in the same color are also simplicial in $G$. We iterate this procedure to obtain the cpeo. 
    \item[(iii)] Since $G$ is decomposable, there exists a perfect elimination ordering of its vertices. By (ii), there exists a cpeo. By (i), $(G,\mathcal{C})$ satisfies (M1) with respect to this cpeo. 
    \item[(iv)] By part (iii), \((G,\mathcal{C})\) is a CER graph with respect to any cpeo. It remains only to verify condition (M2). Since \(\Gamma\) is orbitwise generously transitive or cyclic, condition (M2) follows from Lemma \ref{lem:GG}(ii). Hence, \((G,\mathcal{C})\) is a symmetric CER graph with respect to any cpeo.
\end{itemize}
\end{proof}

\subsection{Proofs for Section~\ref{sec:bcspace}}
\subsubsection{Preliminary results and proof of Theorem~\ref{thm:ZGCisNice}}

In each case, choose a cpeo from the corresponding
definition of CER or symmetric CER
and use it to define the block decomposition.
Write this cpeo as $(V_{\eta_1},\ldots,V_{\eta_r})$ for some
$\eta\in\mathfrak{S}_r$, put $n_i=|V_{\eta_i}|$, $i\in[r]$,
and relabel the vertices so that for each $i$,
\[
V_{\eta_i}=\{N_{i-1}+1,\ldots,N_i\},\qquad
N_i=\sum_{k\in[i]}n_k,\qquad N_0=0.
\]
We use this vertex-numbering to define $\coloredZ$ and relabel
the vertex-color classes according to the chosen ordering,
so that it becomes $(V_1,\ldots,V_r)$.

The proof of Theorem \ref{thm:ZGCisNice} follows from Lemmas \ref{lem:Z1}, \ref{lem:Z2} and \ref{lem:Z0}.

We introduce some additional notation. Recall that $c(v,w) = k \iff \{v,w\}\in E_k$.
\begin{definition}[Base of $\coloredZ$]\label{def:base}
Let \((G,\mathcal{C})\) be a colored graph and let \((\tilde{G},\tilde{\mathcal{E}}) = \RefExt(G,\mathcal{C})\) denote its reflexive extension.
For each $k\in [r+R]$, define the matrix $J^k\in\{0,1\}^{p\times p},$ by
\[
J^k_{vw} =\begin{cases}
1, & \text{if }c(v,w) = k\\
0, & \text{otherwise},
\end{cases}
\]
for $v,w\in V$.

Since every matrix in \(\coloredZ\) is constant on color classes, the set $\{J^k\colon k\in [r+R]\}$ forms a basis of $\coloredZ$.

\end{definition}

\begin{propo}\label{pro:m}
For all $v,w\in V$ and all $k,h\in [r+R]$, we have 
\begin{align*}
\left[\tri\left(J^k\right)\tri\left(J^h\right)^\top\right]_{vw} = m_{v\to w}(k,h).     
\end{align*}
\end{propo}
\begin{proof}[Proof of Proposition \ref{pro:m}]
Fix $k,h\in [r+R]$ and define 
\[
S \coloneq \tri(J^k)\cdot\tri(J^h)^\top.
\]
We have 
\begin{align*}
\tri(J^k)_{vu} &= \begin{cases}
    J_{vu}^k, & \mbox{ if }c(u)\leq c(v), \\ 0 & \mbox{otherwise}.
    \end{cases}
\intertext{and similarly}
[\tri(J^h)^\top]_{uw} &= \begin{cases}
    J_{uw}^h, & \mbox{ if }c(u)\leq c(w), \\ 0 & \mbox{otherwise}.
\end{cases} 
\end{align*}
Thus, for any $v,w\in V$, we obtain
\begin{align*}
S_{vw} &= \left[\tri(J^k)\cdot\tri(J^h)^\top\right]_{vw} = \sum_{u\in V} \tri(J^k)_{vu} [\tri(J^h)^\top]_{uw}\\
&= \sum_{u\in V_{\le c(v)\wedge c(w)}}J_{vu}^kJ_{uw}^h.
\end{align*}
By the definition of $J^k$ and $J^h$, the product $J_{vu}^kJ_{uw}^h$ is $1$ exactly when $c(v,u)=k$ and $c(u,w)=h$; otherwise, it is $0$. Hence, the sum counts the number of $2$-paths $v-u-w$ in the subgraph induced by the set $\{v,w\}\cup V_{\leq c(v)\wedge c(w)}$, where $c(v,u) = k$ and $c(u,w)=h$. Therefore, 
\[
S_{vw} = m_{v\to w}(k,h).
\]
\end{proof}

\begin{lemma}\label{lem:Z1}
If \((G,\mathcal{C})\) is a CER graph, then $\coloredZ$ satisfies condition (Z1). 
\end{lemma}
\begin{proof}
Let \((G,\mathcal{C})\) be a CER graph. Thus, there exists a vertex-color ordering that is a cpeo and $(G,\mathcal{C})$ is $2$-path regular with respect to this ordering. Without loss of generality, we assume that $(V_{1},\ldots,V_r)$ is a cpeo.

Using the standard polarization identity and the linearity of the mapping $\tri$, together with the fact that $(J^k)_{k\in[r+R]}$ form a basis of $\coloredZ$, we see that condition (Z1) is equivalent to   
\begin{align}\label{eq:Z1}
\tri\left(J^k\right)\tri\left(J^h\right)^\top + \tri\left(J^h\right)\tri\left(J^k\right)^\top\in\coloredZ,
\end{align}
for all $k,h\in [r+R]$.

To prove \eqref{eq:Z1}, fix $k,h\in [r+R]$ and define $S = \tri(J^k)\cdot\tri(J^h)^\top$ so that \eqref{eq:Z1} is written as $S+S^\top\in\coloredZ$.  By Proposition \ref{pro:m}, we have $S_{vw} = m_{v\to w}(k,h)$.
 Using the identity $m_{w\to v}(k,h) = m_{v\to w}(h,k)$, we obtain  
 $(S+S^\top)_{vw} = m_{v\leftrightarrow w}(k,h)$. 
 
 Since $\coloredZ$ is already expressed in a block form dictated by vertex coloring, it suffices to verify the following two properties:
\begin{itemize}
    \item[(a)] (zero pattern) If all matrices in $\coloredZ$ have a zero in the $(v,w)$ entry, then \linebreak\mbox{$\left[S + S^\top\right]_{vw} = 0$}.
    \item[(b)] (entry equalities) If all matrices in $\coloredZ$ have equal entries at positions $(v,w)$ and $(v',w')$, then $[S + S^\top]_{vw} = [S + S^\top]_{v'w'}$.
\end{itemize}

For (a), suppose that $v$ and $w$ are distinct and $\{v,w\}\notin E$. Then there are no $2$-paths between $v$ and $w$ in the subgraph induced by the set $\{v,w\}\cup V_{\leq c(v)\wedge c(w)}$. Hence $m_{v\to w}\equiv m_{w\to v}\equiv 0$ and consequently $[S+S^\top]_{vw} = 0$. 

To show (b), suppose that $c(v,w) = c(v',w')$. 
By (M1), we have $m_{v\leftrightarrow w} = m_{v'\leftrightarrow w'}$, which directly implies
\[
[S+S^\top]_{vw} = [S+S^\top]_{v'w'}. 
\]
This completes the proof of (Z1) for the colored subspace $\coloredZ$.
\end{proof}

Recall the definition of $F_i$ from Section \ref{sec:2path}. 
For $i\in[r]$, we have
\[
F_i = \left\{c(v,w)\colon c(v)=i=c(w)\mbox{ and }\{v,w\}\in \tilde{E}\right\}
\]
and $F = \bigcup_{i\in [r]} F_i$.
We note $i\in F_i$ so the set $F_i$ is nonempty.

\begin{propo}\label{prop:T_diag_properties}
Assume that $(G,\mathcal{C})$ is a CER graph. Then, 
\begin{enumerate}
    \item For every $k\in F$ we have  $\bdiag(J^k) =J^k = \tri(J^k)$. 
    \item For every $k\in[r+R]\setminus F$ we have $\bdiag(J^k) =0$. 
    \item For all $k, h \in F$ we have 
    \[
\left[J^kJ^{h}\right]_{vw} =   m_{v\to w}(k,h). 
    \]
    \item Under (M2) for every $i\in [r]$ and for all $k,h\in F_i$ we have
    \[
J^kJ^h = J^hJ^k.
    \]
\end{enumerate}
\end{propo}
\begin{proof}
Points (1) and (2) follow directly from the definition of $F$. For (3), by (1) and Proposition \ref{pro:m}, we have
\begin{align*}
   \left[J^kJ^{h}\right]_{vw} = \left[\tri(J^k)\tri(J^{h})^\top\right]_{vw} = m_{v\to w}(k,h). 
\end{align*}
For (4), for each $i\in[r]$, by (M2) and (3), we have for $k,h\in F_i$,
\[
[J^kJ^h]_{vw} = m_{v\to w}(k,h) = m_{w\to v}(k,h) = [J^hJ^k]_{vw}
\]
provided $c(v)=c(w)=i$. Otherwise, $[J^kJ^h]_{vw} = 0 = [J^hJ^k]_{vw}$ by Lemma \ref{lem:edge_regularity_on_diagonal}.
\end{proof}

\begin{lemma}\label{lem:Z2}
If \((G,\mathcal{C})\) is a symmetric CER graph, then $\coloredZ$ satisfies condition (Z2). 
\end{lemma}
\begin{proof}
By linearity of $\bdiag$, and Proposition \ref{prop:T_diag_properties} (1) and (2),  condition (Z2) is equivalent to 
\begin{align}\label{eq:z2eq}
J^k J^h \in \coloredZ,\qquad k,h\in F.
\end{align}
By Proposition \ref{prop:T_diag_properties} (3), we have $\left[J^kJ^{h}\right]_{vw} = m_{v\to w}(k,h)$.
Note that $m_{v\to w}(k,h)=0$ if $k \in F_i$ and $h\in F_{i'}$ with $i\neq i'$. Thus, it is left to prove $J^k J^h \in \coloredZ$ for $k,h\in F_i$ for every $i\in[r]$. Under (M2), we have $m_{v\to w}(k,h) = m_{v\leftarrow w}(k,h)=\frac12 m_{v\leftrightarrow w}(k,h)$ for $v,w\in V$ with $c(v)=c(w)=i$. Thus, \eqref{eq:z2eq} follows by the same argument as \eqref{eq:Z1} in the proof of Lemma \ref{lem:Z1}.
\end{proof}

\begin{lemma}\label{lem:L_i}
Assume that $(G,\mathcal{C})$ is a CER graph. If $c(v,w)=c(v',w')=k\in [r+R]\setminus F$, $c(v)\le c(w)$ and $c(v')\le c(w')$, then $c(v)=c(v')$.
\end{lemma}
\begin{proof}
By (M1),
\[
m_{v\leftrightarrow w} = m_{v'\leftrightarrow w'}.
\]
Since $k\notin F$, the endpoints of an extended edge of color $k$ cannot have the same vertex color. Hence $c(v)<c(w)$ and $c(v')<c(w')$.

We first record a simple observation. If $c(x,y)=k$ and $c(x)<c(y)$, then, for every loop color $i$,
\begin{equation}\label{eq:lower-endpoint-color}
m_{x\leftrightarrow y}(i,k)=\1_{\{c(x)=i\}}.
\end{equation}
Indeed, in $m_{x\to y}(i,k)$ the loop color $i$ forces the intermediate vertex to be $u=x$, so this term equals $1$ precisely when $c(x)=i$. On the other hand, in $m_{x\to y}(k,i)$ the loop color $i$ would force $u=y$; this is impossible because the intermediate vertex must lie in $V_{\le c(x)}$, whereas $c(y)>c(x)$.

Applying \eqref{eq:lower-endpoint-color} first to $(x,y)=(v,w)$ and then to $(x,y)=(v',w')$, with $i=c(v)$, and using (M1), we obtain
\[
1=m_{v\leftrightarrow w}(c(v),k)
=m_{v'\leftrightarrow w'}(c(v),k)
=\1_{\{c(v')=i\}}.
\]
Therefore $c(v')=i=c(v)$, as claimed.
\end{proof}

\begin{lemma}\label{lem:Z0}
Assume that $(G,\mathcal{C})$ is a CER graph. Then, $\coloredZ$ satisfies (Z0). 
\end{lemma}
\begin{proof}
It is clear that $I_p\in\coloredZ$. It remains to show that
\begin{enumerate}
    \item Suppose $c(v,w)=c(v',w')\in[r+R]$ and $c(v)=c(w)=i\in[r]$. Then, $c(v')=c(w')=i$.
    \item Suppose $c(v,w)=c(v',w')\in [r+R]\setminus F$, $c(v)\le c(w)$ and $c(v')\le c(w')$. Then $c(v)=c(v')$.
\end{enumerate}
These follow from Lemma \ref{lem:edge_regularity_on_diagonal} and Lemma \ref{lem:L_i}.
\end{proof}

\subsubsection{Proof of Theorem~\ref{thm:new_Cholesky}}
Let $\tri(\mathcal{Z})^+$ 
 be the subset of $\tri(\mathcal{Z})$ 
 consisting of elements whose $(i,i)$-block components
 are positive definite symmetric matrices for all $i\in[r]$.

\begin{lemma}\label{lem:gen_Cholesky0} 
Assume (Z0) and (Z1).  
Let $x \in \mathcal{Z} \cap \mathrm{Sym}^+(p)$.\\
{\rm (i)} 
There exist 
 unique $T_i \in \tri(L_i(\mathcal{Z})) = L_i(\mathcal{Z}) e^{(i)}$  for $i \in [r-1]$ 
 and 
 unique $C =\sum_{i\in [r]} C_i \in \bdiag(\mathcal{Z}) \cap \mathrm{Sym}^+(p),\,\,\,
 C_i \in M_i(\mathcal{Z})^+\,\, (i\in [r])$
 such that
 \begin{equation} \label{eqn:TiC}
 x = (I_p+ \sum_{i\in[r-1]} T_i) \,C\, \mytransp{ (I_p + \sum_{i\in[r-1]} T_i)}.
 \end{equation}
{\rm (ii)} 
For $i \in [r]$ and $C_i \in M_i(\mathcal{Z})^+$,
 there exists unique
 $S_i \in \sum_{\alpha\in[d_i]} (\R_+ c^{(i)}_\alpha \oplus c^{(i)}_{>\alpha} M_i(\mathcal{Z}) c^{(i)}_\alpha )$ 
 such that $S_i \mytransp{S_i} = C_i$.\\
{\rm (iii)}
There exists a unique $T \in \hh^+$ such that $x = T \mytransp{T}$.
\end{lemma}

\begin{proof}[Proof of Lemma~\ref{lem:gen_Cholesky0}]

(i) We prove the statement by induction on the integer $r$,
 which we call the rank of the \BC-space structure of $\mathcal{Z}$.
If $r=1$, then the statement holds trivially with $C = x$.
Let us consider the case $r>1$,
 assuming that the statement holds for
 any \BC-space structure whose rank is smaller than $r$.
Let $\mathcal{Z}' \subset \mathcal{Z}$ be
 the linear subspace
 $\Bigl( \bigoplus_{i=2}^r M_i(\mathcal{Z}) \Bigr) \oplus 
  \Bigl( \bigoplus_{i=2}^{r-1} L_i(\mathcal{Z}) \Bigr). $
Then $(1, i)$ block components of any $x' \in \mathcal{Z}'$ are zero for $i \in [r]$,
 and $x' \in \mathcal{Z}'$ is naturally identified with its $p' \times p'$ right-bottom submatrix,
 where $p' := p - n_1$.
In this way,
 $\mathcal{Z}'$ can be regarded as a linear subspace of $\mathrm{Sym}(p')$ with a \BC-space structure of rank $r-1$. 
Define 
 $e' := I_p - e^{(1)} = \sum_{i=2}^{r} e^{(i)}\in  \mathcal{Z}'$.
Then we have $\mathcal{Z}' = e' \mathcal{Z} e'$.

For $x \in \mathcal{Z} \cap \mathrm{Sym}^+(p)$,
 we have  
$$ x = (e^{(1)} + e')x(e^{(1)} + e')
 = e^{(1)} x e^{(1)} + e' x e^{(1)} + e^{(1)} x e' + e' x e'
 = x_1 + y_1 + x', $$
 where
 \begin{align*}
 x_1 := e^{(1)} x e^{(1)}  \in M_1(\mathcal{Z}), \quad
 y_1 := e' x e^{(1)} + e^{(1)} x e' \in L_1(\mathcal{Z}), \quad
 x' := e' x e' \in \mathcal{Z}'.
\end{align*}

Let us verify that, if (\ref{eqn:TiC}) holds, 
 then $C$ and $T_i$ are determined uniquely from $x$.
Note that $C_1 = e^{(1)} C = C e^{(1)}$ and 
 that $C' := \sum_{i=2}^r C_i$ is equal to $e' C =  e'C e'$.
Then we have 
 $T_i C_1 = T_i e' e^{(1)} C_1 = 0$ for $2 \le i \le r-1$
 and $T_1 C' = T_1 e^{(1)} e' C = 0$.
Thus
\begin{equation} \label{eqn:TiCTi}
\begin{aligned}
 x &= (I_p+ \sum_{i\in[r-1]} T_i) C_1 \mytransp{ (I_p + \sum_{i\in[r-1]} T_i)}
        +(I_p+ \sum_{i\in[r-1]} T_i) C' \mytransp{ (I_p + \sum_{i\in[r-1]} T_i)}\\
&= (I_p+T_1) C_1 \mytransp{ (I_p +T_1)}
        +(I_p+ \sum_{i\in[r-1]} T_i) e' C' e' \mytransp{ (I_p + \sum_{i\in[r-1]} T_i)}\\
&= C_1+T_1 C_1 + C_1 \mytransp{T_1} + T_1 C_1 \mytransp{T_1}
        +(e' + \sum_{i=2}^{r-1} T_i) C' \mytransp{ (e' + \sum_{i=2}^{r-1} T_i)}.
\end{aligned}
\end{equation}
Substituting the above to the equality $x_1 =e^{(1)} x e^{(1)}$ 
 and noting that $e^{(1)} T_i = 0$,
 we have   
 $x_1  = e^{(1)} C_1 e^{(1)} = C_1$.
Furthermore, we have
\begin{align*}
  y_1 e^{(1)} = e' x e^{(1)} = T_1 C_1. 
\end{align*}     
Finally, we obtain
 $$
 x' = e' x e' = T_1 C_1 \mytransp{T_1} + (e' + \sum_{i=2}^{r-1} T_i) C' \mytransp{ (e' + \sum_{i=2}^{r-1} T_i)}.
$$
From these observations, 
 we deduce 
  $C_1 = x_1$ and $T_1 = y_1 x_1^{<-1>}$,
 where
 $x_1^{<-1>} \in M_1(\mathcal{Z})$ denotes the inverse of $x_1$ as an element of the Jordan algebra $M_1(\mathcal{Z})$.
 Then we have
 \begin{align*}
 (e' + \sum_{i=2}^{r-1} T_i) C' \mytransp{ (e' + \sum_{i=2}^{r-1} T_i)}
 = x' - T_1 C_1  \mytransp{T_1}
 = x' - y_1 x_1^{<-1>} \mytransp{y_1},
 \end{align*}
 which is regarded as a $p' \times p'$ positive definite symmetric matrix,
 and the induction hypothesis tells us that
 the elements $T_2, \dots, T_{r-1}$ and $C_2, \dots, C_r$ are determined uniquely.

Now let us show the existence of $T_i$ and $C_i$ for which (\ref{eqn:TiC}) holds.
Let $x = x_1 + y_1 + x' \in \mathcal{Z} \cap \mathrm{Sym}^+(p)$.
Then $x_1$ can be identified with the $n_1 \times n_1$ left-top submatrix of $x$, which is positive definite.
Put $T_1 := y_1 x_1^{<-1>} \in \tri(L_1(\mathcal{Z}))$.
We observe that
\begin{align*}
{} & (I_p - T_1) x \mytransp{(I_p - T_1)} 
 = (I_p - T_1) x_1 \mytransp{(I_p - T_1)} + (I_p - T_1)y_1\mytransp{(I_p - T_1)} + x'\\
 &= x_1  - T_1 x_1 - x_1 \mytransp{T_1} + T_1 x_1 \mytransp{T_1} + y_1 - T_1 y_1 - y_1 \mytransp{T_1} + x' \\
 &= x_1 - y_1 e^{(1)} - e^{(1)} \mytransp{y_1} + y_1 x_1^{<-1>} \mytransp{y_1} + y_1 - 2 y_1 x_1^{<-1>} \mytransp{y_1} + x' \\
 &= x_1 + x' - y_1 x_1^{<-1>} \mytransp{y_1},   
\end{align*}
 where the last term belongs to $M_1(\mathcal{Z}) \oplus \mathcal{Z}'$.
Since the left-hand side is positive definite, 
 the element $x' - y_1 x_1^{<-1>} \mytransp{y_1} \in \mathcal{Z'}$
 is regarded as a $p' \times p'$ positive definite symmetric matrix.
The induction hypothesis tells us that there exist $T_i\,\,(2\le i \le r-1)$ and $C'$ such that
$$
x' - y_1 x_1^{<-1>} \mytransp{y_1} 
= (e' + \sum_{i=2}^{r-1} T_i) C' \mytransp{ (e' + \sum_{i=2}^{r-1} T_i)}.
$$
Then
$$ (I_p - T_1) x \mytransp{(I_p - T_1)} 
= C_1 +(e' + \sum_{i=2}^{r-1} T_i) C' \mytransp{ (e' + \sum_{i=2}^{r-1} T_i) },$$
 where we put $C_1 := x_1 \in M_1(\mathcal{Z})^+$.
Since $T_1^2 = 0$, 
we have $I_p - T_1 = (I_p + T_1)^{-1}$.
Thus
$$
x = (I_p + T_1)\{ C_1+ (e' + \sum_{i=2}^{r-1} T_i) C' \mytransp{ (e' + \sum_{i=2}^{r-1} T_i)} \} \mytransp{(I_p + T_1)}, 
$$
 which yields (\ref{eqn:TiC}) as is seen in (\ref{eqn:TiCTi}).

\noindent
(ii) It is enough to show the following statement
 for a general Jordan subalgebra $\mathcal{M}$ of $\mathrm{Sym}(q)$ containing $I_q$:
 
\textit{
Let $\{c_1, \dots c_d\}$ be a Jordan frame of $\mathcal{M}$,
 and define
 $$
 H(\mathcal{M}; \{c_\alpha\})
 := \sum_{\alpha \in [d]} (\R_+ c_{\alpha} \oplus c_{>\alpha}\mathcal{M} c_\alpha )
 $$
 with
 $c_{>\alpha} := \sum_{\beta > \alpha} c_\beta$. 
Then, 
 for any positive element $C \in \mathcal{M}$, 
 there exists a unique
 $S \in H(\mathcal{M}; \{c_\alpha\})$
 for which $C = S \mytransp{S}$.}

Let us verify that, 
 for any $S \in H(\mathcal{M}; \{c_\alpha\})$ and $a \in \mathcal{M}$,
 we have $S a \mytransp{S} \in \mathcal{M}$.
We write
$$ S = \sum_{\alpha \in [d]} (t_\alpha c_{\alpha} + c_{>\alpha} a_\alpha c_\alpha ) $$
 with $t_\alpha > 0$ and $a_\alpha \in \mathcal{M}$ for $\alpha \in [d]$.
Then we have a factorization
$$
 S = (I_q + \sum_{\alpha \in [d]} t_\alpha^{-1} c_{> \alpha} a_\alpha c_\alpha) (\sum_{\alpha \in [d]} t_\alpha c_\alpha).
$$
Furthermore,
 since $c_{\alpha} c_{>(\alpha+1)}= 0$ for $\alpha \in [d-1]$, 
 we have
 \begin{equation} \label{eqn:factorize_S}
 S = (I_q + t_1^{-1} c_{>1} a_1 c_1) (I_q + t_2^{-1} c_{>2} a_2 c_2) \dots (I_q + t_{d-1}^{-1} c_{>(d-1)} a_{d-1} c_{d-1}) 
 (\sum_{\alpha  \in [d]} t_\alpha c_\alpha).  
 \end{equation}
It is well-known that, 
 for any $b,c \in \mathcal{M}$, 
 we have $bab \in \mathcal{M}$ and $cab + bac \in \mathcal{M}$.
Putting $b := \sum_{\alpha  \in [d]} t_\alpha c_\alpha$,
  we have
 $$ (\sum_{\alpha  \in [d]} t_\alpha c_\alpha) a (\sum_{\alpha  \in [d]} t_\alpha c_\alpha) \in \mathcal{M}.$$
On the other hand, 
 putting $b := c_{>\alpha} a_\alpha c_\alpha + c_\alpha a_\alpha c_{>\alpha} \in \mathcal{M}$ for $\alpha \in [d-1]$,
 and keeping in mind $b c_\alpha = c_{>\alpha} a_\alpha c_\alpha$, 
 we have
 $$
 c_{>\alpha} a_\alpha c_\alpha a + a c_\alpha a_\alpha c_{>\alpha} = b c_\alpha a + a c_\alpha b \in \mathcal{M},
 $$
 so that
  \begin{align*}
 {} &  (I_q + t_{\alpha}^{-1}c_{>\alpha} a_\alpha c_\alpha) a \mytransp{(I_q + t_\alpha^{-1}c_{>\alpha} a_\alpha c_\alpha)}\\
  &= a + t_{\alpha}^{-1}(c_{>\alpha} a_\alpha c_\alpha a + a c_\alpha a_\alpha c_{>\alpha}) 
  + t_{\alpha}^{-2} c_{>\alpha} (a_\alpha (c_\alpha a c_\alpha) a_\alpha) c_{>\alpha} \in \mathcal{M}.
  \end{align*}
Therefore, 
 we see from (\ref{eqn:factorize_S}) that
 $S a \mytransp{S} \in \mathcal{M}$.
 
 Now our statement about $C = S \mytransp{S}$
 can be proved 
 by the induction on the rank $d$ of the Jordan algebra structure of $\mathcal{M}$
 in a way quite similar to the proof of (i).

For later use, 
 we make some remarks about the set $H(\mathcal{M}; \{ c_\alpha \})$.
First, 
 for $S \in H(\mathcal{M}; \{ c_\alpha \})$, 
 we see from (\ref{eqn:factorize_S}) that
 \begin{equation} \label{eqn:S_inverse}
 S^{-1} 
 = (\sum_{\alpha \in [d]} t_\alpha^{-1} c_\alpha)
   (I_q - t_{d-1}^{-1} c_{>(d-1)}a_{d-1} c_{d-1}) \dots (I_q - t_1^{-1} c_{>1} a_1 c_1).
 \end{equation}

Secondly,  
 if we define
 $$
  N(\mathcal{M}; \{ c_\alpha \}) := \sum_{\alpha \in [d]} c_{>\alpha}\mathcal{M} c_\alpha,
 $$
 then any $S \in H(\mathcal{M}; \{ c_\alpha \})$ is written uniquely as
  $S = S_0 + \sum_{\alpha \in [d]} t_\alpha c_\alpha$
  with $S_0 \in N(\mathcal{M}; \{ c_\alpha \})$
  and $t_\alpha >0\,\,(\alpha \in [d])$.
Since the Jordan frame $\{c_1, c_2, \dots, c_d\}$ is mutually commutative,
  there exists an orthogonal matrix $U \in O(q)$ such that
  $\mytransp{U} c_\alpha U = \sum_{i = q_{\alpha-1} + 1}^{q_\alpha} E_{ii}$ for $\alpha \in [d]$,
  where $q_0 := 0$ and $q_\alpha := \sum_{\beta \le \alpha} \mathrm{rank}\, c_\beta$ for $\alpha \in [d]$.
Then 
 we see that
 $\mytransp{U} (\sum_{\alpha \in [d]} t_\alpha c_\alpha) U$ equals a diagonal matrix
 in which $t_\alpha$ appears ($\mathrm{rank}\, c_\alpha$)-times as diagonal components,
 and that
 $\mytransp{U} S_0 U$ is a lower triangular matrix whose diagonal components are zero.
It follows that 
 $\det S =  \prod_{\alpha \in [d]} t_\alpha^{\mathrm{rank}\, c_\alpha}$, 
 and we obtain
 $$
 \det C = (\det S)^2 = \prod_{\alpha \in [d]}
  t_\alpha^{2\, \mathrm{rank}\, c_\alpha}. 
 $$

\noindent
(iii)
Take $C = \sum_{i \in [r]} C_i$ and $T_i\,\,(i \in [r-1])$ for which (\ref{eqn:TiC}) holds.
By (ii), 
 we can take unique $S_i \in H(M_i(\mathcal{Z}); \{c^{(i)}_\alpha\})$
 for each $i \in [r]$
 such that $C_i = S_i \mytransp{S_i}$. 
Then 
\begin{align*}
x &= (I_p+ \sum_{i\in[r-1]} T_i) C  \mytransp{(I_p + \sum_{i\in[r-1]} T_i)}\\
  & = (I_p+ \sum_{i\in[r-1]} T_i) (\sum_{i\in[r]} S_i) \mytransp{(\sum_{i\in[r]} S_i)} \mytransp{(I_p + \sum_{i\in[r-1]} T_i)}
  = T \mytransp{T},
\end{align*}
where 
$T :=  (I_p+ \sum_{i\in[r-1]} T_i) (\sum_{i\in[r]} S_i) 
= \sum_{i\in[r]} S_i + \sum_{i\in[r-1]} T_i S_i \in \hh^+$.
Note that $T_i S_i \in \tri(L_i(\mathcal{Z}))$
 because $T_i a \in \tri(L_i(\mathcal{Z}))$ for any $a \in M_i(\mathcal{Z})$.

On the other hand,
 for an element
 $$ T = \sum_{i \in [r]} \sum_{\alpha\in[d_i]} (t_{i,\alpha}  \tilde{e}^{(i)}_\alpha
   + \sum_{\gamma\in J_{i,\alpha}} \tau_{i,\gamma} \tilde{f}^{(i)}_\gamma ) \in \hh^+, $$
 there exist unique $S_{0,i} \in N(M_i(\mathcal{Z}); \{c^{(i)}_\alpha \} )$ for $i \in [r]$
 and unique $T'_i \in \tri(L_i(\mathcal{Z}))$ for $i \in [r-1]$ such that
 $$
 \sum_{\alpha\in[d_i]} \sum_{\gamma\in J_{i,\alpha}} \tau_{i,\gamma} \tilde{f}^{(i)}_\gamma
 = \begin{cases} S_{0,i} + T'_i & (i \in [r-1]), \\ S_{0,i} & (i=r). \end{cases} 
 $$
We put
 $$S_i
  := S_{0,i} +  \sum_{\alpha\in[d_i]} t_{i,\alpha}  \tilde{e}^{(i)}_\alpha 
 = S_{0,i} +  \sum_{\alpha\in[d_i]} \frac{ t_{i,\alpha} }{\sqrt{\mu_{i,\alpha}}}  c^{(i)}_\alpha 
 \in H(M_i(\mathcal{Z}); \{ c^{(i)}_\alpha\} ) $$ 
 for $i \in [r]$.
Recalling the inverse element $S_i^{<-1>}$ given by the formula (\ref{eqn:S_inverse}), 
 we have
\begin{align*}
T &= S_r + \sum_{i \in [r-1]} (S_i + T'_i)
  = (I_p + \sum_{i \in [r-1]} T'_i S_i^{<-1>})(\sum_{i \in [r]} S_i) \\
  &= (I_p + \sum_{i \in [r-1]} T_i)(\sum_{i \in [r]} S_i) \qquad (T_i := T'_i S_i^{<-1>} \in \tri(L_i(\mathcal{Z}))).
\end{align*}
It follows that $T \mytransp{T}$ belongs to $\mathcal{Z} \cap \mathrm{Sym}^+(p)$, 
 and the factorization of $T$ above implies the uniqueness of $T$. 
 
 \end{proof}
 
\begin{proof}[Proof of Theorem~\ref{thm:new_Cholesky}]
(i)
As in the proof of Lemma~\ref{lem:gen_Cholesky0} (iii), 
 the element $T \in \mathfrak{h}^+$ is written of the form
$$
T = (I_p + \sum_{i \in [r-1]} T_i)(\sum_{i \in [r]} S_i), \quad
 S_i = S_{0,i} + \sum_{\alpha \in [d_i]} t_{i,\alpha} \tilde{e}^{(i)}_\alpha.   
$$
Then 
$$
\det T = \det (I_p + \sum_{i \in [r-1]} T_i) \cdot \det(\sum_{i \in [r]} S_i) = \prod_{i \in [r]} \det\nolimits_{\R^{n_i \times n_i}} S_i,
$$
 where $\det_{\R^{n_i \times n_i}} S_i$ stands for the determinant
  of $S_i$ identified with the $\R^{n_i \times n_i}$ matrix of its $(i,i)$-block component. 
As is discussed in the proof of Lemma~\ref{lem:gen_Cholesky0} (ii), 
 we have
 $\det_{\R^{n_i \times n_i}} S_i
  = \prod_{\alpha \in [d_i]} (\frac{t_{i,\alpha}}{\sqrt{\mu_{i,\alpha}} } )^{\mu_{i,\alpha}}$.
Therefore, we obtain
\[
\det x = \det TT^\top = (\det T)^2
 =  \prod_{i\in[r]} ( \det\nolimits_{\R^{n_i \times n_i}} S_i )^2 
= \prod_{i \in [r]} \prod_{\alpha \in [d_i]} \Bigl( \frac{t_{i,\alpha}^2}{\mu_{i,\alpha} } \Bigr)^{\mu_{i,\alpha}}.
\]

(ii) 

By the definition of $\mathfrak{h}_{i,\alpha}$, 
 we have
 $$
  t_{i,\alpha}  \tilde{e}^{(i)}_\alpha
  + \sum_{\gamma\in J_{i,\alpha}} \tau_{i,\gamma} \tilde{f}^{(i)}_\gamma 
  =
  (t_{i,\alpha}  \tilde{e}^{(i)}_\alpha
   + \sum_{\gamma\in J_{i,\alpha}} \tau_{i,\gamma} \tilde{f}^{(i)}_\gamma) c^{(i)}_\alpha
  \in \R\tilde{e}^{(i)}_\alpha \oplus \mathfrak{h}_{i,\alpha}.
  $$
Then we observe
\begin{align*}
{} & x = T \mytransp{T} \\
&= \sum_{i \in [r]}\sum_{\alpha \in [d_i]} \sum_{i' \in [r]} \sum_{\alpha' \in [d_{i'}]}
(t_{i,\alpha}  \tilde{e}^{(i)}_\alpha
   + \sum_{\gamma\in J_{i,\alpha}} \tau_{i,\gamma} \tilde{f}^{(i)}_\gamma ) c^{(i)}_\alpha c^{(i')}_{\alpha'} 
 \mytransp{(t_{i',\alpha'}  \tilde{e}^{(i')}_{\alpha'}
   + \sum_{\gamma' \in J_{i',\alpha'}} \tau_{i',\gamma'} \tilde{f}^{(i')}_{\gamma'} )}.   
 \end{align*}
Since $c^{(i)}_\alpha c^{(i')}_{\alpha'} =0$ unless $(i,\alpha) = (i', \alpha')$, 
 we obtain
 \begin{equation} \label{eqn:x_TT}
 \begin{aligned}
 x &= \sum_{i \in [r]}\sum_{\alpha \in [d_i]} 
 (t_{i,\alpha}  \tilde{e}^{(i)}_\alpha
 + \sum_{\gamma\in J_{i,\alpha}} \tau_{i,\gamma} \tilde{f}^{(i)}_\gamma ) 
 \mytransp{(t_{i,\alpha}  \tilde{e}^{(i)}_{\alpha}
   + \sum_{\gamma' \in J_{i,\alpha}} \tau_{i,\gamma'} \tilde{f}^{(i)}_{\gamma'} )}\\
   &= \sum_{i \in [r]}\sum_{\alpha \in [d_i]} 
 \Bigl( t_{i,\alpha}^2  \tilde{e}^{(i)}_\alpha \tilde{e}^{(i)}_\alpha
 + \sum_{\gamma\in J_{i,\alpha}} t_{i,\alpha} \tau_{i,\gamma} 
  (\tilde{e}^{(i)}_\alpha \mytransp{(\tilde{f}^{(i)}_\gamma) }
   + \tilde{f}^{(i)}_\gamma \mytransp{(\tilde{e}^{(i)}_\alpha)} )\\
 &\hphantom{={}\sum_{i \in [r]}\sum_{\alpha \in [d_i]}\Bigl(}{}+ \sum_{\gamma,\,\gamma' \in J_{i,\alpha} } 
   \tau_{i,\gamma} \tau_{i,\gamma'}  \tilde{f}^{(i)}_\gamma \mytransp{ (\tilde{f}^{(i)}_{\gamma'}) } \Bigr).
 \end{aligned}
 \end{equation}
Let $A \in \mathrm{Sym}(p)$.
Recalling the definitions 
\begin{align*}
\lambda_{i,\alpha}(A)
 = \tr(A \tilde{e}^{(i)}_{\alpha} \tilde{e}^{(i)}_{\alpha}),\quad 
(v_{i,\alpha})_{1\gamma}
 = \tr (A \tilde{e}^{(i)}_{\alpha}  \mytransp{(\tilde{f}^{(i)}_\gamma)} ), \quad 
\psi_{i,\alpha}(A)_{\gamma \gamma'}
 = \tr\, (A\tilde{f}^{(i)}_{\gamma} \mytransp{ (\tilde{f}^{(i)}_{\gamma'}) }),
\end{align*}
 we see that $\tr(Ax)$  equals
$$
\sum_{i\in[r]} \sum_{\alpha\in[d_i]} 
\Bigl( (t_{i,\alpha})^2 \lambda_{i,\alpha}(A)
+2 \sum_{\gamma \in J_{i,\alpha}}
    t_{i,\alpha} \tau_{i,\gamma}  v_{i,\alpha}(A)_{1 \gamma} 
+ \sum_{\gamma,\, \gamma' \in J_{i,\alpha}}
 \tau_{i,\gamma} \tau_{i, \gamma'} \psi_{i,\alpha}(A)_{\gamma \gamma'} \Bigr)
$$
 and this formula is rewritten as
\begin{equation*}
 \tr(xA) =
\sum_{i\in[r]} \sum_{\alpha\in[d_i]} 
\begin{bmatrix} t_{i,\alpha} & (\tau^{(i)}_\alpha)^\top \end{bmatrix}
\begin{bmatrix} \lambda_{i,\alpha}(A) &  v_{i,\alpha}(A) \\
 v_{i,\alpha}(A)^\top & \psi_{i,\alpha}(A) \end{bmatrix}
\begin{bmatrix} t_{i,\alpha} \\ \tau^{(i)}_\alpha \end{bmatrix},
\end{equation*}
where $\tau_{\alpha}^{(i)}=(\tau_{i,\gamma})_{\gamma\in J_{i,\alpha}}\in \R^{m_{i,\alpha}}$.

(iii) 
For $\gamma \in J_{i,\alpha}$, 
 put
 $\tilde{g}^{(i)}_\gamma
 = \frac{1}{\sqrt{2}} (\tilde{f}^{(i)}_\gamma + (\tilde{f}^{(i)}_\gamma)^\top) 
 \in \mathcal{Z}$. 
Then
\[
 \{\tilde{e}^{(i)}_\alpha\}_{i \in [r],\,\alpha \in [d_i]}
 \cup \{\tilde{g}^{(i)}_\gamma\}_{i \in [r],\,\alpha \in [d_i],\,\gamma\in J_{i,\alpha}}
\]
 is an orthonormal basis of
 $\mathcal{Z}$.
An element $x \in \mathcal{Z}$ is expressed as
\[ x = \sum_{i\in[r]} \sum_{\alpha\in[d_i]} 
(x_{i,\alpha} \tilde{e}^{(i)}_\alpha
 + \sum_{\gamma\in J_{i,\alpha}}
 y_{i,\gamma} \tilde{g}^{(i)}_\gamma) 
 \]
 with
 $x_{i,\alpha} := \mathrm{tr}\,(x \tilde{e}^{(i)}_\alpha)$
  and 
  $y_{i,\gamma}:= \mathrm{tr}\, (x \tilde{g}^{(i)}_\gamma)$.
Since
 $\tilde{e}^{(i)}_\alpha =  c^{(i)}_\alpha/\sqrt{\mu_{i,\alpha}}$ and $c^{(i)}_\alpha$ is an idempotent, 
 we have 
\[
\tilde e^{(i)}_\alpha \tilde e^{(i)}_\alpha 
= \frac{1}{\sqrt{\mu_{i,\alpha}}} \tilde{e}^{(i)}_\alpha
\]
and if $\gamma \in J_{i,\alpha}$,
\begin{align*}
 \tilde{e}^{(i)}_\alpha (\tilde{f}^{(i)}_\gamma)^\top
+ \tilde{f}^{(i)}_\gamma (\tilde{e}^{(i)}_\alpha)^\top 
 = \frac{1}{\sqrt{\mu_{i,\alpha}}}
( (\tilde{f}^{(i)}_\gamma c^{(i)}_\alpha)^\top + \tilde{f}^{(i)}_\gamma c^{(i)}_\alpha) = \sqrt{\frac{2}{\mu_{i,\alpha}} } \tilde{g}^{(i)}_\gamma.
\end{align*}
Then (\ref{eqn:x_TT}) is rewritten as
\begin{align*}
x&= \sum_{i \in [r]}\sum_{\alpha \in [d_i]} 
 \Bigl( \frac{t_{i,\alpha}^2}{\sqrt{\mu_{i,\alpha}} }  \tilde{e}^{(i)}_\alpha 
 + \sum_{\gamma\in J_{i,\alpha}} \sqrt{\frac{2}{\mu_{i,\alpha}} } t_{i,\alpha} \tau_{i,\gamma} \,\tilde{g}^{(i)}_\gamma \Bigr)\\ 
 &\qquad + \sum_{i' \in [r]}\sum_{\alpha' \in [d_{i'}]} \sum_{\gamma',\,\delta' \in J_{i',\alpha'} } 
   \tau_{i',\gamma'} \tau_{i',\delta'}  \tilde{f}^{(i')}_{\gamma'} \mytransp{ (\tilde{f}^{(i')}_{\delta'}) }.
 \end{align*}  
Let us discuss the quantities
 $\mathrm{tr}\,( \tilde{f}^{(i')}_{\gamma'} \mytransp{ (\tilde{f}^{(i')}_{\delta'}) } \tilde{e}^{(i)}_\alpha)$
 and 
 $\mathrm{tr}\,( \tilde{f}^{(i')}_{\gamma'} \mytransp{ (\tilde{f}^{(i')}_{\delta'}) } \tilde{g}^{(i)}_\gamma)$.
We put $e^{(>i')} := \sum_{j >i'} e^{(j)} \in \mathcal{Z}$. 
By the definition of $\mathfrak{h}_{i',\alpha'}$, 
 we have $(e^{(>i')} + c^{(i')}_{> \alpha'}) \tilde{f}^{(i')}_{\gamma'} = \tilde{f}^{(i')}_{\gamma'}$.
If $i < i'$ or $i = i'$ with $\alpha \le \alpha'$,
 which means that $(i, \alpha) \preceq (i', \alpha')$ in the lexicographic order, 
 then
 $$
  c^{(i)}_\alpha (e^{(>i')} + c^{(i')}_{> \alpha'})  =(e^{(>i')} + c^{(i')}_{> \alpha'})  c^{(i)}_\alpha = 0.
 $$
In this case,
 noting that
 $\tilde{f}^{(i)}_\gamma = \tilde{f}^{(i)}_\gamma c^{(i)}_\alpha$
 for $\gamma \in J_{i,\alpha}$, 
 we have
\begin{align*}
{} & \mathrm{tr}\,( \tilde{f}^{(i')}_{\gamma'} \mytransp{ (\tilde{f}^{(i')}_{\delta'}) } \tilde{g}^{(i)}_\gamma)\\
&= \frac{1}{\sqrt{2}}\mathrm{tr}\,( (e^{(>i')} + c^{(i')}_{> \alpha'}) \tilde{f}^{(i')}_{\gamma'} \mytransp{ (\tilde{f}^{(i')}_{\delta'}) } \tilde{f}^{(i)}_\gamma)
+\frac{1}{\sqrt{2}}\mathrm{tr}\,( \tilde{f}^{(i')}_{\gamma'} \mytransp{ (\tilde{f}^{(i')}_{\delta'}) } (e^{(>i')} + c^{(i')}_{> \alpha'}) \mytransp{(\tilde{f}^{(i)}_\gamma)} )\\
&= \frac{1}{\sqrt{2}}\mathrm{tr}\,( \tilde{f}^{(i')}_{\gamma'} \mytransp{ (\tilde{f}^{(i')}_{\delta'}) } \tilde{f}^{(i)}_\gamma
 c^{(i)}_{\alpha} (e^{(>i')} + c^{(i')}_{> \alpha'})  )
+\frac{1}{\sqrt{2}}\mathrm{tr}\,( \tilde{f}^{(i')}_{\gamma'} \mytransp{ (\tilde{f}^{(i')}_{\delta'}) }
 (e^{(>i')} + c^{(i')}_{> \alpha'}) c^{(i)}_\alpha \mytransp{(\tilde{f}^{(i)}_\gamma)})\\
&=0.
\end{align*}
Similarly, 
 we have
 $$\mathrm{tr}\,( \tilde{f}^{(i')}_{\gamma'} \mytransp{ (\tilde{f}^{(i')}_{\delta'}) } \tilde{e}^{(i)}_\alpha) = 0.$$

Therefore, 
 we obtain
\begin{align*}
x_{i,\alpha} &= \frac{t_{i,\alpha}^2}{\sqrt{\mu_{i,\alpha}}}
+\sum_{(i', \alpha') \prec (i,\alpha) } \sum_{\gamma', \delta' \in J_{i',\alpha'}}
 \tau_{i', \gamma'} \tau_{i', \delta'} 
\tr\, (\tilde{f}^{(i')}_{\gamma'} (\tilde{f}^{(i')}_{\delta'})^\top 
\tilde{e}^{(i)}_\alpha ), \\
y_{i,\gamma} &= \sqrt{ \frac{2}{\mu_{i,\alpha}} } t_{i,\alpha} \tau_{i,\gamma}
+  \sum_{ (i', \alpha') \prec (i,\alpha) } \sum_{\gamma', \delta' \in J_{i',\alpha'}}
 \tau_{i', \gamma'} \tau_{i', \delta'} 
\tr\, (\tilde{f}^{(i')}_{\gamma'} (\tilde{f}^{(i')}_{\delta'})^\top \tilde{g}^{(i)}_\gamma ). \\
\end{align*}
Thus
\begin{align*}
\dd x_{i,\alpha} &= \frac{2 t_{i,\alpha}}{\sqrt{\mu_{i,\alpha} } }\,\dd t_{i,\alpha}
 + (\mbox{terms of $\dd\tau_{i',\gamma'}$'s} ),\\
\dd y_{i,\gamma} &= \sqrt{ \frac{2}{\mu_{i,\alpha}} } t_{i,\alpha}\, \dd\tau_{i,\gamma}
 +  (\mbox{terms of $\dd t_{i,\alpha}$ and $\dd\tau_{i',\gamma'}$'s} ).
\end{align*}
Considering the lexicographic order among variables $t_{i,\alpha}$'s and $\tau_{i',\gamma'}$'s,
 where we assume that $t_{i,\alpha} \prec \tau_{i,\gamma}$ for $\gamma \in J_{i,\alpha}$, 
 the Jacobian matrix is lower triangular and
 we obtain
\begin{equation*}
\begin{aligned}
\dd x 
&= \prod_{i\in[r]} \prod_{\alpha\in[d_i]}
\Bigl(  \dd x_{i,\alpha} \cdot \prod_{\gamma\in J_{i,\alpha}} \dd y_{i,\gamma}\Bigr) \\
&= \prod_{i\in[r]} \prod_{\alpha\in[d_i]} 
\left\{ \frac{2 t_{i,\alpha}}{\sqrt{\mu_{i,\alpha} } }\,\dd t_{i,\alpha}
 \cdot \prod_{\gamma \in J_{i,\alpha} } 
\sqrt{ \frac{2}{\mu_{i,\alpha}} } t_{i,\alpha}\, \dd\tau_{i,\gamma} \right\} \\
&= \prod_{i\in[r]} \prod_{\alpha\in[d_i]}
 \Bigl\{ 
2^{1 + m_{i,\alpha}/2} \mu_{i,\alpha}^{-(1+m_{i,\alpha})/2} t_{i,\alpha}^{1 + m_{i,\alpha}}
 \, \dd t_{i,\alpha} \dd\tau^{(i)}_\alpha
 \Bigr\},
\end{aligned}
\end{equation*}
where $\dd\tau^{(i)}_\alpha$ denotes the Lebesgue measure on $\R^{m_{i,\alpha}}$.

This completes the proof of (iii) and hence of Theorem~\ref{thm:new_Cholesky}.
\end{proof}

\subsubsection{Proof of Theorem \ref{thm:integral} }
Combining (i), (ii) and (iii) of Theorem \ref{thm:new_Cholesky},
 we see that
 the integral $\int_{\mathcal{P}} e^{-\tr(xA)} (\det x)^s\, \dd x$ 
equals the product of
\begin{equation} \label{eqn:integral_ialpha}
\begin{aligned}
{} & 2^{1 + m_{i,\alpha}/2} \mu_{i,\alpha}^{-(\mu_{i,\alpha} s + (1+m_{i,\alpha})/2)}\\
&\times
\int_{(0,\infty) \times \R^{m_{i,\alpha}} }
 \exp \Bigl( - \begin{bmatrix} t_{i,\alpha} & (\tau^{(i)}_\alpha)^\top \end{bmatrix}
\phi_{i,\alpha}(A)\begin{bmatrix} t_{i,\alpha} \\ \tau^{(i)}_\alpha \end{bmatrix}\Bigr)
 t_{i,\alpha}^{2 \mu_{i,\alpha} s + 1 + m_{i,\alpha}}\,\dd t_{i,\alpha} \dd\tau^{(i)}_\alpha
\end{aligned}
\end{equation}
over $i\in[r]$ and $\alpha\in[d_i]$.
\begin{lemma} \label{lemma:exp_quad}
Assume that
 $\hat{S} = \begin{bmatrix} \lambda & v \\ v^\top & S \end{bmatrix} 
 \in \mathrm{Sym}^+(1+m)$, 
 where
$\lambda \in \R,\, v \in \R^{1 \times m}$ and $S \in \mathrm{Sym}(m)$.
Then the integral
\begin{equation} \label{eqn:Gauss-Gamma} 
\int_{(0,\infty) \times \R^m}
\exp \Bigl( - \begin{bmatrix} t & \tau^\top \end{bmatrix}
\hat{S}
\begin{bmatrix} t\\ \tau \end{bmatrix} \Bigr)
\,t^{2s} \,\dd t\, \dd\tau
\end{equation} 
 converges if and only if $s >-1/2$, and in this case it equals 
\[
\frac12 \pi^{m/2} \Gamma\Bigl(s + \frac12\Bigr) \Bigl( \frac{\det \hat{S} }{\det S} \Bigr)^{-(s + 1/2)}
 (\det S)^{-1/2}.
\]
\end{lemma}
\begin{proof}
First of all,
we note that
\[
\hat{S}
 = \begin{bmatrix} \lambda & v \\ v^\top & S \end{bmatrix}
= \begin{bmatrix} 1 & v S^{-1} \\ 0 & I_m \end{bmatrix} 
\begin{bmatrix} \lambda - v S^{-1} v^\top & 0 \\ 0 & S \end{bmatrix}
\begin{bmatrix} 1 & 0\\ S^{-1} v^\top & I_m \end{bmatrix}, 
\]
which implies that $\hat{S}$ is positive definite if and only if $S$ is positive definite
 and $\lambda - v S^{-1} v^\top >0$.
Moreover, we have
\begin{equation} \label{eqn:det_hatS}
\det \hat{S} = (\lambda - v S^{-1} v^\top) \det S. 
\end{equation}
On the other hand,
 we have
\begin{align*}
{} &
 \begin{bmatrix} t & \tau^\top \end{bmatrix}
\begin{bmatrix} \lambda & v \\ v^\top & S \end{bmatrix}
\begin{bmatrix} t\\ \tau \end{bmatrix} 
= \lambda t^2 + 2 t v \tau + \tau^\top S \tau \\
&= (\lambda - v S^{-1} v^\top) t^2 + (\tau^\top + t v S^{-1}) S (\tau + t S^{-1}v^\top).
\end{align*}
The classical Gaussian integral formula tells us that
$$ \int_{\R^m}
\exp(-  (\tau^\top + t v S^{-1}) S (\tau + t S^{-1}v^\top ))
\,\dd\tau = \pi^{m/2} (\det S)^{-1/2}, $$
which is independent of $t>0$.
Thus \eqref{eqn:Gauss-Gamma} equals
\begin{align*}
{} & \pi^{m/2} (\det S)^{-1/2}
\int_0^\infty e^{- (\lambda - v S^{-1} v^\top) t^2} t^{2s}\,\dd t\\
&= \pi^{m/2} (\det S)^{-1/2} \int_0^{\infty} e^{- (\lambda - v S^{-1} v^\top)u}
u^s \cdot \frac{1}{2} u^{-1/2} \,\dd u\\
&= \frac{\pi^{m/2}}{2} (\det S)^{-1/2} 
 \Gamma(s + 1/2) (\lambda - v S^{-1} v^\top)^{-(s+1/2)}.
\end{align*}
Therefore, the assertion follows from \eqref{eqn:det_hatS}.
\end{proof}

Lemma \ref{lemma:exp_quad} tells us that \eqref{eqn:integral_ialpha}
 equals
\begin{align*}
{} & 2^{m_{i,\alpha}/2} \mu_{i,\alpha}^{-(\mu_{i,\alpha} s + (1+m_{i,\alpha})/2)}\\
& \times
\pi^{m_{i,\alpha}/2} \Gamma\Bigl(\mu_{i,\alpha} s + 1 + \frac{m_{i,\alpha}}{2}\Bigr)
\Bigl( \frac{\det \phi_{i,\alpha}(A)}{\det \psi_{i,\alpha}(A)} 
\Bigr)^{-(\mu_{i,\alpha} s + 1+ m_{i,\alpha} /2 )}
 (\det \psi_{i,\alpha}(A))^{-1/2}, 
\end{align*}
whence Theorem \ref{thm:integral} follows.

\subsection{Proofs for Section~\ref{sec:computation}}
\begin{proof}[Proof of Theorem~\ref{thm:GC}]
Fix $(i,\alpha)$.
Under the standing assumption (Z2), \eqref{eq:defhh_simple} gives
\[
  \hh_{i,\alpha} = L_i(\mathcal{Z})\,c_\alpha^{(i)}
  = \mathrm{span}\{B_t^{(i)}c_\alpha^{(i)} \colon t\in[q_i]\}.
\]
The matrices $U_t := B_t^{(i)}c_\alpha^{(i)}$ (or, equivalently, their compressed
versions $\widehat U_t$) thus form a spanning family of $\hh_{i,\alpha}$.
The Gram matrix $G_{i,\alpha}$ is precisely the Gram matrix of
$\{U_t\}_{t=1}^{q_i}$ with respect to the Frobenius inner product
$\langle X,Y\rangle_F = \tr(XY^\top)$.
Hence
\[
  \mathrm{rank}(G_{i,\alpha}) = \dim \hh_{i,\alpha}.
\]
A Cholesky factorization with complete pivoting yields pivot indices
$\mathcal{I}_{i,\alpha} = \{t_1,\dots,t_{m_{i,\alpha}}\}$ such that
$U_{t_1},\dots,U_{t_{m_{i,\alpha}}}$ are linearly independent and span
$\hh_{i,\alpha}$. Therefore $m_{i,\alpha} = \dim\hh_{i,\alpha}$,
and $\widetilde G_{i,\alpha}$ is the Frobenius Gram matrix of this basis.

Let $U_b$ be the column-block matrix of dimension
$p^2\times m_{i,\alpha}$ whose columns are the vectorizations
$\mathrm{vec}(U_{t_k})$, $k\in [m_{i,\alpha}]$. Then
\[
  \widetilde G_{i,\alpha} = U_b^\top U_b.
\]
Let $C\in\R^{m_{i,\alpha}\times m_{i,\alpha}}$ be an invertible matrix such that
the columns of $U_b C$ form an orthonormal basis of $\hh_{i,\alpha}$,
i.e.
\[
  C^\top \widetilde G_{i,\alpha} C = I_{m_{i,\alpha}}.
\]
If we denote these orthonormal basis elements by
$\{\tilde f_\gamma^{(i)}\}_{\gamma\in J_{i,\alpha}}$, then
by definition of $\psi_{i,\alpha}(A)$,
\[
  \psi_{i,\alpha}(A)
  = C^\top \Psi_{i,\alpha}(A) C,
\]
where $\Psi_{i,\alpha}(A)$ is the $A$-weighted Gram matrix in the
(non-orthonormal) basis $\{U_{t_k}\}_{k=1}^{m_{i,\alpha}}$.
Taking determinants and using
$C^\top \widetilde G_{i,\alpha} C = I_{m_{i,\alpha}}$,
we obtain
\[
  \det\psi_{i,\alpha}(A)
  = (\det C)^2 \det\Psi_{i,\alpha}(A),
  \qquad
  (\det C)^2 \det\widetilde G_{i,\alpha} = 1,
\]
and thus
\[
  \det\psi_{i,\alpha}(A)
  = \frac{\det\Psi_{i,\alpha}(A)}
         {\det\widetilde G_{i,\alpha}}.
\]

Next, consider the direct sum
$\R\tilde e^{(i)}_\alpha \oplus \hh_{i,\alpha}$.
In the orthonormal basis
$\{\tilde e^{(i)}_\alpha\}\cup\{\tilde f_\gamma^{(i)}\}_{\gamma\in J_{i,\alpha}}$,
the matrix of the bilinear form
$(X,Y)\mapsto \tr(AXY^\top)$ is $\phi_{i,\alpha}(A)$.
In the mixed basis
$\{\tilde e^{(i)}_\alpha\}\cup\{U_{t_k}\}_{k=1}^{m_{i,\alpha}}$,
the same form has matrix
\[
  \Phi_{i,\alpha}(A)
  =
  \begin{bmatrix}
    \lambda_{i,\alpha}(A) & V_{i,\alpha}(A) \\
    V_{i,\alpha}(A)^\top & \Psi_{i,\alpha}(A)
  \end{bmatrix}.
\]
The change-of-basis matrix from the orthonormal basis to the mixed basis
is block-diagonal,
\[
  \begin{bmatrix}
    1 & 0\\ 0 & C
  \end{bmatrix},\quad\mbox{so}\quad 
  \phi_{i,\alpha}(A)
  =
  \begin{bmatrix}
    1 & 0\\ 0 & C^\top
  \end{bmatrix}
  \Phi_{i,\alpha}(A)
  \begin{bmatrix}
    1 & 0\\ 0 & C
  \end{bmatrix},
\]
and hence
\[
  \det\phi_{i,\alpha}(A)
  = (\det C)^2 \det\Phi_{i,\alpha}(A).
\]
Therefore
\[
  \frac{\det\phi_{i,\alpha}(A)}{\det\psi_{i,\alpha}(A)}
  = \frac{\det\Phi_{i,\alpha}(A)}{\det\Psi_{i,\alpha}(A)}.
\]
Since $\Psi_{i,\alpha}(A)$ is positive definite (because $A\in\mathrm{Sym}^+(p)$
and the $U_k$ are linearly independent), the Schur complement formula gives
\[
  \frac{\det\Phi_{i,\alpha}(A)}{\det\Psi_{i,\alpha}(A)}
  = \lambda_{i,\alpha}(A)
    - V_{i,\alpha}(A) \Psi_{i,\alpha}(A)^{-1} V_{i,\alpha}(A)^\top.
\]
This is exactly the expression returned by the algorithm.

Finally, in the case $m_{i,\alpha}=0$, we have $\hh_{i,\alpha}=\{0\}$,
so the corresponding $\psi_{i,\alpha}(A)$ block is the empty $0\times 0$
matrix, whose determinant is $1$ by convention, whereas
$\phi_{i,\alpha}(A)=[\lambda_{i,\alpha}(A)]$ by the definition in
Section~\ref{sec:structure_constants}. Hence the formulas in the algorithm
coincide with the definitions. This completes the proof.
\end{proof}

\subsection{Proofs for Section~\ref{sec:relations}}

\subsubsection{Proof of Theorem~\ref{thm:RDAG}}

We relate the local isomorphism conditions \eqref{eq:i-iii} formulated by \citet{Seigal23} to our CER conditions in two steps:
(i) is the perfect-DAG (no-immorality) condition and yields chordality of the skeleton, while
(ii)--(iii) yield $2$-path regularity with respect to a suitable
ordering of the vertex-color classes, through counts of children
and common children.

The following result is standard; it is the characterization of perfect DAGs (DAGs with no immoralities),
see \citep[Sec.~2.1.3 and Prop.~2.17]{L96}.

\begin{lemma}\label{lem:revtop_is_peo}
Let $\mathcal G$ be a DAG with skeleton $G$, and let
$(v_1,\ldots,v_p)$ be an inverse topological ordering,
so that $v_j\to v_i$ implies $i<j$.
Then $\mathcal G$ has no induced subgraph
$i\to k\leftarrow j$ with $i$ and $j$ non-adjacent
if and only if $(v_1,\ldots,v_p)$ is a perfect
elimination ordering of $G$.
\end{lemma}

\begin{proof}
The neighbors occurring after each vertex in this ordering
are exactly its parents. Thus, the perfect elimination
condition is equivalent to requiring that any two distinct
parents of a vertex be adjacent.
\end{proof}

\begin{lemma}\label{lem:cpeoRDAG}
Let $(\mathcal{G},c)$ be a colored DAG with compatible coloring and assume \eqref{eq:i-iii} (ii).
Then there exists an inverse topological order $\succ$ of $V$ such that each vertex color class forms a consecutive block in $\succ$, i.e.
\[
i\succ j\succ k \ \text{ and }\ c(i)=c(k)\quad\Longrightarrow\quad c(j)=c(i).
\]
\end{lemma}

\begin{proof}
Fix vertex colors $\alpha,\beta$.
By \eqref{eq:i-iii} (ii), for any $v,v'$ with $c(v)=c(v')=\alpha$ we have $(\mathcal{G}_v,c)\cong(\mathcal{G}_{v'},c)$,
hence $v$ has a $\beta$-colored child iff $v'$ has a $\beta$-colored child.

Let $\mathcal{C}_V:=c(V)$ and define the color directed graph $H=(\mathcal{C}_V,E^H)$ by
\[
(\alpha\to\beta)\in E^H
\ \Longleftrightarrow\
\exists\, (u\to v)\in \vec{E} \ \text{ with }\ c(u)=\alpha,\ c(v)=\beta.
\]
Then $H$ is acyclic. Indeed, if $H$ contained a directed cycle
$\alpha_0\to\alpha_1\to\cdots\to\alpha_{t-1}\to\alpha_0$, choose $v_0$ with $c(v_0)=\alpha_0$. By previous considerations,  every vertex of color $\alpha_r$ has a child of color
$\alpha_{r+1}$ (indices mod $t$), hence we can build a directed walk
\[
v_0\to v_1\to\cdots\to v_t \quad\text{with}\quad c(v_r)=\alpha_r,\ \ c(v_t)=\alpha_0.
\]
Iterating produces an infinite directed walk in the finite graph $\mathcal{G}$, forcing a repeated vertex and
therefore a directed cycle in $\mathcal{G}$, a contradiction. Thus $H$ is a DAG.

Choose an inverse topological order $\succ_{\mathcal{C}}$ of $H$ and order vertices by their colors under $\succ_{\mathcal{C}}$
(breaking ties arbitrarily within each color class). The resulting order $\succ$ is inverse topological for $\mathcal{G}$ and
has the required block property by construction.
\end{proof}

\begin{lemma}\label{lem:no-same-colour-edges}
Let $(\mathcal{G},c)$ be a colored DAG with compatible coloring and assume \eqref{eq:i-iii} (ii).
Then there are no directed edges between vertices of the same vertex color.
\end{lemma}

\begin{proof}
Assume for contradiction that $j\to i$ with $c(i)=c(j)=:\alpha$. 
If $v$ is any vertex with $c(v)=\alpha$, condition \eqref{eq:i-iii} (ii) gives $(\mathcal{G}_v,c)\cong(\mathcal{G}_j,c)$,
hence $v$ also has a child of color $\alpha$. Iterating produces a directed cycle in the finite DAG $\mathcal{G}$, a contradiction.
\end{proof}

\begin{proof}[Proof of Theorem~\ref{thm:RDAG}]
\noindent{$(\Rightarrow)$}
Assume that $c$ is a compatible coloring and that $(\mathcal{G},c)$ satisfies \eqref{eq:i-iii}.
By Lemma~\ref{lem:cpeoRDAG} choose an inverse topological order $\succ$ whose color classes are consecutive blocks. 
By \eqref{eq:i-iii} (i) and Lemma~\ref{lem:revtop_is_peo}, the order $\succ$ is a perfect elimination ordering of the skeleton $G$.
By Lemma~\ref{lem:no-same-colour-edges}, $G$ has no edges within a vertex color class; hence the induced ordering of color blocks
is a cpeo.

Relabel the vertex-color classes as $(V_1,\ldots,V_r)$ in the order induced by $\succ$, and let $m$ be the $2$-path function defined with respect to this ordered partition.
For an extended edge $\{v,w\}\in\tilde E$, the $2$-paths
$v-u-w$ involving two ordinary edges and satisfying
$c(u)\le c(v)\wedge c(w)$ correspond exactly to common
children $u$ of $v$ and $w$ (to children of $v$ if $v=w$).
Consequently, for ordinary edge colors $h,h'$ the value $m_{v\leftrightarrow w}(h,h')$ counts common children $u\in\ch(v)\cap\ch(w)$
with unordered pair of incoming edge colors $\{c(v\to u),c(w\to u)\}=\{h,h'\}$.
When $h=h'$, symmetrization counts each child twice;
this common factor does not affect the equality of counts.
The contributions involving loops depend only on the vertex
color when $v=w$, and otherwise on the edge color and the
color of the vertex to which the edge is directed.
They therefore agree on extended-edge color
classes by compatibility, which also determines each child's
vertex color from its incoming edge colors.
Therefore, for vertices $i,j$ of the same color,
$\mathcal{G}_i\cong\mathcal{G}_j$ iff $m_{i\leftrightarrow i}=m_{j\leftrightarrow j}$.
Similarly, for directed edges $j\to i$ and $\ell\to k$ of the same color,
$\mathcal{G}_{(j\to i)}\cong\mathcal{G}_{(\ell\to k)}$ iff $m_{j\leftrightarrow i}=m_{\ell\leftrightarrow k}$.
Hence \eqref{eq:i-iii} (ii)-(iii) are equivalent to (M1), i.e., $2$-path regularity.
Since there are no edges within a vertex color class, (M2) holds trivially, so $(G,\mathcal{C})$ is symmetric CER.

\noindent$(\Leftarrow)$
Assume $(G,\mathcal{C})$ is symmetric CER and that $G$ has no edges between vertices of the same color.
Choose a cpeo from the definition of symmetric CER, relabel the vertex-color classes as $(V_1,\ldots,V_r)$ in this order, and choose any vertex order $\succ$ consistent with it
(so vertices are ordered by color blocks).
Orient every edge from its endpoint in the later color class to its endpoint in the earlier one.
The resulting directed graph $\mathcal G$ is acyclic, since each arrow decreases the color-class index, and $\succ$ is an inverse topological order of $\mathcal G$.
Because there are no edges within any color class, every vertex in a given block has all its later neighbors in later blocks,
and the cpeo condition implies that these later neighbors form a clique; hence $\succ$ is a perfect elimination ordering of $G$.
Applying Lemma~\ref{lem:revtop_is_peo} yields \eqref{eq:i-iii} (i).

Since there are no edges within a vertex-color class, $F=[r]$.
By Lemma~\ref{lem:L_i}, equally colored edges have earlier endpoints of the same vertex color.
The corresponding arrows point towards these endpoints,
so the inherited coloring is compatible.
Moreover, (M1) is exactly constancy of $m_{v\leftrightarrow w}$ on extended-edge color classes, which (by the same common-child
interpretation as above) is equivalent to \eqref{eq:i-iii} (ii)-(iii).
\end{proof}

\subsubsection{Proof of Theorem~\ref{thm:group}}

\begin{proof}[Proof of Theorem~\ref{thm:group}]
Set $\mathcal{T}:=\tri(\Zsp)$ and $\mathcal{T}^+:=\tri(\Zsp)^+$. Our proof is based on the following fact, which is not hard to show:
\begin{equation}\label{eq:group_iff_algebra}
\mathcal{T}^+\ \text{is a group}\quad\Longleftrightarrow\quad \mathcal{T}\ \text{is closed under multiplication}.
\end{equation}
It remains to identify multiplicative closure of $\mathcal{T}$ with (M3a)--(M3b).

Recall the matrices $(J^k)$ from Definition~\ref{def:base}. Set $T^{k}:=\tri(J^{k})$. 
The family $\{J^{k}\}_{k\in[r+R]}$ spans $\Zsp$, hence $\{T^{k}\}_{k\in[r+R]}$ spans $\mathcal{T}$. Therefore, the space $\mathcal{T}$ is closed under multiplication if and only if
$T^{k}T^{h}\in\mathcal{T}$ for all $k,h$.
Since $T^kT^h$ is block lower triangular, membership in $\mathcal{T}$ is equivalent to the following two requirements: its $(v,w)$ entry vanishes whenever $v\neq w$, $v\nsim w$, and $c(w)\leq c(v)$; and, for all $v,w,v',w'\in V$ with $\{v,w\},\{v',w'\}\in\tilde E$,
\begin{align}\label{eq:M3'}
c(v,w)=c(v',w')\ \land\ c(w)\leq c(v)\ \land\ c(w')\leq c(v')
\quad\Longrightarrow\quad
\bigl(T^{k}T^{h}\bigr)_{vw}=\bigl(T^{k}T^{h}\bigr)_{v'w'}.
\end{align}
These requirements also ensure symmetry of each diagonal block: when $c(v)=c(w)$, compare $(v,w)$ with $(w,v)$ in \eqref{eq:M3'} if $\{v,w\}\in\tilde E$; otherwise both entries vanish.
Since $T^{k}$ is block lower triangular,
\[
T^{k}_{vu}\neq 0 \ \implies\ c(u)\le c(v),
\qquad
T^{h}_{uw}\neq 0 \ \implies\ c(w)\le c(u).
\]
Therefore
\begin{align*}
(T^{k}T^{h})_{vw}
&=\sum_{u\in V} T^{k}_{vu}T^{h}_{uw} \notag\\
&=\Bigl|\Bigl\{u\in V:\ c(w)\le c(u)\le c(v),\ c(v,u)=k,\ c(u,w)=h\Bigr\}\Bigr|
=\widetilde m_{v\to w}(k,h).
\end{align*}
Thus, for all $k,h\in[r+R]$, the zero requirement is equivalent to (M3a), and \eqref{eq:M3'} is equivalent to (M3b). Combining with \eqref{eq:group_iff_algebra} finishes the proof.
\end{proof}




\appendix

\section{Examples}\label{app:graphical-examples}

This appendix collects detailed graphical examples complementing the main text,
together with a calculation of the integral for strongly regular graph models.

\subsection{Examples separating cpeo, CER, and symmetric CER}
\label{subsec:three-levels}

This subsection gives two small colored graphs separating the successive
conditions considered in the main text. The first admits a cpeo but is not
$2$-path regular, and hence is not CER. The second is CER but not symmetric CER.

\begin{example}[A cpeo without $2$-path regularity]
\label{ex:k3-not-2path-regular}
Let $G=K_3$ have vertex set $V=\{1,2,3\}$ and a single vertex-color class
$V_1=V$. Give the edges $\{1,2\}$ and $\{1,3\}$ color
$\textcolor{graphorange}{a}$, and give the edge $\{2,3\}$ color
$\textcolor{graphgreen}{b}$. The colored graph and its associated matrix space
are shown below.

\begin{center}
\begin{minipage}[c]{0.31\textwidth}
    \centering
    \begin{tikzpicture}[-, thick,
        blue node/.style={circle, draw=graphblue, font=\sffamily\bfseries,
            minimum size=0.65cm, inner sep=0pt}]

        \node[blue node] (1) at (0, 1.5) {\color{graphblue}1};
        \node[blue node] (2) at (3.2, 1.5) {\color{graphblue}2};
        \node[blue node] (3) at (1.6, 0) {\color{graphblue}3};

        \path
            (1) edge[graphorange] (2)
            (1) edge[graphorange] (3)
            (2) edge[graphgreen, dashed] (3);
    \end{tikzpicture}
\end{minipage}
\hfill
\begin{minipage}[c]{0.66\textwidth}
    \centering
    \[
    \coloredZ=
    \left\{
    \begin{bmatrix}
    \textcolor{graphblue}{\alpha} & \textcolor{graphorange}{a} & \textcolor{graphorange}{a}\\
    \textcolor{graphorange}{a} & \textcolor{graphblue}{\alpha} & \textcolor{graphgreen}{b}\\
    \textcolor{graphorange}{a} & \textcolor{graphgreen}{b} & \textcolor{graphblue}{\alpha}
    \end{bmatrix}
    \colon
    \textcolor{graphblue}{\alpha},\textcolor{graphorange}{a},
    \textcolor{graphgreen}{b}\in\R
    \right\}.
    \]
\end{minipage}
\end{center}

Since $G$ is complete, the unique vertex-color ordering $(V_1)$ is a cpeo.
However, (M1) fails. Indeed, the loops $\{1,1\}$ and $\{2,2\}$ have the same
color $\textcolor{graphblue}{\alpha}$, while
\[
m_{1\to 1}(\textcolor{graphorange}{a},\textcolor{graphorange}{a})=2,
\qquad
m_{2\to 2}(\textcolor{graphorange}{a},\textcolor{graphorange}{a})=1.
\]
Indeed, for the first loop, the intermediate vertex can be either $2$ or $3$, whereas
for the second loop only vertex $1$ contributes. Consequently,
\[
m_{1\leftrightarrow 1}(\textcolor{graphorange}{a},\textcolor{graphorange}{a})=4,
\qquad
m_{2\leftrightarrow 2}(\textcolor{graphorange}{a},\textcolor{graphorange}{a})=2.
\]
Thus, the symmetric $2$-path function is not constant on the loop-color class,
so the colored graph is not $2$-path regular. Since $(V_1)$ is its only
vertex-color ordering, it is not CER despite admitting a cpeo.
\end{example}

\begin{example}[A CER graph that is not symmetric CER]\label{ex:six-vertex-cer-not-symmetric}
We revisit the colored graph from Example~\ref{exp:not_nice}.

The underlying graph is complete and has a single
vertex-color class $V_1=V$. Hence the unique vertex-color ordering $(V_1)$ is a
cpeo, since every vertex is simplicial in $G[V_1]=G$. Moreover,
the coloring is induced by the permutation group
\[
\Gamma=\left\langle (1\,3\,5)(2\,4\,6),\ (1\,2)(3\,6)(4\,5)\right\rangle.
\]
It therefore satisfies (M1) by Proposition~\ref{prop:RCOP}(i), and consequently the
colored graph is CER.

It remains to see why (M2) fails. We use the symbols
$\textcolor{graphpurple}{\alpha},\textcolor{graphgreen}{a},
\textcolor{graphbrown}{b},\textcolor{graphyellow}{c},
\textcolor{graphblue}{d}$ also for the corresponding extended-edge colors in the matrix
space displayed in Example~\ref{exp:not_nice}. Since all vertices have the same
color, every extended-edge color belongs to $F$, so
\[
F=\{\textcolor{graphpurple}{\alpha},\textcolor{graphgreen}{a},
\textcolor{graphbrown}{b},\textcolor{graphyellow}{c},
\textcolor{graphblue}{d}\}.
\]
Take $v=1$ and $w=3$. The only $2$-path from $1$ to $3$ whose first edge has color
$\textcolor{graphgreen}{a}$ and whose second edge has color
$\textcolor{graphbrown}{b}$ is
$1\mathbin{\textcolor{graphgreen}{-}}2\mathbin{\textcolor{graphbrown}{-}}3$.
Therefore
\[
m_{1\to 3}(\textcolor{graphgreen}{a},\textcolor{graphbrown}{b})=1.
\]
In the reverse direction, there is no
$\bigl(\textcolor{graphgreen}{a},\textcolor{graphbrown}{b}\bigr)$-colored $2$-path from
$3$ to $1$, so
\[
m_{3\to 1}(\textcolor{graphgreen}{a},\textcolor{graphbrown}{b})=0.
\]
Since vertices $1$ and $3$ have the same color and
$(\textcolor{graphgreen}{a},\textcolor{graphbrown}{b})\in F\times F$, condition
(M2) fails. Hence the colored graph is CER but not symmetric CER.
\end{example}

\subsection{Symmetric CER colorings not induced by permutation groups}\label{subsec:non_group_cer}

\subsubsection{An elementary example with two vertex colors}

\begin{example}\label{ex:elementary-non-group-cer}
Let $G=(V,E)$ be the graph with
\[
V=\{1,2,3\},\qquad E=\bigl\{\{1,3\}\bigr\},
\]
and equip it with the vertex- and edge-color partitions
\[
\mathcal V=\bigl\{\{1,2\},\{3\}\bigr\},
\qquad
\mathcal E=\bigl\{\bigl\{\{1,3\}\bigr\}\bigr\}.
\]
The graph and its associated colored matrix pattern are shown below.

\begin{center}
    \centering
    \begin{minipage}[c]{0.35\textwidth}
        \centering
        \begin{tikzpicture}[-, thick,
            first node/.style={circle, draw=graphpurple, font=\sffamily\bfseries, minimum size=0.6cm, inner sep=0pt},
            second node/.style={rectangle, draw=graphred, font=\sffamily\bfseries, minimum size=0.6cm, inner sep=0pt}]

            \node[first node] (1) at (0, 1.5) {\color{graphpurple}1};
            \node[first node] (2) at (2, 1.5) {\color{graphpurple}2};
            \node[second node] (3) at (1, 0) {\color{graphred}3};

            \path (1) edge[graphgreen, solid] (3);
        \end{tikzpicture}
    \end{minipage}
    \hfill
    \begin{minipage}[c]{0.62\textwidth}
        \centering
        \[
        \coloredZ =
        \left\{
        \begin{bmatrix}
        \textcolor{graphpurple}{\alpha}&0&\textcolor{graphgreen}{a}\\
        0&\textcolor{graphpurple}{\alpha}&0\\
        \textcolor{graphgreen}{a}&0&\textcolor{graphred}{\beta}
        \end{bmatrix}
        \colon \textcolor{graphpurple}{\alpha},\textcolor{graphred}{\beta},\textcolor{graphgreen}{a}\in\R
        \right\}.
        \]
    \end{minipage}
\end{center}

This colored graph is symmetric CER with respect to the ordering $(\{1,2\},\{3\})$. Indeed, this ordering is a cpeo. The only extended-edge color that occurs more than once is the loop color shared by vertices $1$ and $2$. Since the relevant intermediate vertices for this color belong to $\{1,2\}$, the corresponding symmetric $2$-path functions agree, and (M1) holds. For the only distinct pair of vertices of the same color, namely $1$ and $2$, both directed $2$-path functions vanish on $F\times F$, so (M2) also holds.

The coloring is not induced by any permutation subgroup. If it were, vertices $1$ and $2$ would belong to the same orbit. This is impossible because they have different degrees, $\deg(1)=1$ and $\deg(2)=0$, whereas graph automorphisms preserve vertex degrees.
\end{example}

\subsubsection{A one-vertex-color example: the Shrikhande graph}

\begin{example}\label{ex:shrikhande-non-schurian}
Consider a symmetric CER complete graph with a single vertex color.
Section~\ref{sec:decomp} shows that such colorings arise, for example, from
transitive permutation groups that are generously transitive or cyclic. This
prompts the question: are there symmetric CER colorings in this setting that
do not come from a subgroup?

The answer is yes. Any such coloring (with undirected edge colors) determines a symmetric association scheme, and it is well known that not every association scheme is Schurian (i.e., arises as the orbitals of a subgroup $\Gamma$). 

A concrete example comes from the Shrikhande graph,
a strongly regular graph with parameters
$(p,k,\lambda,\mu)=(16,6,2,2)$;
see Appendix~\ref{subsec:strongly-regular-graph-models}
for the definition and further details.
Let $B\in\{0,1\}^{16\times16}$ be its adjacency matrix
on the vertex set $V=[16]$, and let $I_{16}$ and $J_{16}$
be the identity and all-ones matrices, respectively.
The identities in \eqref{eq:srg-adjacency-identities}
imply that $\mathrm{span}\{I_{16},B,J_{16}-I_{16}-B\}$
is a commutative Bose--Mesner algebra.
Thus the two-edge-colored complete graph with one vertex color
defined below is symmetric CER.
Define a $2$-edge-coloring of the complete graph $G=K_{16}$ as follows: for distinct vertices $v,w$, color $\{v,w\}$ blue if $B_{vw}=1$ and red otherwise.

Let \(\mathrm{Aut}(\mathrm{Sh})\) be the automorphism group of the Shrikhande graph. A permutation of \(V\) preserves our two edge-color classes if and only if it is an element of \(\mathrm{Aut}(\mathrm{Sh})\), so every color-preserving permutation lies in \(\mathrm{Aut}(\mathrm{Sh})\). However, the action of \(\mathrm{Aut}(\mathrm{Sh})\) on unordered pairs of distinct vertices has three orbits: one consisting of edges and two distinct orbits of nonedges.  Consequently, our $2$-color partition is coarser than the orbital partition of any subgroup of \(\mathrm{Aut}(\mathrm{Sh})\). Therefore this symmetric CER does not come from any subgroup, i.e., it is non-Schurian.
\end{example}

\subsection{Strongly regular graph models}\label{subsec:strongly-regular-graph-models}

A graph on $p$ vertices is called \emph{strongly regular}
with integer parameters $(p,k,\lambda,\mu)$ if $0<k<p-1$,
every vertex has degree $k$, every two adjacent vertices
have exactly $\lambda$ common neighbors, and every two
distinct nonadjacent vertices have exactly $\mu$ common neighbors.
It is called \emph{primitive} if both the graph and its complement
are connected; see \citet[Section~1.1]{srg}.

The Shrikhande construction in
Example~\ref{ex:shrikhande-non-schurian} is a special case of the following
family. We now derive the integral formula for the colored models associated
with arbitrary primitive strongly regular graphs.

\begin{example}[Strongly regular graph model]\label{ex:strongly-regular-graph-model}
Let $G_0$ be a (primitive) strongly regular graph with parameters
\[
\mathrm{srg}(p,k,\lambda,\mu),
\]
and adjacency matrix $B \in \{0,1\}^{p\times p}$. Thus
\begin{equation}\label{eq:srg-adjacency-identities}
BJ_p = J_pB = kJ_p,
\qquad
B^2 = kI_p + \lambda B + \mu(J_p - I_p - B),
\end{equation}
where $J_p$ is the all-ones $p\times p$ matrix. The eigenvalues of $B$ are
\[
k,\ \theta_1,\ \theta_2\mbox{ with multiplicities }
1,\ f_1,\ f_2.
\]
By \citet{srg}, we have
\[
\theta_{1/2} = \frac12((\lambda-\mu)\pm \sqrt{\Delta}),\qquad \Delta = (\lambda-\mu)^2+4(k-\mu)
\]
and
\[
f_{1/2} = \frac12\left(p-1\mp\frac{2k+(p-1)(\lambda-\mu)}{\sqrt{\Delta}} \right).
\]
Let 
\begin{align*}
E_0 &= \frac{1}{(k-\theta_1)(k-\theta_2)} (B-\theta_1 I_p)(B-\theta_2 I_p) = \frac{1}{p}J_p,\\
E_1 &= \frac{1}{(\theta_1-k)(\theta_1-\theta_2)} (B-k I_p)(B-\theta_2 I_p),\\
E_2 &= \frac{1}{(\theta_2-k)(\theta_2-\theta_1)} (B-k I_p)(B-\theta_1 I_p).
\end{align*}
These are the unique matrices satisfying
\[
E_i E_j=\delta_{ij}E_i,\quad
E_0+E_1+E_2=I_p,\quad B=k E_0+\theta_1 E_1+\theta_2 E_2.
\]
We form the colored complete graph $(G,\mathcal{C})$ as follows:
\begin{itemize}
    \item The vertex set is $V = \{1,\dots,p\}$ with a single vertex color class $V_1=V$.
    \item The underlying graph $G$ is the complete graph $K_p$.
    \item For distinct vertices $i,j$, color $\{i,j\}$ blue if it is an edge of $G_0$, and red otherwise.
\end{itemize}
The associated colored matrix space is
\[
\mathcal{Z}
= \mathcal{Z}_G^{\mathcal{C}}= \mathrm{span}\{I_p, B, J_p-I_p-B\}=\mathrm{span}\{E_0,E_1,E_2\}.
\]

Since $E_0,E_1,E_2$ are minimal idempotents in a commutative Euclidean Jordan algebra, 
$\mathcal{Z}$, with its one-block decomposition ($r=1$), is a \cBC-space in the sense of Definition~\ref{def:I-space}. We now identify the corresponding structure constants.

We have a single block $i=1$ and a Jordan frame
\[
c^{(1)}_1 = E_0,\quad c^{(1)}_2 = E_1,\quad c^{(1)}_3 = E_2.
\]
Their ranks are
\[
\mu_{1,1} = \mathrm{rank}(E_0) = 1,\qquad
\mu_{1,2} = \mathrm{rank}(E_1) = f_1,\qquad
\mu_{1,3} = \mathrm{rank}(E_2) = f_2.
\]
Thus the Jordan rank of $M_1(\mathcal{Z})=\mathcal{Z}$ is $d_1 = 3$.

Because there is only a single block and no off-diagonal part, we have
$L_1(\mathcal{Z})=\{0\}$ and hence
\[
\hh_{1,\alpha} = L_1(\mathcal{Z}) c^{(1)}_\alpha = \{0\},\qquad
m_{1,\alpha} = \dim \hh_{1,\alpha} = 0,\quad \alpha=1,2,3.
\]
Therefore, the structure constants of $\mathcal{Z}$ are
\[
r=1,\quad
d_1 = 3,\quad
(\mu_{1,\alpha})_{\alpha=1}^3 = (1,f_1,f_2),
\quad
(m_{1,\alpha})_{\alpha=1}^3 = (0,0,0).
\]

Let $A\in\mathrm{Sym}^+(p)$ be arbitrary.
Following Section~\ref{sec:structure_constants}, set
\[
\tilde e^{(1)}_\alpha 
:= \frac{c^{(1)}_\alpha}{\sqrt{\mu_{1,\alpha}}}
= \frac{E_{\alpha-1}}{\sqrt{\mu_{1,\alpha}}},\qquad \alpha=1,2,3.
\]
Since $m_{1,\alpha}=0$, the corresponding matrices $\phi_{1,\alpha}(A)$
are scalars:
\[
\phi_{1,\alpha}(A) 
= \lambda_{1,\alpha}(A) 
= \mathrm{tr}(A\,\tilde e^{(1)}_\alpha \tilde e^{(1)}_\alpha)
= \frac{1}{\mu_{1,\alpha}}\mathrm{tr}(A E_{\alpha-1}),
\qquad \alpha=1,2,3.
\]

We normalize the Lebesgue measure $\dd x$ on $\mathcal{Z}$ 
with respect to the trace inner product as in Remark~\ref{rem:Leb}. 
By Theorem~\ref{thm:integral}, for $s\in\mathbb{R}$ the integral
\[
I(s,A)
:= \int_{\mathcal{Z}\cap\mathrm{Sym}^+(p)} (\det x)^s e^{-\mathrm{tr}(Ax)} \,\dd x
\]
converges if and only if
\[
s > \max_{\alpha=1,2,3} \left\{-\frac{2 + m_{1,\alpha}}{2\mu_{1,\alpha}}\right\}
= \max\left\{
-1,\; -\frac{1}{f_1},\; -\frac{1}{f_2}
\right\}
= -\frac{1}{\max(f_1,f_2)}.
\]
In that case,
\[
I(s,A)
= e^{-p_{\mathcal{Z}}s-q_{\mathcal{Z}}}
  \prod_{\alpha=1}^3
  \Gamma\bigl(\mu_{1,\alpha}s + 1\bigr)
  \bigl(\det \phi_{1,\alpha}(A)\bigr)^{-\mu_{1,\alpha}s - 1},
\]
where
\[
p_{\mathcal{Z}}
= \sum_{\alpha=1}^3 \mu_{1,\alpha}\log\mu_{1,\alpha}
= f_1\log f_1 + f_2\log f_2,
\]
and
\[
q_{\mathcal{Z}}
= \frac12 \sum_{\alpha=1}^3 \log\mu_{1,\alpha}
= \frac12\log(f_1f_2).
\]
In particular, if $A=I_p$, then $\phi_{1,\alpha}(I_p)= 1$ for $\alpha=1,2,3$ and we recover the special case
\[
I(s,I_p)
= (f_1f_2)^{-1/2}\, f_1^{-f_1s}\,f_2^{-f_2s}\,
   \Gamma(s+1)\,\Gamma(f_1s+1)\,\Gamma(f_2s+1).
\]
\end{example}

\section{A worked example: from the Random Method to the integral}\label{subsec:random_method_example}
\begin{example}\label{ex:random-method-integral}
Consider the complete graph on five vertices with two vertex colors. The first vertex color class is $V_1=\{1,2,3,4\}$, while the second is the singleton $V_2=\{5\}$. Inside $V_1$, the edges $\{1,3\}$ and $\{2,4\}$ have color $b$, while the remaining four edges have color $a$. All four edges joining $V_1$ to $V_2$ have color $c$. The resulting colored graph is symmetric CER with respect to the ordering $(V_1,V_2)$. The colored graph and its associated matrix space are shown below.

We compute the Jordan frame of the first block by direct
diagonalization and by the Random Method, then evaluate
the integral using the Gram--Cholesky method and
Theorem~\ref{thm:integral}.

\begin{center}
    \centering
    \begin{minipage}[c]{0.35\textwidth}
    \centering
    \begin{tikzpicture}[-, thick,
        first node/.style={circle, draw=graphpurple, font=\sffamily\bfseries, minimum size=0.6cm, inner sep=0pt},
        second node/.style={rectangle, draw=graphred, font=\sffamily\bfseries, minimum size=0.6cm, inner sep=0pt}]

            \node[first node] (1) at (0, 2) {\color{graphpurple}1};
            \node[first node] (2) at (2, 2) {\color{graphpurple}2};
            \node[first node] (3) at (2, 0) {\color{graphpurple}3};
            \node[first node] (4) at (0, 0) {\color{graphpurple}4};
            \node[second node] (5) at (4, 1) {\color{graphred}5};

            \path
                (1) edge[graphgreen, solid] (2)
                (2) edge[graphgreen, solid] (3)
                (3) edge[graphgreen, solid] (4)
                (4) edge[graphgreen, solid] (1);

            \path
                (1) edge[graphblue, dashed] (3)
                (2) edge[graphblue, dashed] (4);

            \path
                (1) edge[graphbrown, densely dotted] (5)
                (2) edge[graphbrown, densely dotted] (5)
                (3) edge[graphbrown, densely dotted] (5)
                (4) edge[graphbrown, densely dotted] (5);

    \end{tikzpicture}
    \end{minipage}
    \hfill
    \begin{minipage}[c]{0.62\textwidth}
    \centering
    \[
    \coloredZ =
    \left\{
    \begin{bmatrix}
    \textcolor{graphpurple}{\alpha}&\textcolor{graphgreen}{a}&\textcolor{graphblue}{b}&\textcolor{graphgreen}{a}&\textcolor{graphbrown}{c}\\
    \textcolor{graphgreen}{a}&\textcolor{graphpurple}{\alpha}&\textcolor{graphgreen}{a}&\textcolor{graphblue}{b}&\textcolor{graphbrown}{c}\\
    \textcolor{graphblue}{b}&\textcolor{graphgreen}{a}&\textcolor{graphpurple}{\alpha}&\textcolor{graphgreen}{a}&\textcolor{graphbrown}{c}\\
    \textcolor{graphgreen}{a}&\textcolor{graphblue}{b}&\textcolor{graphgreen}{a}&\textcolor{graphpurple}{\alpha}&\textcolor{graphbrown}{c}\\
    \textcolor{graphbrown}{c}&\textcolor{graphbrown}{c}&\textcolor{graphbrown}{c}&\textcolor{graphbrown}{c}&\textcolor{graphred}{\beta}
    \end{bmatrix}
    \colon \textcolor{graphpurple}{\alpha},\textcolor{graphred}{\beta},\textcolor{graphgreen}{a},\textcolor{graphblue}{b},\textcolor{graphbrown}{c}\in\R
    \right\}.
    \]
    \end{minipage}
\end{center}

In this example, $p=5$, $r=2$, $n_1=4$, $n_2=1$, and
\[
F_1=\{\textcolor{graphpurple}{\alpha},\textcolor{graphgreen}{a},\textcolor{graphblue}{b}\},
\qquad
F_2=\{\textcolor{graphred}{\beta}\}.
\]

The color matrices on $V_1$ are
\[
A^{(1)}_{\textcolor{graphpurple}{\alpha}}=
\begin{bmatrix}
1&0&0&0\\
0&1&0&0\\
0&0&1&0\\
0&0&0&1
\end{bmatrix},
\qquad
A^{(1)}_{\textcolor{graphgreen}{a}}=
\begin{bmatrix}
0&1&0&1\\
1&0&1&0\\
0&1&0&1\\
1&0&1&0
\end{bmatrix},
\qquad
A^{(1)}_{\textcolor{graphblue}{b}}=
\begin{bmatrix}
0&0&1&0\\
0&0&0&1\\
1&0&0&0\\
0&1&0&0
\end{bmatrix}.
\]
For the singleton block, $A^{(2)}_{\textcolor{graphred}{\beta}}=[1]$.

When used in $\mathcal Z\subset\mathrm{Sym}(5)$, these block
matrices are embedded as
\begin{equation}\label{eq:worked-example-embedded-basis}
\widehat A^{(1)}_k=
\begin{bmatrix}
A^{(1)}_k&0_{4\times1}\\
0_{1\times4}&0
\end{bmatrix},
\quad k\in F_1,
\qquad
\widehat A^{(2)}_{\textcolor{graphred}{\beta}}=
\begin{bmatrix}
0_{4\times4}&0\\
0&1
\end{bmatrix}.
\end{equation}
For example, $A^{(1)}_{\textcolor{graphgreen}{a}}$ has eight
entries equal to one, since each of the four undirected edges
of color $\textcolor{graphgreen}{a}$ occurs twice in the symmetric
matrix. Appending a zero fifth row and column does not change
the Frobenius norm, so
\[
\|\widehat A^{(1)}_{\textcolor{graphgreen}{a}}\|_F
=\|A^{(1)}_{\textcolor{graphgreen}{a}}\|_F
=2\sqrt2.
\]
These four embedded matrices in
\eqref{eq:worked-example-embedded-basis},
together with the adjacency matrix
$B_{\textcolor{graphbrown}{c}}$ of the edges between $V_1$ and $V_2$
(displayed in equation~\ref{eq:worked-example-cross-block-basis}),
form an orthogonal basis of $\mathcal Z$ for the trace inner product.
In the parameter order $(\textcolor{graphpurple}{\alpha},\textcolor{graphred}{\beta},\textcolor{graphgreen}{a},\textcolor{graphblue}{b},\textcolor{graphbrown}{c})$, their norms are
\[
2,\quad 1,\quad 2\sqrt2,\quad 2,\quad 2\sqrt2.
\]
Thus, in the natural coordinates used above, the measure from Remark~\ref{rem:Leb} is
\begin{equation*}
\dd x=32\,\dd\textcolor{graphpurple}{\alpha}\,
\dd\textcolor{graphred}{\beta}\,
\dd\textcolor{graphgreen}{a}\,
\dd\textcolor{graphblue}{b}\,
\dd\textcolor{graphbrown}{c}.
\end{equation*}

Our first goal is to compute a Jordan frame for $M_1(\mathcal Z)=\mathrm{span}\{A^{(1)}_{\textcolor{graphpurple}{\alpha}},A^{(1)}_{\textcolor{graphgreen}{a}},A^{(1)}_{\textcolor{graphblue}{b}}\}$, that is, three mutually orthogonal projection matrices $c_1^{(1)},c_2^{(1)},c_3^{(1)}\in\mathrm{Sym}(4)$ satisfying
\[
c_1^{(1)}+c_2^{(1)}+c_3^{(1)}=I_4,
\qquad
c_k^{(1)}c_h^{(1)}=\delta_{kh}c_k^{(1)},
\quad k,h\in\{1,2,3\}.
\]
The Jordan frame of $M_2(\mathcal Z)$ is simply $c_1^{(2)}=[1]$.

\noindent\emph{Direct diagonalization.}
As explained in Section~\ref{sec:random_method}, we can compute the Jordan frame for the first block in $O(n_1^3)$ time by diagonalizing the generic matrix $x_1=\textcolor{graphpurple}{\alpha} A^{(1)}_{\textcolor{graphpurple}{\alpha}}+\textcolor{graphgreen}{a}A^{(1)}_{\textcolor{graphgreen}{a}}+\textcolor{graphblue}{b}A^{(1)}_{\textcolor{graphblue}{b}}$.

An orthogonal diagonalization of $x_1$ is $x_1=ODO^\top$, where
\[
O=\frac12
\begin{bmatrix}
-\sqrt{2} & 0 & -1 & 1 \\
0 & -\sqrt{2} & 1 & 1 \\
\sqrt{2} & 0 & -1 & 1 \\
0 & \sqrt{2} & 1 & 1
\end{bmatrix},
\qquad
D=
\begin{bmatrix}
\textcolor{graphpurple}{\alpha}-\textcolor{graphblue}{b} & 0 & 0 & 0 \\
0 & \textcolor{graphpurple}{\alpha}-\textcolor{graphblue}{b} & 0 & 0 \\
0 & 0 & -2\textcolor{graphgreen}{a}+\textcolor{graphpurple}{\alpha}+\textcolor{graphblue}{b} & 0 \\
0 & 0 & 0 & 2\textcolor{graphgreen}{a}+\textcolor{graphpurple}{\alpha}+\textcolor{graphblue}{b}
\end{bmatrix}.
\]

Denote the columns of $O$ by $u_1,u_2,u_3,u_4$. For generic values of $\alpha$, $a$ and $b$, the matrix $D$ has three distinct eigenvalues. The corresponding spectral projections are
\begin{equation}\label{eq:random-method-spectral-projections}
\begin{gathered}
c_1^{(1)}=u_1u_1^\top+u_2u_2^\top
=\frac12
\begin{bmatrix}
1&0&-1&0\\
0&1&0&-1\\
-1&0&1&0\\
0&-1&0&1
\end{bmatrix},\\[2mm]
c_2^{(1)}=u_3u_3^\top
=\frac14
\begin{bmatrix}
1&-1&1&-1\\
-1&1&-1&1\\
1&-1&1&-1\\
-1&1&-1&1
\end{bmatrix},
\qquad
c_3^{(1)}=u_4u_4^\top
=\frac14
\begin{bmatrix}
1&1&1&1\\
1&1&1&1\\
1&1&1&1\\
1&1&1&1
\end{bmatrix}.
\end{gathered}
\end{equation}

\noindent\emph{Diagonalization using intersection matrices.}
With respect to the ordered basis
\[
(A^{(1)}_{\textcolor{graphpurple}{\alpha}},
 A^{(1)}_{\textcolor{graphgreen}{a}},
 A^{(1)}_{\textcolor{graphblue}{b}}),
\]
the intersection matrices are
\[
L^{(1)}_{\color{graphpurple}\alpha}=
\begin{array}{r | ccc |}
\multicolumn{1}{r}{} & \color{graphpurple}\alpha & \color{graphgreen}a & \multicolumn{1}{c}{\color{graphblue}b} \\
\cline{2-4}
\color{graphpurple}\alpha & 1&0&0\\
\color{graphgreen}a & 0&1&0\\
\color{graphblue}b & 0&0&1\\
\cline{2-4}
\end{array},
\qquad
L^{(1)}_{\color{graphgreen}a}=
\begin{array}{r | ccc |}
\multicolumn{1}{r}{} & \color{graphpurple}\alpha & \color{graphgreen}a & \multicolumn{1}{c}{\color{graphblue}b} \\
\cline{2-4}
\color{graphpurple}\alpha & 0&2&0\\
\color{graphgreen}a & 1&0&1\\
\color{graphblue}b & 0&2&0\\
\cline{2-4}
\end{array},
\qquad
L^{(1)}_{\color{graphblue}b}=
\begin{array}{r | ccc |}
\multicolumn{1}{r}{} & \color{graphpurple}\alpha & \color{graphgreen}a & \multicolumn{1}{c}{\color{graphblue}b} \\
\cline{2-4}
\color{graphpurple}\alpha & 0&0&1\\
\color{graphgreen}a & 0&1&0\\
\color{graphblue}b & 1&0&0\\
\cline{2-4}
\end{array}.
\]

By definition, $(L_k)_{\ell h}=\eta_{kh}^{\ell}$. Under (M2), if $c(v,w)=\ell$, then $(L_k)_{\ell h}=m_{v\to w}(k,h)$.

Thus, $(L_k)_{\ell h}$ counts the $2$-paths from $v$ to $w$ whose first extended edge has color $k$ and whose second extended edge has color $h$. Conditions (M1) and (M2) ensure that this number depends only on the color $\ell=c(v,w)$, rather than on the chosen extended edge $\{v,w\}$. For example:

\begin{itemize}
  \item The edge $\{1,2\}$ has color $\color{graphgreen}a$. There is exactly one $2$-path from $1$ to $2$ whose colors are first $\color{graphgreen}a$ and then $\color{graphblue}b$, namely, $1-4-2$. Hence
  $(L^{(1)}_{\color{graphgreen}a})_{\color{graphgreen}a,\color{graphblue}b}=1$.

  \item The loop $\{1,1\}$ has color $\color{graphpurple}\alpha$. There are exactly two $2$-paths from $1$ to itself whose two colors are $\color{graphgreen}a$, namely, $1-2-1$ and $1-4-1$. Hence
  $(L^{(1)}_{\color{graphgreen}a})_{\color{graphpurple}\alpha,\color{graphgreen}a}=2$.

\end{itemize}

The matrices $L^{(1)}_{\color{graphpurple}\alpha}$, $L^{(1)}_{\color{graphgreen}a}$, and $L^{(1)}_{\color{graphblue}b}$ are of size $|F_1|\times |F_1|$. They commute, and their simultaneous diagonalization can be computed in $O(|F_1|^3)$ operations. The Random Method can be used here to find a simultaneous diagonalizing matrix $V$. In this example, it gives
\[
V=
\begin{bmatrix}
-1&1&1\\
0&-1&1\\
1&1&1
\end{bmatrix}.
\]

Let $v_1,v_2,v_3$ denote the columns of $V$. Since
$|(v_k)_1|=1=\max_{\ell\in[3]}|(v_k)_\ell|$ for every
$k\in[3]$, Algorithm~\ref{alg:random_method} allows us to choose
$\ell_k=1$ and recover the entries of $P$ directly from the
eigenvector relations:
\[
P_{nk}=\frac{(L_n^{(1)}v_k)_1}{(v_k)_1},
\qquad n,k\in[3],
\qquad
P = \begin{bmatrix}
1&1&1\\
0&-2&2\\
-1&1&1
\end{bmatrix}.
\]

For comparison, the same matrix could have been recovered from the following similarity transformations:
\[
\resizebox{\textwidth}{!}{$\displaystyle
V^{-1}L^{(1)}_{\color{graphpurple}\alpha}V
=
\begin{bmatrix}
1&0&0\\
0&1&0\\
0&0&1
\end{bmatrix},
\qquad
V^{-1}L^{(1)}_{\color{graphgreen}a}V
=
\begin{bmatrix}
0&0&0\\
0&-2&0\\
0&0&2
\end{bmatrix},
\qquad
V^{-1}L^{(1)}_{\color{graphblue}b}V
=
\begin{bmatrix}
-1&0&0\\
0&1&0\\
0&0&1
\end{bmatrix}.
$}
\]

The diagonals of the three displayed matrices reproduce the respective
rows of $P$.
Moreover,
\[
V^\top P
=
\begin{bmatrix}
-2&0&0\\
0&4&0\\
0&0&4
\end{bmatrix}
=4\operatorname{diag}\left(-\frac12,1,1\right),
\]
so \((\gamma_1,\gamma_2,\gamma_3)=(-\frac12,1,1)\). Hence, following Algorithm~\ref{alg:random_method}, we obtain
\[
Q
=\operatorname{diag}\left(-\frac12,1,1\right)^{-1}V^\top
=\begin{bmatrix}
2&0&-2\\
1&-1&1\\
1&1&1
\end{bmatrix}
=4P^{-1}.
\]

Finally, the rows of $Q$ give the matrices
\begin{align*}
c_1^{(1)}&=\frac14\bigl(2A^{(1)}_{\textcolor{graphpurple}{\alpha}}-2A^{(1)}_{\textcolor{graphblue}{b}}\bigr),\\
c_2^{(1)}&=\frac14\bigl(A^{(1)}_{\textcolor{graphpurple}{\alpha}}-A^{(1)}_{\textcolor{graphgreen}{a}}+A^{(1)}_{\textcolor{graphblue}{b}}\bigr),\\
c_3^{(1)}&=\frac14\bigl(A^{(1)}_{\textcolor{graphpurple}{\alpha}}+A^{(1)}_{\textcolor{graphgreen}{a}}+A^{(1)}_{\textcolor{graphblue}{b}}\bigr).
\end{align*}
These agree with the spectral projections in
\eqref{eq:random-method-spectral-projections}.
Their ranks can be computed as
$\mu_{i,j}=\operatorname{rank}(c_j^{(i)})=\tr(c_j^{(i)})$
or read directly from $Q$.
Since the identity matrix
$A^{(1)}_{\textcolor{graphpurple}{\alpha}}=I_4$
is the first element of the color-matrix basis, we have $s_1=1$.
Thus Algorithm~\ref{alg:random_method} gives
\[
\begin{bmatrix}
\mu_{1,1}\\
\mu_{1,2}\\
\mu_{1,3}
\end{bmatrix}
=
\begin{bmatrix}
Q_{11}\\
Q_{21}\\
Q_{31}
\end{bmatrix}
=
\begin{bmatrix}
2\\
1\\
1
\end{bmatrix},
\]
in agreement with the traces of the displayed projections.
For the singleton block, $s_2=1$ and $Q^{(2)}=[1]$,
giving $\mu_{2,1}=1$.

\noindent\emph{The Gram--Cholesky step.}
Let $B_{\textcolor{graphbrown}{c}}$ be the adjacency matrix of the edges between $V_1$ and $V_2$:
\begin{equation}\label{eq:worked-example-cross-block-basis}
B_{\textcolor{graphbrown}{c}}=
\begin{bmatrix}
0&0&0&0&1\\
0&0&0&0&1\\
0&0&0&0&1\\
0&0&0&0&1\\
1&1&1&1&0
\end{bmatrix}.
\end{equation}
Then $L_1(\mathcal Z)=\mathrm{span}\{B_{\textcolor{graphbrown}{c}}\}$ and $L_2(\mathcal Z)=\{0\}$. We identify the projections $c_1^{(1)}$, $c_2^{(1)}$, and $c_3^{(1)}$ with their canonical embeddings in $\mathrm{Sym}(5)$. We obtain
\[
B_{\textcolor{graphbrown}{c}}c_1^{(1)}=B_{\textcolor{graphbrown}{c}}c_2^{(1)}=0.
\]
For the remaining projection, Step~1 of the Gram--Cholesky method uses
the factor
\[
W_{1,3}=\frac12
\begin{bmatrix}1&1&1&1&0\end{bmatrix}^{\top},
\qquad
W_{1,3}^{\top}W_{1,3}=[1],
\qquad
W_{1,3}W_{1,3}^{\top}=c_3^{(1)}.
\]
Its compressed factor is
\[
\widehat U
:=B_{\textcolor{graphbrown}{c}}W_{1,3}
=\begin{bmatrix}0&0&0&0&2\end{bmatrix}^{\top},
\qquad
B_{\textcolor{graphbrown}{c}}c_3^{(1)}
=\widehat U W_{1,3}^{\top}\ne0.
\]
Consequently, $\hh_{1,1}=\hh_{1,2}=\{0\}$, whereas
$\hh_{1,3}=\mathrm{span}\{\widehat U W_{1,3}^{\top}\}$.
Following the preprocessing phase of Algorithm~\ref{alg:gram_cholesky},
the only non-zero Gram matrix is
\[
\widetilde G_{1,3}=G_{1,3}
=[\tr(\widehat U\widehat U^\top)]=[4].
\]
Its Cholesky factor is $L_G=[2]$, and hence the retained determinant is
$g_{1,3}=\det\widetilde G_{1,3}=4$. We therefore obtain
\[
m_{1,1}=m_{1,2}=0,\qquad m_{1,3}=1,\qquad m_{2,1}=0.
\]
The structure constants are therefore
\begin{equation}\label{eq:worked-example-structure-constants}
\begin{aligned}
d_1&=3,& (\mu_{1,1},m_{1,1})&=(2,0),&
(\mu_{1,2},m_{1,2})&=(1,0),&
(\mu_{1,3},m_{1,3})&=(1,1),\\
d_2&=1,& (\mu_{2,1},m_{2,1})&=(1,0).&&&&
\end{aligned}
\end{equation}

\noindent\emph{The determinant factors for a general $A$.}
Write an arbitrary $A\in\mathrm{Sym}^+(5)$ in block form as
\[
A=
\begin{bmatrix}
A_{1:4,\,1:4}&A_{1:4,\,5}\\
A_{5,\,1:4}&A_{5,5}
\end{bmatrix},
\]
where $A_{1:4,\,1:4}\in\mathrm{Sym}(4)$, $A_{5,\,1:4}^\top = A_{1:4,\,5}\in\R^4$, and $A_{5,5}>0$,
and calculate
\[
\begin{aligned}
\lambda_{1,1}(A)
&=\tr\!\left(A\tilde e_1^{(1)}\tilde e_1^{(1)}\right)
=\frac12\tr(A_{1:4,\,1:4}c_1^{(1)})\\
&=\frac14\bigl(A_{11}+A_{22}+A_{33}+A_{44}-2A_{13}-2A_{24}\bigr),\\[1mm]
\lambda_{1,2}(A)
&=\tr\!\left(A\tilde e_2^{(1)}\tilde e_2^{(1)}\right)
=\tr(A_{1:4,\,1:4}c_2^{(1)})\\
&=\frac14\bigl(A_{11}+A_{22}+A_{33}+A_{44}
-2A_{12}+2A_{13}-2A_{14}-2A_{23}+2A_{24}-2A_{34}\bigr),\\[1mm]
\lambda_{1,3}(A)
&=\tr\!\left(A\tilde e_3^{(1)}\tilde e_3^{(1)}\right)
=\tr(A_{1:4,\,1:4}c_3^{(1)})\\
&=\frac14\bigl(A_{11}+A_{22}+A_{33}+A_{44}
+2A_{12}+2A_{13}+2A_{14}+2A_{23}+2A_{24}+2A_{34}\bigr),\\[1mm]
\lambda_{2,1}(A)
&=\tr\!\left(A\tilde e_1^{(2)}\tilde e_1^{(2)}\right)
=A_{5,5}.
\end{aligned}
\]
For the three pairs with $m_{i,\alpha}=0$, the evaluation phase of
Algorithm~\ref{alg:gram_cholesky} gives
\[
\det\psi_{i,\alpha}(A)=1,
\qquad
\frac{\det\phi_{i,\alpha}(A)}{\det\psi_{i,\alpha}(A)}
=\lambda_{i,\alpha}(A),
\quad
(i,\alpha)\in\{(1,1),(1,2),(2,1)\}.
\]
For the only pair with $m_{i,\alpha}>0$, the retained compressed factors
give
\[
\Psi_{1,3}(A)
=[\tr(A\widehat U\widehat U^\top)]
=[4A_{5,5}],
\qquad
V_{1,3}(A)
=[\tr(AW_{1,3}\widehat U^\top)]
=\left[\sum_{j=1}^4 A_{5j}\right].
\]
We identify the $1\times1$ vector $V_{1,3}(A)$ with its single entry.
The Cholesky factorization of $\Psi_{1,3}(A)$ has
$L=[2\sqrt{A_{5,5}}]$, and solving
$Ly=V_{1,3}(A)^\top$ gives
\[
 y=\frac{\sum_{j=1}^4 A_{5j}}{2\sqrt{A_{5,5}}}.
\]
Consequently, Algorithm~\ref{alg:gram_cholesky} returns
\begin{align*}
\det\psi_{1,3}(A)
&=\frac{(2\sqrt{A_{5,5}})^2}{g_{1,3}}
=A_{5,5}=\lambda_{2,1}(A),\\
\frac{\det\phi_{1,3}(A)}{\det\psi_{1,3}(A)}
&=\lambda_{1,3}(A)-y^2
=\lambda_{1,3}(A)
-\frac{V_{1,3}(A)^2}{4\lambda_{2,1}(A)}.
\end{align*}

\noindent\emph{Evaluation of the integral.}
For this space, the convergence condition in Theorem~\ref{thm:integral} reduces to $s>-1/2$. Moreover, the constants in that theorem are
\[
p_{\mathcal Z}=2\log2,
\qquad
q_{\mathcal Z}=-\frac12\log\pi.
\]
Substituting \eqref{eq:worked-example-structure-constants} and the determinant factors above into Theorem~\ref{thm:integral} gives
\begin{equation}\label{eq:worked-example-general-integral}
\begin{aligned}
I(s,A)
&:=\int_{\mathcal Z\cap\mathrm{Sym}^+(5)}
(\det x)^s e^{-\tr(Ax)}\,\dd x\\
&=\sqrt\pi\,2^{-2s}
\Gamma(2s+1)\Gamma(s+1)^2\Gamma\left(s+\frac32\right)\\
&\quad\times\lambda_{1,1}(A)^{-2s-1}\lambda_{1,2}(A)^{-s-1}
\bigl(\lambda_{1,3}(A)\lambda_{2,1}(A)
-\tfrac14 V_{1,3}(A)^2\bigr)^{-s-3/2}.
\end{aligned}
\end{equation}

\noindent\emph{A normalizing-constant ratio for two observations.}
Consider the Bayesian model introduced in
Section~\ref{sec:bayesian-model-selection}, specialized to the colored
graph of the present example. Conditional on $K$, suppose that two
independent random observations
$Z_1,Z_2\stackrel{\mathrm{i.i.d.}}{\sim}\mathrm{N}_5(0,K^{-1})$
have realized values $z_1,z_2$, collected as the rows of the matrix
\[
z=
\begin{bmatrix}
1&1&1&1&1\\
1&-1&1&-1&0
\end{bmatrix}.
\]
These artificial data are chosen to keep the calculations simple.
We are interested in the resulting posterior probability
$\P((G,\mathcal C)=(g,c)\mid Z_1=z_1,Z_2=z_2)$.
Here $n=2$, and the realized value of the random sufficient statistic
$U:=\sum_{i=1}^n Z_iZ_i^\top$ in
\eqref{eq:posterior-model-probability} is
$u:=\sum_{i=1}^n z_i z_i^\top=z^\top z$.
For simplicity, assume that all candidate models have the same prior
probability, so that $p_{g,c}$ is constant. Consequently,
\[
\P((G,\mathcal C)=(g,c)\mid Z_1=z_1,Z_2=z_2)
\propto
\frac{I_{g,c}(\delta+n,D+u)}{I_{g,c}(\delta,D)}.
\]
For $\mathcal Z=\mathcal Z_g^c$, using the $I(s,A)$ convention, we have
\[
I_{g,c}(\delta,D)
=I\left(\frac{\delta-2}{2},\frac D2\right).
\]
We evaluate this ratio using
\eqref{eq:worked-example-general-integral}.
For the standard prior hyperparameters $\delta=3$ and $D=I_5$,
the denominator and numerator correspond, respectively, to
\[
s_{\mathrm{d}}:=\frac{\delta-2}{2}=\frac12,
\qquad
A_{\mathrm{d}}:=\frac D2=\frac12I_5,
\]
and
\[
s_{\mathrm{n}}:=\frac{\delta+n-2}{2}=\frac32,
\qquad
A_{\mathrm{n}}:=\frac{D+u}{2}
=
\begin{bmatrix}
3/2&0&1&0&1/2\\
0&3/2&0&1&1/2\\
1&0&3/2&0&1/2\\
0&1&0&3/2&1/2\\
1/2&1/2&1/2&1/2&1
\end{bmatrix}.
\]
The determinant factors calculated above give
\[
\begin{aligned}
&\bigl(\lambda_{1,1}(A_{\mathrm{d}}),
\lambda_{1,2}(A_{\mathrm{d}}),
\lambda_{1,3}(A_{\mathrm{d}})\lambda_{2,1}(A_{\mathrm{d}})
-\tfrac14V_{1,3}(A_{\mathrm{d}})^2\bigr)
=\left(\tfrac12,\tfrac12,\tfrac14\right),\\
&\bigl(\lambda_{1,1}(A_{\mathrm{n}}),
\lambda_{1,2}(A_{\mathrm{n}}),
\lambda_{1,3}(A_{\mathrm{n}})\lambda_{2,1}(A_{\mathrm{n}})
-\tfrac14V_{1,3}(A_{\mathrm{n}})^2\bigr)
=\left(\tfrac12,\tfrac52,\tfrac32\right).
\end{aligned}
\]
Consequently, \eqref{eq:worked-example-general-integral} yields
\[
I(s_{\mathrm{d}},A_{\mathrm{d}})
=16\sqrt2\,\pi^{3/2},
\qquad
I(s_{\mathrm{n}},A_{\mathrm{n}})
=4\pi^{3/2}\left(\frac25\right)^{5/2},
\]
and hence the ratio in \eqref{eq:posterior-model-probability} is
\begin{equation}\label{eq:worked-example-exact-ratio}
\frac{I_{g,c}(\delta+n,D+u)}{I_{g,c}(\delta,D)}
=\frac{I(s_{\mathrm{n}},A_{\mathrm{n}})}
{I(s_{\mathrm{d}},A_{\mathrm{d}})}
=5^{-5/2}=\frac{1}{25\sqrt5}.
\end{equation}
Its logarithm is
\begin{equation}\label{eq:worked-example-exact-log-ratio}
\log I(s_{\mathrm{n}},A_{\mathrm{n}})
-\log I(s_{\mathrm{d}},A_{\mathrm{d}})
=-\frac52\log5
\approx -4.023595.
\end{equation}

\noindent\emph{Numerical check with \texttt{gips}.}
The coloring in this example is the orbit coloring generated by the
cyclic group
\[
\Gamma=\langle(1\,2\,3\,4)(5)\rangle.
\]
This special case is covered by the \texttt{gips} package \citep{gips}.
For a known zero mean and $S=z^\top z/n$,
\texttt{log\_posteriori\_of\_gips()} evaluates the log-ratio in
\eqref{eq:worked-example-exact-log-ratio}.
The following R code compares this value with our closed-form result.

\noindent\begin{minipage}{\linewidth}
\begin{lstlisting}
# install.packages("gips")
library(gips)

z <- rbind(c(1, 1, 1, 1, 1), c(1, -1, 1, -1, 0))
n <- nrow(z)
S <- crossprod(z) / n

g <- gips(S, n, delta = 3, D_matrix = diag(5),
          was_mean_estimated = FALSE, perm = "(1,2,3,4)(5)")

c(
  gips_value = log_posteriori_of_gips(g),
  our_closed_form = log(5^(-5 / 2)),
  difference = log_posteriori_of_gips(g) - log(5^(-5 / 2))
)
#>      gips_value our_closed_form      difference
#>   -4.023595e+00   -4.023595e+00   -8.881784e-16
\end{lstlisting}
The two values agree to numerical precision.
\end{minipage}

\end{example}

\section*{Use of generative AI}

During revision, we used OpenAI's Codex to assist with reviewing
the manuscript. This helped identify several errors and gaps in
\href{https://arxiv.org/abs/2601.16945v1}{version~1 of the preprint}.
The tool also assisted with checking selected calculations and
improving language and presentation. The authors evaluated the
suggestions, revised and verified the mathematical arguments, and
take full responsibility for the final manuscript.

\section*{Acknowledgments}
We thank the anonymous referees for their constructive
comments and suggestions, which helped improve the manuscript.

\section*{Funding}
This research was funded in part by National Science Centre, Poland, UMO-2022/45/B/ST1\\/00545, KAKENHI 24KK0059, MEXT Promotion of Distinctive Joint Research Center Program JPMXP0723833165 and Osaka Metropolitan University Strategic Research Promotion Project (Development of International Research Hubs). This research  was also supported by the International Emerging Action (IEA) "GRAPHICAL MODELS AND APPLICATIONS" (2024-2025) of the CNRS 
and by the France 2030 Program - Centre Henri Lebesgue (ANR-11-LABX-0020-01).


\bibliographystyle{plainnat}
\bibliography{Bibl}

\end{document}